\documentclass[10pt,reqno]{amsart}

\usepackage[a4paper]{geometry}%,text={16cm,23cm}]{geometry}
\usepackage{amsmath,amsthm,amssymb}
\usepackage{mathrsfs}
\usepackage{enumerate,enumitem}
\usepackage{xfrac,faktor}
\usepackage{xcolor}
\usepackage{hyperref}
\hypersetup{colorlinks=true,breaklinks=true,urlcolor=blue,linkcolor=blue,bookmarksopen=true} 

\newtheorem{theo}{Theorem}[section]
\newtheorem{prop}[theo]{Proposition}
\newtheorem{coro}[theo]{Corollary}
\newtheorem{lem}[theo]{Lemma}

\definecolor{darkred}{rgb}{0.9,0.1,0.1}

\theoremstyle{definition}

\newtheorem{Rq}[theo]{Remark}

\newtheorem*{hypo1}{Assumption A}
\newtheorem*{hypoB}{Assumption B}

\numberwithin{equation}{section}

\newcommand*{\X}{\mathcal{X}}
\newcommand*{\M}{\mathcal{M}}
\newcommand*{\A}{\mathscr A}
\renewcommand*{\L}{\mathcal L}
\newcommand*{\LL}{\mathscr L}
\renewcommand{\P}{\mathcal P}
\newcommand{\1}{\mathbf 1}
\newcommand*{\N}{\mathbb{N}}
\newcommand*{\B}{\mathcal{B}}

\newcommand*{\R}{\mathbb{R}}
\newcommand*{\E}{\mathbb{E}}
\newcommand*{\RR}{\mathfrak p}
\newcommand*{\ro}{R}
\newcommand*{\autro}{\rho}
\newcommand*{\nn}{\mathfrak{n}}
\newcommand*{\kk}{\mathfrak{m}}
\newcommand*{\e}{\mathrm{e}}
\newcommand*{\dis}{\displaystyle}
\newcommand*{\C}{\theta}
\newcommand*{\CC}{\Theta}
\newcommand\constante[1]{[#1]}
\newcommand{\boulet}{\Xi}

\DeclareMathOperator{\supp}{supp}

\newcommand{\pg}[1]{{\color{black} #1}} 
 
\begin{document}

\title{A non-conservative Harris ergodic theorem}

\author{
	Vincent \textsc{Bansaye}
		%\thanks{First footnote}
		\and
	Bertrand \textsc{Cloez}
		%\thanks{Second footnote}
		\and
	Pierre \textsc{Gabriel}
	\and
	Aline  \textsc{Marguet}
}
\def\runauthor{
	Vincent \textsc{Bansaye}, Bertrand \textsc{Cloez}, Pierre \textsc{Gabriel}, Aline \textsc{Marguet}
}

\date{  \today}
    \address[V. \textsc{Bansaye}]{CMAP, \'Ecole Polytechnique, Route de Saclay, 91128 Palaiseau cedex, France.}
    \email{vincent.bansaye@polytechnique.edu}
    \address[B. \textsc{Cloez}]{MISTEA, INRA, Montpellier SupAgro, Univ. Montpellier, 2 place Pierre Viala, 34060 Montpellier, France.}
    \email{bertrand.cloez@inra.fr}
    \address[P. \textsc{Gabriel}]{Universit\'e Paris-Saclay, UVSQ, CNRS, Laboratoire de Math\'ematiques de Versailles, 78000, Versailles, France.}
    \email{pierre.gabriel@uvsq.fr}
    \address[A. \textsc{Marguet}]{Univ. Grenoble Alpes, INRIA, 38000 Grenoble, France}
    \email{aline.marguet@inria.fr}

\begin{abstract}
We consider non-conservative positive semigroups and obtain necessary and sufficient conditions for uniform exponential contraction in weighted total variation norm.
This  ensures the existence of Perron eigenelements and provides quantitative estimates of the spectral gap, complementing Krein-Rutman theorems and generalizing  probabilistic approaches.
The proof is based on a non-homogenous $h$-transform of the semigroup and the construction of Lyapunov functions for this latter.
It exploits then the classical necessary and sufficient conditions of Harris's theorem for  conservative semigroups and recent techniques developed for the study of absorbed Markov processes.
We apply these results to population dynamics. We  obtain exponential convergence of birth and death processes conditioned on survival to their quasi-stationary distribution, as well as  estimates on exponential relaxation to stationary profiles in growth-fragmentation PDEs.

\tableofcontents
\end{abstract}



\maketitle

\section{Introduction}

Iteration of a positive linear operator is a fundamental issue in operator analysis, linear Partial Differential Equations (PDEs), probability theory, optimization and control. In continuous time, the associated structure is a  positive semigroup. A  positive semigroup is a family of linear operators $(M_t)_{t\in \R_+}$ acting on a space of measurable functions $f$ on $\X$. It satisfies $M_0f=f$,
$M_{t+s}f=M_t(M_s(f))$  for any $t,s\geq 0$, and $M_tf\geq 0$ for any $f\geq 0$ and $t\geq 0$.
In finite dimension, {\it i.e.} when $ \X$ is a finite set, under a strong positivity assumption, the Perron-Frobenius theorem~\cite{Frobenius,Perron} ensures the existence of unique right and left positive eigenvectors $h=(h(x): x\in \X)$ and $\gamma=(\gamma(x): x\in \X)$ associated to the maximal eigenvalue  $\lambda \in \R$ :
$M_th=\e^{\lambda t}h$ and $\gamma M_t=\e^{\lambda t}\gamma$. Moreover, the asymptotic profile when $t\to+\infty$
 is given  by
$$M_tf= \e^{\lambda t} \langle\gamma,f\rangle h  +\mathcal O(\e^{(\lambda-\omega) t}),$$
for any vector $f$, any $\omega>0$ strictly smaller than the spectral gap and
$\langle\gamma,f\rangle=\sum_{x\in \X}\gamma(x)f(x)$.

The generalization in infinite dimension has attracted lots of attention. It  is motivated in particular by the asymptotic analysis of linear PDEs counting the density of  particles (or individuals) $u_t^x(y)$ in location $y$ at time $t$ when initially the particles are located in $x$, {\it i.e.} $u_0^x=\delta_x$. The semigroup is then given by $M_tf(x)=\int_{\X} u_t^x(y)f(y)dy$. Semigroups play also a key role in the study of the convergence in law of Markov processes via $M_tf(x)=\E_x[f(X_t)]$ or the study of quasi-stationary regimes via 
$M_tf(x)=\E_x[f(X_t)\1_{X_t\not \in \mathcal D}]$, where $\mathcal D$ is an absorbing domain.
In spectral theory, the Krein-Rutman theorem~\cite{Krein-Rutman} yields an extension to positive compact operators for the existence of eigenelements.
Generalizations have been then obtained~\cite{DautrayLions3,Harris_book,Nummelin}. The most recent providing asymptotic profiles with exponential convergence associated to a spectral gap, see in particular \cite{MS16} and the references therein.  When positive eigenelements exist, an alternative tool to prove convergence is the dissipation of entropy, especially the general relative entropy introduced in~\cite{MMP2}.\\

Our aim is to relax and simplify some assumptions involved in the spectral approach and provide more quantitative results. We  apply them to  two classical models: growth fragmentation PDEs and  birth and death processes. \\
 For that purpose, we are inspired and use techniques  developed in probability theory. When $M_t{\bf 1}={\bf 1}$ for any $t\geq 0$, the semigroup is said to be \emph{conservative}. Indeed, the total number of particles is preserved along time for the corresponding linear PDE or for the first moment semigroup of the branching Markov process.  It holds when particles move without reproduction or death.
Conservative semigroups arise classically in the study of the law of  Markov 
processes $(X_t)_{t\geq0}$  via $M_tf(x)=\E_x[f(X_t)]$. In that case, the ergodic theory  and stability of Markov processes can be invoked to analyze the long time behavior of the semigroup, in the spirit of the pioneering works of Doeblin~\cite{D40} and Harris~\cite{harris1956} and coupling techniques. We refer to \cite{MT} for an overview.  These results allow proving the existence of a stationary probability measure and the  uniform exponential convergence for the weighted total variation distance towards this latter.  They rely on the two following conditions, corresponding respectively to Lyapunov function and small set property, for  some  time $\tau>0$ and some subset $K$ of $\X$. First,   there exists  $V:\X\to[1,+\infty)$ such that
\begin{equation}\label{eq:Lyap_cons}M_{\tau}V\leq {\mathfrak a}V+{\mathfrak b}\1_K\end{equation}
for some constants ${\mathfrak a} \in(0,1)$ and  ${\mathfrak b}\geq0$.
% where $\1_K$ stands for the indicator function of the set $K$
Furthermore, there exists   a probability measure $\nu$ on $\X$ such that for any non-negative $f$ and any $x\in K$,
\begin{equation}\label{eq:Doeb_cons}M_{\tau}f(x)\geq {\mathfrak c}\,\nu(f)\end{equation}
with ${\mathfrak c}>0$.
Under aperiodicity condition, these assumptions ensure that $M_{\tau}$ admits a unique invariant probability measure $\gamma$ and satisfies a  contraction principle for the weighted total variation distance, see \cite{KM12,MT,MTIII}.
In our study, we will assume that the probability measure $\nu$ is supported by the small set $K$.
\pg{In this case, Assumptions~\eqref{eq:Lyap_cons} and~\eqref{eq:Doeb_cons} together, with the same set $K$, guarantee aperiodicity, see the comments after Theorem~\ref{th:ergo} in Appendix~\ref{sect:conservatif}.} 
Under mild conditions of local boundedness, the contraction property captures the continuous time by iteration. Then there exist explicit constants ${C,\omega>0}$ such that for all  $x\in \mathcal X$, $f$ such that $ \vert  f\vert \leq V$, and all $t\geq0$,
$$\big\vert  M_tf(x)-\gamma(f) \big\vert \leq CV(x)\e^{-\omega t},$$
where $\gamma(f)=\int_\X f\,d\gamma$.
It provides   a necessary  and sufficient condition and  quantitative estimates \cite[Chapter~15]{DMPS} and is referred to as $V$-uniform ergodicity~\cite{MT}.\\

{\bf Main results and ideas of the proof.}
In the present work we provide a counterpart to $V$-uniform ergodicity in the non-conservative framework.
We obtain necessary and sufficient conditions for exponential convergence.  We also estimate  the constants involved, in particular the amplitude of the spectral gap, using
only the parameters in the assumptions. The approach relies on the reduction of the problem to a non-homogeneous but conservative one. \\ 
In this vein, the Doob $h$-transform is a historical and inspiring technique  and we mention \cite{DaviesSimon,D57,doob2012classical,KimSong,Nummelin,RW00,Schaefer}. In the case when we already know some positive  eigenfunction $h$ associated with the eigenvalue $\lambda$,  a conservative semigroup is indeed defined for $t\geq 0$ by
$$P_tf=\frac{ M_t(h f)}{\e^{\lambda t}h}.$$
The ergodic analysis of the conservative semigroup $P$ using the theory mentioned above  yields the expected estimates on the original semigroup $M_tf=\e^{\lambda t}\, h\, P_t(f/h)$. \\

Our work focuses on the case when we do not know {\it a priori} the existence of positive eigenelements or need more information on it to apply Harris's theory. 
We simultaneously prove their existence and estimate them. The novelty of the 
approach mainly lies in the form of the $h$-transform and the general and explicit construction of  Lyapunov functions for this latter. It also strongly  exploits recent progresses on the convergence to quasistationnary distribution. More precisely, we introduce the  transform for $s\leq u\leq t$: 
$$P^{(t)}_{s,u}f=\frac{M_{u-s}(fM_{t-u}\psi)}{M_{t-s}\psi}.$$
For a fixed time $t$ and positive function~$\psi$, this family is a conservative propagator (or semiflow). To mimic the stabilizing property of the eigenfunction $h$ in the $h$-transform, we  assume that $\psi$ satisfies
\begin{equation}
\label{eq:lyap-psi}
M_{\tau} \psi \geq \beta \psi, 
\end{equation}
for some $\beta>0$ and $\tau>0$, similarly as in~\cite{CV18,CMMS13, MS16}.
Section~\ref{sec:applications} provides examples where $\psi$ can be obtained, sometimes explicitly, while we do not know a priori eigenelements. 
We construct a family of Lyapunov functions $(V_k)_{k\geq 0}$ as follows
$$
V_k= \nu\left(\frac{M_{k\tau}\psi}{\psi} \right)\frac{ V}{ M_{k\tau}\psi},
$$
where $\nu$ is a probability measure.  Assuming  that  $V$ satisfies
\begin{equation}
\label{eq:lyap-non-cons-intro}
M_{\tau}  V \leq \alpha   V + \C  \mathbf{1}_K \psi,
\end{equation}
these functions allow us to prove that $P^{(t)}$ satisfies  Lyapunov type condition \eqref{eq:Lyap_cons}. To  prove  that $P^{(t)}$  also satisfies the  small set condition \eqref{eq:Doeb_cons} and  invoke the contraction of the conservative framework, we  follow techniques due to Champagnat and Villemonais \cite{CV14} for the study for the convergence of processes conditioned on non-absorption.
This yields the following generalized small set condition. We assume that there exist $c,d>0$ such that for  all $x\in K$, 
\begin{equation}
\label{eq:doeb-noncons-intro}
M_{\tau }(f\psi)(x)\geq c\, \nu(f)\, M_{\tau}\psi(x)
\end{equation}
for any $f$ positive and 
\begin{equation}\label{eq:control-mass}
 M_{n\tau}\psi(x)\leq \frac{1}{d}\,  \nu\left(\frac{ M_{n\tau}\psi}{\psi}\right) \psi(x)
\end{equation}
for any integer $n$.
The two latter conditions involve uniform but local bounds on the space $\X$.
The first inequality is restricted to a fixed time, unlike the second one.  Convenient sufficient condition can be given to restrict also the second one to a finite time estimate, in particular via coupling, see the next paragraph and applications. \\
Assuming that $\mathfrak a= {\alpha}/{\beta}<1$ and $V\geq\psi$,  we obtain a contraction result  for $P^{(t)}$  in forthcoming Proposition~\ref{prop:ergo-auxi}. By  estimating the mass $M_t\psi$, we then prove that a triplet $(\lambda, h,\gamma)$ of eigenelements exists and provide estimates on these elements. We get that for all $x\in \mathcal X$  and  $f$  such that $\vert f\vert\leq V$,
$$ \big\vert \e^{-\lambda t} M_tf(x)-h(x)\gamma(f) \big\vert \leq CV(x)\e^{-\omega t},$$
with constants $\omega>0$ and $C>0$, that are explicit.
We also prove that the conditions given above are actually necessary for this uniform convergence in weighted total variation distance.\\

{\bf Tractable conditions and applications.}
We apply our results to describe the profile of populations in growth-fragmentation PDEs and birth and death processes conditioned on non-absorption. We also refer to the follow-up paper~\cite{CloezGabriel}
for an application to mutation-selection PDEs. \\
%We both relax some of the  assumptions needed for these models in the literature and obtain more quantitative results.\\
Sufficient conditions are helpful for such issues. In particular, one can check  Lyapunov conditions~\eqref{eq:lyap-psi} and~\eqref{eq:lyap-non-cons-intro} using the generator.  Loosely speaking, writing $\LL$  the generator of $(M_t)_{t\geq0}$,  we check  that it is enough to find $V,\psi$  and  $a<b$ such that
\begin{equation}
\label{eq:Lyap-intro}
\LL V \leq aV +\zeta \psi \qquad\text{and}\qquad  \LL \psi\geq b \psi.
\end{equation}
The second part of the assumptions deals with local  bounds on the semigroup.
It is worth noticing that in the case  $0<\inf_{K} \psi\leq \sup_K \psi <\infty$,
%when $\psi$ is bounded from above and below by positive contants on $K$, 
Assumption~\eqref{eq:doeb-noncons-intro} is equivalent to the  small set condition~\eqref{eq:Doeb_cons}.
For a discrete state space $\X$,  the conditions~\eqref{eq:doeb-noncons-intro} and~\eqref{eq:control-mass} can be reduced  to irreducibility properties. In the continuous setting, a   coupling condition implying~\eqref{eq:control-mass} is proposed and applied in \cite{CloezGabriel}:
\[\frac{M_{\tau} f(x)}{\psi(x)}\geq c_0\int_0^{\tau}\frac{M_{\tau-s}f(y)}{\psi(y)}\sigma_{x,y}(ds)\]
for all $x,y\in K$, $f\geq0$ and some constants $c_0,\tau>0$ and probability measures $(\sigma_{x,y})_{x,y\in K}$ on $[0,\tau]$.
This condition is stronger than Assumption~\eqref{eq:control-mass}, but  it is easier to check for several  population models.\\

Let us briefly  illustrate the interest of the conditions above and  the novelty  of the result for applications on the so-called growth-fragmentation equation. This equation is central in the modeling of various physical or biological phenomena~\cite{Banasiak,BLL19,DGR,MD86,BP,RTK}.
In its classical form it can be written as
\[\partial_{t}u_t(x)+\partial_x u_t(x) + B(x)u_t(x) = \int_0^1B\Bigl(\frac xz\Bigr)u_t\Bigl(\frac xz\Bigr)\frac{\wp(d z)}{z},\]
where $t,x>0$, and is complemented with a zero flux boundary condition $u_t(0)=0$.
It drives the time evolution of the population density $u_t$ of particles characterized by a structural size $x\geq 0$.
Each particle with size $x$ grows with speed one and splits with rate $B(x)$ to produce smaller particles of sizes $zx$, with $0<z<1$ distributed with respect to the fragmentation kernel~$\wp$.
The corresponding dual equation reads
\begin{equation*}\label{eq:GF_dual}
\partial_{t}\varphi_t(x)=\LL\varphi_t(x), \quad \text{where }  \ \ 
\LL f(x) = f'(x) + B(x) \left( \int_0^1 f(zx) \wp(dz) - f(x) \right).
%$$=\partial_x \varphi_t(x) + B(x) \left( \int_0^1 \varphi_t(zx) \wp(dz) - \varphi_t(x) \right).
\end{equation*}
It generates a semigroup $(M_t)_{t\geq0}$ on the state space $\X=[0,\infty)$ as follows: $M_tf=\varphi_t$ is the solution to the dual equation with initial condition $\varphi_0=f$.
The duality property ensures the relation $\int_\X u_t(x)f(x)dx=\int_\X  u_0(x)M_tf(x)dx$ for all time $t\geq0$.
A direct  computation ensures that the infinitesimal generator $\LL$ verifies~\eqref{eq:Lyap-intro} with
$$
V:x\mapsto 1+x^k,\ k>1 \qquad \text{and}\qquad\psi:x\mapsto 1+x.
$$
%for instance for the celebrated example of the equal mitosis kernel $\wp=2\delta_{\frac12}$. 
Consequently the semigroup $(M_t)_{t\geq0}$ satisfies~\eqref{eq:lyap-psi}-\eqref{eq:lyap-non-cons-intro}.
Using the reachability of any sufficiently large size and a monotonicity argument on $B$, we can show~\eqref{eq:doeb-noncons-intro}-\eqref{eq:control-mass}.
It gives the existence of eigenelements and the exponential convergence. %, see Theorem~\ref{th:GF}.
It improves the existing results we know, where a polynomial bound is assumed for~$B$. It additionally provides  estimates on the spectral gap.
We refer the reader to Section~\ref{EDP} for details and more references.  \\

{\bf State of the art and related works}. Our results and method
are  linked to the study of  geometric ergodicity of Feynman-Kac type semigroups
$$M_tf(x) = \mathbb{E}_x\left[ f(X_t) \e^{\int_0^t F(X_s) ds} \right].$$
The study of these non-conservative semigroups has been developed in a general setting \cite{KM03,KM05}. We refer also to
 \cite{DM13,DM04} in discrete time.
These semigroups appear in particular for the study of branching processes \cite{BW18,C17,M16} and large deviations \cite{ferre-these,KM03,KM05}.
%We  relax here some assumptions for convergence and provide more quantitative estimates. Indeed
In~\cite{KM03} the function $F$ is supposed to be bounded, with a norm smaller than an explicit constant.
The unbounded case is addressed in~\cite{KM05} where some weighted norm of $F$ is required to be small, together with either a ``density assumption'' %requiring some regularizing effect, 
or some irreducibility and aperiodicity conditions. 
The growth-fragmentation model would correspond to $F=B$ and we  relax
here the assumptions on $B$.  Nevertheless, we observe that the criteria in~\cite{KM03,KM05} may be simpler to apply, and under the density assumption the authors prove a stronger result of discrete spectrum.  \\
In the same vein, let us mention~\cite{FRS}, which uses the Krein-Rutman theorem for the existence of eigenelements  and an $h$-transform.
It requires additional regularity assumptions to derive the geometric ergodicity of general Feynman-Kac semigroups, as well as a stronger small set condition.
The regularity requirement is a local strong Feller assumption. It is well-suited for diffusive  equations, see also~\cite{HuisingaMeynSchutte} for a result on non-conservative hypoelliptic diffusions. It does not seem to be  adapted to our motivations with less regularizing effect.\\
In~\cite{Bertoin18,BW18}, a Feynman-Kac approach is developed for  growth-fragmentation equation with exponential individual growth.
It relies on the application of the Krein-Rutman theorem on a well-suited operator.
It  provides sharp conditions on the growth rate but it requires the division rate $B$ to be bounded.
%Our approach is thus different both for proving the existence of eigenelements and the convergence to the associated profile. 
We expect that our results   provide quantitative estimates and  relax the boundedness assumption on $B$ in \cite{Bertoin18,BW18} for some  relevant classes of growth rates.
\\
Finally, Hilbert metric and 
Birkhoff  contraction  yield another  powerful method for analysis of semigroup, which has been well developed \cite{B57,Nussbaum, Seneta}. As far as we see, it requires uniform bounds on the whole space for the semigroup at fixed time, which are relaxed by the backward approach exploited here.

More generally, our work is motivated by the study of structured populations and  varying environment. The conditions are well adapted to complex trait spaces for populations, including discrete  and continuous components: space, age, phenotype, genotype, size, etc., see  Section~\ref{sec:applications}  and {\it e.g.} \cite{BCG17, BDMT, CloezGabriel, marguet17}. The contraction with explicit bounds allows relevant compositions appearing  when parameters of the population dynamic vary along  time. Such a method has already been exploited in \cite{BCG17} for the analysis of PDEs in varying environment in the case of uniform exponential convergence. Several models  need the more general framework considered in this paper. Indeed, the typical  trait  does not come back in compact sets in a bounded time for models like growth-fragmentation and mutation-selection. \\

{\bf Outline of the paper.}
\pg{The main result of the paper, Theorem~\ref{V-contraction}, is stated in Section~\ref{mainetpastweedie}, and its proof is given at the end of Section~\ref{sec:proofs}.
Useful sufficient conditions are proposed in Section~\ref{ssec:generator}.}
In Section~\ref{sect:auxiliaire}, we consider the conservative embedded propagator and we establish its contraction property.
We present in Section~\ref{sec:eigenelements} our results on the existence  of the eigenelements and our quantitative estimates.
Section~\ref{sec:proofs} contains the proofs of these statements.
Section~\ref{sec:applications}  is devoted to applications:  the random walk on integers absorbed at $0$ and the growth-fragmentation equation. We finally briefly mention  some additional results and perspectives, including the discrete time framework and the reducible case.

\section{Definitions and main result}\label{labelrouge}
\label{mainetpastweedie}
We start by stating precisely our framework.
Let $\X$ be a measurable space.
For any measurable function $\varphi:\X\to(0,\infty)$ we denote by $\B(\varphi)$ the space of measurable functions $f:\X\to\R$ which are dominated by $\varphi,$
{\it i.e.} such that the quantity
\[\left\|f\right\|_{\B(\varphi)}=\sup_{x\in\X}\frac{|f(x)|}{\varphi(x)}\]
is finite.
Endowed with this weighted supremum norm, $\B(\varphi)$ is a Banach space.
Let $\mathcal{B}_+(\varphi) \subset \B(\varphi)$ be its positive cone, namely the subset of nonnegative functions.

\

Let $\M_+(\varphi)$ be the cone of positive measures on $\X$ which integrate $\varphi,$
{\it i.e.} the set of positive measures $\mu$ on $\X$ such that the quantity
$\mu(\varphi)=\int_\X\varphi\,d\mu$
is finite.
We denote by $\M(\varphi)=\M_+(\varphi)-\M_+(\varphi)$ the set of the differences of measures that belong to $\M_+(\varphi)$.
When $\inf_\X\varphi>0$, this set is the subset of signed measures which integrate $\varphi$.
When $\inf_\X\varphi=0$ it is not a subset of signed measures, but it can be rigorously built as the quotient space
\[\M(\varphi)=\faktor{\M_+(\varphi)\times\M_+(\varphi)}{\sim}\]
where $(\mu_1,\mu_2)\sim(\tilde\mu_1,\tilde\mu_2)$ if $\mu_1+\tilde\mu_2=\mu_2+\tilde\mu_1$,
and we denote $\mu=\mu_1-\mu_2$ to mean that $\mu$ is the equivalence class of $(\mu_1,\mu_2)$.
The Hahn-Jordan decomposition theorem ensures that for any $\mu\in\M(\varphi)$ there exists a unique couple $(\mu_+,\mu_-)\in\M_+(\varphi)\times\M_+(\varphi)$ of mutually singular measures such that $\mu=\mu_+-\mu_-$.
An element $\mu$ of $\M(\varphi)$ acts on $\B(\varphi)$ through $\mu(f)=\mu_+(f)-\mu_-(f)$ and $\M(\varphi)$ is endowed with the weighted total variation norm
\[\left\|\mu\right\|_{\M(\varphi)}=\mu_+(\varphi)+\mu_-(\varphi)=\sup_{\|f\|_{\B(\varphi)}\leq1}|\mu(f)|,\]
which makes it a Banach space.
\pg{Note that allowing $\inf\varphi$ to be zero is important for applications. For instance, for the growth-fragmentation equation with a drift term $\partial_x(xu_t(x))$, the natural weight function for exponential convergence is $V(x)=x^{k_1}+x^{k_2}$ with $1/2<k_1<1<k_2$, see~\cite{CCM10}.
Another example is given by the heat equation on a bounded convex subset of $\R^n$ with Dirichlet boundary conditions, for which the square root of the distance to the boundary is a relevant Lyapunov function.}

\

We are interested in semigroups $(M_t)_{t\geq0}$ of kernel operators, {\it i.e.} semigroups of linear operators $M_t$ which act both on $\B(\varphi)$ (on the right $f\mapsto M_tf$) and on $\M(\varphi)$ (on the left $\mu\mapsto\mu M_t$), are positive in the sense that they leave invariant the positive cones of both spaces, and enjoy the duality relation $(\mu M_t)(f)=\mu(M_tf)$.
We work in the continuous time setting, $t\in \R^+$,
and we refer to Section~\ref{ssec:discrete} for comments on the discrete time case. Finally, we use the shorthand notation  $f\lesssim g$ when, for two functions $f,g : \Omega \rightarrow \R$, there exists a constant $C>0$ such that $f\leq C g$ on $\Omega.$

\subsection{Assumptions and criterion for exponential convergence}

We consider two measurable functions $ V:\X\to(0,\infty)$ and $\psi:\X\to(0,\infty)$, with $\psi\leq V$, and a positive semigroup $(M_t)_{t\geq0}$ of kernel operators acting on $\B( V)$ and $\M( V)$.
Let us state   Assumption~{\bf A}, which  gathers  conditions given in Introduction. %and allows to derive the general results in discrete time.

\begin{hypo1} \label{assu:condition_generales}
There exist $\tau>0$, $\beta >\alpha>0$, $\C> 0$, $(c,d)\in(0,1]^2$, $K \subset \X$ and $\nu$ a probability measure on $\X$ supported by $K$ such that $\sup_K V/\psi<\infty$ and

\vskip3mm

\begin{enumerate}[label=(A\arabic*), parsep=3mm]
\item \label{A1} $ M_{\tau}  V \leq \alpha   V + \C  \mathbf{1}_K \psi$,
\item \label{A2} $M_{\tau} \psi \geq \beta \psi$,
\item \label{A3} 
$\dis\inf_{x\in K} \frac{M_{\tau }(f\psi)(x)}{M_{\tau}\psi(x)}  \geq c\, \nu(f)\quad$ for all $f\in\B_+(V/\psi)$, 
\item \label{A4}  $\dis \nu\left(\frac{ M_{n\tau}\psi}{\psi}\right)\geq d \sup_{x\in K}  \frac{M_{n\tau}\psi(x)}{\psi(x)}\quad$ for all positive integers $n$.
\end{enumerate}
\end{hypo1}

Assumption {\bf A} 
%yields an extension of the criteria for  exponential convergence of conservative semigroups for the weighted total variation distance~\cite{HM11,harris1956,KM12,MT}.It
 is  linked and relax classical assumptions for the ergodicity of non-conservative semigroups  \cite{Nagel86,BCG17,CV18,FRS,HuisingaMeynSchutte,KM03,KM05,MS16,V18,V19,Webb87}.\\
%It relaxes  such assumptions and provides necessary conditions for exponential convergence in total variation distance. \\
More precisely,  we can compare {\bf A} to the assumptions  (1-4) of~\cite[Theorem 5.3]{MS16}, where (2) corresponds to~\ref{A2}
\pg{and the strong maximum principle in (4), which is not a necessary condition for ergodicity, shares links with~\ref{A3}.
Assumption~(1) in~\cite[Theorem 5.3]{MS16} contains two conditions: a dissipativity condition which is closely related to~\ref{A1}, see~\cite{Wu}, and a compactness condition which is replaced here by~\ref{A4}.
Assumption~\ref{A4} essentially means that, uniformly in $x\in K$, the asymptotic growth of the $\psi$-mass starting from $x$ is dominated by the growth of the same $\psi$-mass starting from the measure $\nu$.
It is inspired from conditions that appear for the stability of Feynman-Kac semigroups, see Condition ($\mathcal Z$) in~\cite[Section~3]{MM02}.} \\
Assumption~{\bf A} provides  an extension of the conditions of~\cite{BCG17} in the homogeneous case. In that latter, the small set  condition~\ref{A3} was required on the whole space $\X$, which imposes uniform ergodicity.  Uniformity does not hold  in the two applications we consider in the present paper.\\
Our assumptions are  similar and inspired from  \cite{CV18}, which is dedicated to convergence to quasi-stationary distribution. Our techniques differ, especially in the form of the embedded propagator and the construction of Lyapunov functions, see Section~\ref{sect:auxiliaire}. We relax the boundedness of $\psi$  required in \cite{CV18}  and  obtain necessary conditions for weighted exponential convergence. As far as we see, this approach   also provides  more quantitative estimates, see   forthcoming Section~\ref{sec:quantitative} and the   proofs. \\
%It enables  us to  treat the two examples developed  in Section \ref{sec:applications}. \\

The main result of the paper can be stated as follows.  It is proved in Section~\ref{proofthun}, where   constants  involved are explicit.
To state the result in the continuous time setting, we need
 that
$M_tV(x)/V(x)$ is bounded  on compact time intervals, uniformly on $\X$.
We
refer  to~\cite[Section 20.3]{MT} for details on this classical assumption.
\begin{theo}\label{V-contraction}
\pg{Let $V:\X\to(0,\infty)$ measurable and let $(M_t)_{t\geq0}$ be a positive semigroup of kernel operators on $\B(V)$ and $\M(V)$ such} that
  $t\mapsto\|M_tV\|_{\B(V)}$ %= \sup_{x\in\X} M_tV(x)/V(x)$
   is locally bounded on $[0,\infty)$. \\
i) \pg{Let $\psi:\X\to(0,\infty)$ be a measurable function such that  Assumption~{\bf A} is satisfied.} Then, there exists a unique triplet $(\gamma,h,\lambda)\in\M_+( V)\times \B_+( V)\times \mathbb{R}$ of eigenelements of $M$ with $\gamma (h) =\|h\|_{\B(V)}= 1$ satisfying for all $t\geq0$,
\begin{equation}\label{vecteursales}
\gamma M_t=\e^{\lambda t}\gamma\qquad\text{and}\qquad M_th=\e^{\lambda t}h.
\end{equation}
Moreover,  there exist $C,\omega>0$ such that for all $t\geq0$ and  $\mu \in \M( V)$, 
\begin{equation}\label{eq:conv_norm}
\big\|\e^{-\lambda t}\mu M_t-\mu(h)\gamma\big\|_{\M( V)}\leq C\left\|\mu\right\|_{\M( V)}\e^{-\omega t}.
\end{equation}
ii) Assume that there exist a triplet $(\gamma,h,\lambda)\in\M_+( V)\times \B_+( V)\times \mathbb{R}$ and constants $C,\omega>0$ such that
\eqref{vecteursales} and \eqref{eq:conv_norm} hold. Then, \pg{Assumption~${\bf A}$ is satisfied by the function $\psi=h$.}
\end{theo} 

%In case {\it i)},  we  add that  $\log(\beta)/\tau\leq\lambda\leq\log(\alpha+\theta)/\tau$ and $(\psi/V)^q \psi\lesssim h\lesssim V $ for some $q>0$.
%where we recall that for two functions $f,g : \Omega \rightarrow \R$, the notation $f\lesssim g$ means that there exists a constant $C>0$ such that $f\leq C g$ on $\Omega.$

%above are expressed only using the constants of Assumption {\bf A} in the proofs.

\subsection{Sufficient conditions: drift  and irreducibility}\label{ssec:generator}

Assumptions~\ref{A1}-\ref{A2} can be checked more easily using  the generator $\L$ of the semigroup $(M_t)_{t\geq0}.$
We give  convenient sufficient conditions by adopting a mild formulation of $\L=\partial_t{M_t}_{|_{t=0}}$, similar to  \cite{H16}.
For $F,G\in\B(V)$ we say that  $\L F=G$
if for all $x\in\X$ the function $s\mapsto M_sG(x)$ is locally integrable, and for all $t\geq0$,
\[M_t F=F+\int_0^tM_sG\,ds.\]
In general for $F\in \B(V)$,  there may not exist $G\in\B(V)$ such that $\L F=G,$ meaning that $F$ is not in the domain of $\L.$
Therefore we  relax the definition by saying that
\[ \L F\leq G,\qquad\text{resp.}\quad \L F\geq G,\]
if for all $t\geq0$
\[ M_t F-F\leq \int_0^tM_sG\,ds,\qquad\text{resp.}\quad M_t F-F\geq \int_0^tM_sG\,ds.\]
We can now state the drift conditions on $\L$ guaranteeing the validity of Assumptions~\ref{A1}-\ref{A2}. It will be useful for the applications in Section~\ref{sec:applications}.
For convenience, we use the shorthand $\varphi\simeq\psi$ to mean that $\psi\lesssim\varphi\lesssim\psi$, {\it i.e.} that the ratios of the two functions are bounded.

\begin{prop}
\label{prop:lyapun}
Let $V,\psi,\varphi:\X\to(0,\infty)$ such that  $\psi\leq V$ and  $\varphi\simeq \psi.$
Assume that there exist constants $a<b$ and  $\zeta \geq 0$, $\xi\in \R$ such that
\[\L V\leq aV+\zeta\psi,\qquad \L\psi\geq b\psi,\qquad\L\varphi\leq \xi\varphi.\]
Then, for any $\tau>0$, there exists $R>0$ such that $(V,\psi)$ satisfies~\ref{A1}-\ref{A2} with $K=\{V\leq R\,\psi\}$.
\end{prop}

We provide now a sufficient  condition for~\ref{A3}-\ref{A4}\pg{, which may typically apply in the case of a discrete state space.}

\begin{prop}
\label{prop:irr}
Let $K$ be a non-empty finite subset of $\X$ and assume that there exists $\tau>0$ such that for any $x,y\in K$,
$$\delta_xM_\tau(\mathbf 1_{\{y\}})>0.$$
Then \pg{there exists a probability measure $\nu$ supported by  $K$ such that}~\ref{A3}-\ref{A4} are satisfied for any positive function $\psi \in \mathcal B(V)$.
\end{prop}
This sufficient condition is relevant for the study of irreducible processes on discrete spaces.
\pg{For applications, let us remark that the two conditions can be easily combined. Indeed,
the conclusion of Proposition \ref{prop:lyapun} holds for any $R$ large enough and when $V/\psi$ goes to infinity, the set $K$ is finite and non empty for $R$ large enough.}
We refer to Section~\ref{QSD} for an application to the convergence to quasi-stationary distribution
of birth and death processes. As a motivation, let us also mention the study of the first moment
semigroup of discrete branching processes in continuous time \pg{(see \cite{B15, C17, M16} and the references therein)} and 
 the  exponential of
denumerable non-negative matrices~\cite{Nummelin}.
For  continuous state space, the irreducibility condition above is not  relevant, and we refer to~\cite{CloezGabriel} for more general conditions  to check~\ref{A4} via a coupling argument.

%Finally, let us mention  that these propositions provide explicit constants for Assumptions~\ref{A1}-\ref{A2} and~\ref{A3}-\ref{A4} respectively, see the proofs in Section~\ref{proofthun}. 

\section{Quantitative estimates}\label{sec:quantitative}
To exploit the conservative  theory, we   consider a relevant conservative propagator associated to $M$. 
We then derive   the expected estimates for the eigenelements and obtain the quantitative estimates for the original semigroup $M$.

\subsection{The embedded conservative propagator}\label{sect:auxiliaire}

Let us fix a positive function $\psi\in\B(V)$ and a time $t>0$. For any $0\leq s\leq u\leq t$, we define the operator $P^{(t)}_{s,u}$ acting on bounded measurable functions $f$ through
\begin{equation}
\label{eq:def-aux}
P^{(t)}_{s,u}f=\frac{M_{u-s}(fM_{t-u}\psi)}{M_{t-s}\psi}.
\end{equation}
We observe that the family $P^{(t)}=(P^{(t)}_{s,u})_{0\leq s\leq u\leq t}$  is a conservative propagator (or semiflow), meaning that for any $0\leq s\leq u\leq v\leq t$, \[P^{(t)}_{s,u}P^{(t)}_{u,v}=P^{(t)}_{s,v}.\] It   has a probabilistic interpretation in terms of particles systems, see {\it e.g.}  \cite{M16} and references below.
Roughly speaking, it provides the position of the  backward lineage  of a particle  at  time $t$ sampled with a bias $\psi$.

\

The particular case $\psi=\1$ corresponds to uniform sampling and  has been successfully used in the study of positive semigroups, see \cite{B15, BCG17, CV14,CV18,DelM04, M16}.
 Whenever possible, the right  eigenfunction of the semigroup provides a relevant choice for $\psi$.
Indeed,  if $h$ is a positive eigenfunction, \pg{the propagator $P^{(t)}$ associated to $\psi=h$ is actually a semigroup which, moreover, is independent of~$t$. More precisely, $P^{(t)}_{u,u+s}=M_s(h\, \cdot)/\e^{\lambda s}h$ does not depend on $u$ nor on $t$ and defines a semigroup $(P_s)_{s\geq0}$.} This corresponds to the $h$-transform given in Introduction.
This transformation  provides a powerful tool   for the analysis of branching processes and absorbed Markov process, see {\it e.g.} respectively  \cite{C17, EHK10,KLPP} and \cite{CMM13,D57}.
To shed some light on the sequel,  let us explain  how to apply Harris's ergodic theorem (\cite{HM11} or Theorem~\ref{th:ergo} in Appendix~\ref{sect:conservatif} with $\mathscr W=\mathscr V$) to $(P_s)_{s\geq0}$
and get the asymptotic behavior of   $M$.
Inequality~\eqref{assu:auxi-lyap} for $P_{\tau}$
reads $P_{\tau}\mathscr V\leq \mathfrak a \mathscr V  +\mathfrak c$ and  yields   $M_{\tau}(\mathscr V h)\leq \alpha \mathscr Vh+\theta h$, with $\alpha = \mathfrak a \e^{\lambda\tau},\ \theta = \mathfrak c \e^{\lambda\tau}$. 
This inequality involving $M_{\tau}$  is guaranteed by~\ref{A1} by setting $ V=\mathscr Vh$ and $\psi=h.$
Additionally, Equation~\eqref{assu:auxi-smallset} corresponds exactly to~\ref{A3}.

\

In this paper, we  deal with the general case and consider  a  positive function $\psi$ satisfying~\ref{A2}. The  analogy with the $h$-transform above suggests to 
 look for Lyapunov functions of the form $\mathscr V= V/\psi$. 
The family of functions $(V/M_{k \tau}\psi)_{k\geq 0}$ satisfies, under Assumption~\ref{A1}, an extended version of the Lyapunov condition for $P^{(t)}$.
But their  level sets may degenerate as $k$ goes to infinity, which raises a problem to check the small set condition.
We compensate the magnitude of $M_{k \tau}\psi$  and consider  for $k\geq 0$,
\begin{equation}\label{eq:def-lyap}
V_k= \nu\left(\frac{M_{k\tau}\psi}{\psi} \right)\frac{ V}{ M_{k\tau}\psi}.
\end{equation}
%This  family of Lyapunov  functions plays a key role in our analysis.
The two following lemmas, which are proved in Section~\ref{Ergsgaux}, ensure that $(V_k)_{k\geq 0}$  provides   Lyapunov functions whose sublevel sets are small for $P^{(t)}$.

\begin{lem}
\label{coro:Lyapunov}
\pg{Assume that $V$ and $\psi$ are such that Assumptions~\ref{A1}-\ref{A2}-\ref{A3} are met.}
Then for all integers $k\geq 0$ and $n\geq m\geq k+1$, we have
\begin{equation*}
%\label{eq:Lyapunov}
P^{(n\tau)}_{k\tau,m\tau} V_{n-m}  
\leq \mathfrak{a} V_{n-k} +\mathfrak c,
\end{equation*}
where 
\begin{equation}\label{eq:fraka-frakc}
\mathfrak a=\frac{\alpha}{\beta} \in (0,1), \qquad  \mathfrak c=\frac{ \C  }{c(\beta-\alpha)} \geq 0.
\end{equation}
\end{lem}

\begin{lem}\label{lem:minorization}
\pg{Suppose that Assumption~{\bf A} is satisfied and} let $\mathfrak R>0$.
Then there exists a family of probability measures $\{ \nu_{k,n}, \ k\leq n \}$  such that for all
$0\leq k\leq n-\RR$ and $x\in \{V_{n-k} \leq \mathfrak R \},$ 
\begin{align*}
\delta_x P^{(n\tau)}_{k\tau,(k+\RR)\tau} \geq \mathfrak b \,  \nu_{k,n},
\end{align*}
where $\RR\in\N$ and $\mathfrak b\in(0,1]$ are given by
\begin{equation}
\label{Rrrr}
\RR = \bigg\lfloor \frac{\log\left(\frac{2\mathfrak R (\alpha+ \C)}{ (\beta-\alpha)d}\right)}{\log(\beta/\alpha)}\bigg\rfloor +1
\qquad\text{and}\qquad
\mathfrak{b} = \frac{d^2\beta}{2 \mathfrak c^2 (\alpha/\theta+1)(\alpha R+\theta)}\, \frac{1}{\sum_{j=1}^\ell (\mathfrak a /cr)^j}
\end{equation}
with
 $ R = \sup_K\,  V / \psi$ and $r=\left(\beta/(\alpha(\ro+ \C/(\beta-\alpha)) +\C)\right)^2$.
\end{lem}

We can now state the key contraction result. Its proof lies in the two previous lemmas and a slight adaptation of Harris theorem provided in Appendix~\ref{sect:conservatif}.

\begin{prop}\label{prop:ergo-auxi}
Let $(V,\psi)$ be a couple of measurable functions from $\X$ to $(0,\infty)$ satisfying  Assumption {\bf A}. Let $\mathfrak R>2\mathfrak c/(1-\mathfrak a)$, $\mathfrak b'\in(0,\mathfrak b),$ $\mathfrak a'\in(\mathfrak a+2\mathfrak c/\mathfrak R,1)$ and set
$$
\kappa=\frac{\mathfrak b'}{\mathfrak c}, \qquad \mathfrak y= \max\left\{1-(\mathfrak b-\mathfrak b'),\frac{2+\kappa \mathfrak R \mathfrak a'}{2+\kappa \mathfrak R}\right\}.
$$ 
Then, for any $\mu_1,\mu_2 \in\M_+(V/\psi)$ and any integers $k$ and $n$ such that $0\leq k\leq n-\RR,$
\[\big\|\mu_1 P^{(n\tau)}_{k\tau ,(k+\RR)\tau}-\mu_2 P^{(n\tau)}_{k\tau,(k+\RR)\tau}\big\|_{\M(1+\kappa V_{n-k-\RR})}\leq \mathfrak y \left\|\mu_1-\mu_2\right\|_{\M(1+\kappa V_{n-k})}.\]
\end{prop}

\smallskip

\subsection{Eigenelements and quantitative estimates}\label{sec:eigenelements}
Using the notations introduced in the previous section, we set
\begin{align}\label{eq:autro}
\autro = \max\left\{\mathfrak y,\mathfrak{a}^{\RR} \right\}= \max\left\{1-(\mathfrak b-\mathfrak b'),\frac{2+\kappa \mathfrak R \mathfrak a'}{2+\kappa \mathfrak R},  \mathfrak{a}^{\RR}\right\} \in (0,1),
\end{align}
with $\mathfrak b'\in(0,\mathfrak b)$ and $\mathfrak a'\in(\mathfrak a+2\mathfrak c/\mathfrak R,1).$
We first deal with the right eigenelement.

\begin{lem}\label{lem:h}
\pg{Under Assumption~{\bf A},}
there exists $h\in \mathcal{B}_+(V)$ and $\lambda\in \mathbb{R}$ such that 
$$M_{\tau}h=\e^{\lambda \tau}h.$$
Moreover, we have the estimates
\[\left(\frac{\psi}{V} \right)^{q} \psi \lesssim h\lesssim V, \quad \text{with }q=\frac{\log(cr)}{\log(\mathfrak{a})}>0,\]
and there exists $C>0$ such that
for all integer $k\geq 0$ and $\mu \in \M_+(V)$,
\begin{equation}
\label{eq:h_speed_conv}
\left| \mu(h) - \frac{\mu M_{k\tau} \psi }{\nu( M_{k\tau} \psi/\psi)} \right| \leq C \frac{\mu(V)^2}{\mu(\psi)} \autro^{\lfloor k/\RR \rfloor}.
\end{equation}
\end{lem} 

\medskip
Let us turn to the left eigenmeasure and provide a similar result. 
\begin{lem}
\label{lem:asymptotic-measure}
\pg{Under Assumption~{\bf A},}
there exists $\gamma \in \mathcal{M}_+(V)$ such that $\gamma(h)=1$ and
$$\gamma M_{\tau} = \e^{\lambda \tau }\gamma.$$
Moreover,  there exists $C>0$ such that 
for all integer $k\geq 0$ and $\mu \in \M_+(V)$,
\begin{align}
\label{eq:cvgamma}
\left\| \frac{\gamma( \cdot\, \psi)}{\gamma(\psi)} - \frac{\mu M_{k\tau} (\cdot\, \psi)}{\mu M_{k\tau} (\psi)} \right\|_{\mathcal{M}(1+\kappa V_0)}\leq C\left(\frac{\mu(V)}{\mu(\psi)}+ \frac{\C}{\beta-\alpha}\right)\autro^{\left\lfloor k/\RR\right\rfloor} .
\end{align}
\end{lem}

\medskip

Using Lemma~\ref{lem:asymptotic-measure}, $\psi\leq V$, \ref{A1} and \ref{A2} yields
$\e^{\lambda \tau} \gamma(V) = \gamma M_{\tau}V \leq (\alpha+\theta) \gamma(V)$ and $\e^{\lambda \tau} \gamma(\psi) =  \gamma M_{\tau}\psi \geq \beta \gamma(\psi)$.
It gives the following estimate of the eigenvalue
\begin{equation}\label{eq:vpestimates}
\frac{\log(\beta)}{\tau}\leq\lambda\leq\frac{\log(\alpha+\theta)}{\tau}.
\end{equation}

\pg{Now we give a convergence result in discrete time.

\begin{prop} 
\label{th:main_discret}
Under Assumption~{\bf A},
there exists $C>0$ such that for all $\mu \in \M_+( V)$ and all integers $k\geq 0$ we have
\begin{equation}\label{eq:main_result1_discret}
\left\| \e^{-\lambda k \tau}\mu M_{k\tau} - \mu(h) \gamma  \right\|_{\mathcal{M}( V)} \leq C 
\, \frac{\mu(V)}{\mu(\psi)} \, \e^{-\sigma k\tau} \min\left\{\e^{-\lambda k\tau}\mu M_{k\tau} \psi, \mu(V)\right\},
\end{equation}
where 
$$
\sigma=\frac{-\log\autro}{\RR\tau} >0.
$$
\end{prop}
}

Theorem~\ref{V-contraction} follows from Lemmas \ref{lem:h} and \ref{lem:asymptotic-measure} \pg{and Proposition~\ref{th:main_discret}.
First, the convergence in~\eqref{eq:main_result1_discret} allows proving that $h$ and $\gamma$ are eigenelements of $M_t$ associated to the eigenvalue $\e^{\lambda t}$ for any positive time $t$.}
Then, the estimates on the eigenelements allow checking that $(V,h)$ satisfies Assumption~{\bf A}.
Consequently, \eqref{eq:cvgamma} with $\psi$ replaced by $h$  yields~\eqref{eq:conv_norm}.
Details are given in Section~\ref{proofthun}.

%\subsection{Further results about exponential convergence}\label{sec:uniformexpoconv}
%
%We give here complementary results about the convergence to the profile given by the eigenelements.
%The first result provides a more explicit rate of convergence than in Theorem~\ref{V-contraction}, \pg{but the dependence on $\mu$ of the bound is more involved.}
%It requires the following additional assumption to capture the continuous time:
%\begin{equation}\label{as:psi}
%t\mapsto \sup_{x\in\X}\frac{M_tV(x)}{V(x)} \, \text{ and } \, 
%t\mapsto \sup_{x\in\X}\frac{\psi(x)}{M_t\psi(x)}\quad\text{are locally bounded on}\ [0,\infty).
%\end{equation}
%
%\begin{theo} 
%\label{th:main}
%Under Assumptions~{\bf A} and \eqref{as:psi}, there exists $C>0$ such that for all $\mu \in \M_+( V)$ and $t\geq 0$,  
%\begin{equation*}%\label{eq:main_result1}
%\left\| \e^{-\lambda t}\mu M_t - \mu(h) \gamma  \right\|_{\mathcal{M}( V)} \leq C \frac{ \mu(V)}{\mu(\psi)} \e^{-\sigma t} \min\left\{\e^{-\lambda t}\mu M_t \psi, \mu(V)\right\},
%\end{equation*}
%where 
%$$
%\sigma=\frac{-\log\autro}{\RR\tau} >0.
%$$
%\end{theo}

\pg{We end this section by stating another exponential convergence result where the semigroup is  normalized by its mass $M_t \mathbf 1$. Such a normalization can be relevant for applications.}
Let us mention for instance the study of the convergence of the conditional probability to a quasi-stationary distribution and the study of the typical trait in a structured branching process, see respectively Section~\ref{QSD} and {\it e.g.} \cite{BCG17, marguet17, M16}.
We denote by $\P(V)$ the set of probability measures which belong to $\M(V)$ and we define the total variation norm for finite signed measures by
$\left\|\mu\right\|_{\mathrm{TV}}=\left\|\mu\right\|_{\M(\1)}$.
The following result is direct a corollary of Theorem~\ref{V-contraction}.
\begin{coro}
\label{lesnormandsmecassentlescouillesavecleurslabels}
Assume that the conditions of Theorem~\ref{V-contraction}-{\it(i)} hold and 
 $\inf_{\X} V>0$.
Then there exist $C,\omega>0$ and $\pi\in \mathcal{P}(V)$ such that for every  $\mu \in \M_+( V)$  and $t\geq 0$,
\begin{equation}\label{eq:cor_result1}
\left\| \frac{\mu M_t}{\mu M_t \mathbf 1}  -  \pi \right\|_{\mathrm{TV}} \leq C \frac{\mu(V)}{\mu(h)} \e^{-\omega t}.
\end{equation}
\end{coro}

\smallskip

\section{Proofs}\label{sec:proofs}

\subsection{Preliminary inequalities}

For all $t\geq 0$, let us define the following  operator
$$\widehat{M}_t:f\mapsto M_{t}(\mathbf{1}_{K^c} f),$$
and for convenience, we introduce  the following constants
\begin{equation}
\label{eq:thetaR}
  \CC=\frac{\C}{\beta-\alpha}, \qquad R= \sup_K \frac{V}{\psi}, \qquad \boulet=\alpha(\ro+\CC)+ \C,
  \end{equation}
which are well-defined and finite under Assumption {\bf A}. We first give  some estimates  which are directly deduced from Assumptions~\ref{A1}-\ref{A2}-\ref{A3}.
%Note that all constants are well defined under Assumption~\ref{assu:condition_generales}. \cmb{Mettre dans tous les Lemmes etc, Sous l'hypothèse A, il n'est peut-être pas clair quelle lemme le suppose ou pas avec quelle V et psi quand on regarde toutes les differentes sections...}

\begin{lem}
\label{lem:first-est}
\pg{Under Assumptions~\ref{A1}-\ref{A2}-\ref{A3} we have, for all integer $k\geq 0$,}
\begin{enumerate}[label= \roman*), parsep=3mm]
\item %, we have \cmb{Pour que ca soit plus facile a lire il faudra faire gaffe a mettre des $k$ quand ca va varier et agrder $n$ pour les valeurs terminales}
$\widehat{M}^k_{\tau} M_{\tau}  V \leq \alpha^k M_{\tau} V,$
\item for all $\mu \in \mathcal{M}_+(V)$,
$$
\frac{\mu M_{k\tau} V}{\mu M_{k\tau}\psi} \leq \mathfrak a^k \frac{\mu ( V)}{\mu(\psi)} + \CC,
$$
%\item for all $x\in K$
%$$M_{k\tau}  V(x) \leq (\ro+\CC)  M_{k\tau} \psi(x),
%$$
\item for all $x\in K$ and $n\geq k$,
$M_{n\tau} \psi(x) \leq \boulet^{k} M_{(n-k) \tau} \psi(x),$
\item for all $x\in K$, and $f\in \mathcal{B}_+(V/\psi)$,
\begin{equation}
\label{lestic}
M_{(k+1)\tau}\left(f\psi\right)(x)\geq c_{k+1} \nu(f) M_{(k+1)\tau}\psi(x),
\quad \text{with } \ 
c_{k+1} = c^{k+1}  \left(\beta/ \boulet \right)^{k}.
\end{equation}
 \end{enumerate}
\end{lem}

\begin{Rq}
We observe that $(c_k)_{k\geq 1}$ is a decreasing geometric sequence. Indeed since $\psi\leq V$, \ref{A1} and~\ref{A2} ensure that on $K$,
$\beta\psi\leq M_{\tau}\psi\leq M_{\tau}V\leq (\alpha R +\theta)\psi,$
so that $\beta<\boulet$. Adding that $c<1$ ensures that
 $(c_k)_{k\geq 1}$ decreases geometrically. \\
 Points $i)$ and $ii)$ of Lemma~\ref{lem:first-est} are sharp inequalities , while  $iv)$ extends Assumption~\ref{A3} for any time.
\end{Rq}

\begin{proof} Using  \ref{A1} we readily have $\1_{K^c}M_{\tau} V\leq \alpha  V$
and $i)$ follows by induction.\\
Composing respectively \ref{A1} and \ref{A2} with $M_{k\tau}$ yields
$$M_{(k+1)\tau} V \leq \alpha M_{k\tau} V+ \C M_{k\tau}\psi; \quad  M_{(k+1)\tau}\psi\geq \beta M_{k\tau}\psi.$$
Combining these inequalities gives 
$$\frac{M_{(k+1)\tau} V}{ M_{(k+1)\tau}\psi} \leq \mathfrak a \frac{M_{k\tau} V}{M_{k\tau}\psi} +\frac{\C}{\beta}$$
and $ii)$ follows by induction, recalling that $\mathfrak a<1$.\\
By definition of $R$, we immediately deduce from $ii)$ that for any $x\in K$ and any integer $k$,
$$\frac{ M_{k\tau} V(x)}{ M_{k\tau}\psi (x)} \leq \ro+\CC.$$
Combining this inequality with $M_{n\tau} \psi\leq M_{(n-1)\tau} M_{\tau}  V\leq  M_{(n-1)\tau} (\alpha  V+ \C\psi),$ which results from \ref{A1} and $\psi\leq  V$, we get that $M_{n\tau} \psi(x) \leq \boulet M_{(n-1) \tau} \psi(x)$ for all $x\in K$.
The proof of $iii)$ is completed by induction.\\
Finally, for $iv)$, we have for any $x\in K$, using \ref{A3} with the function  $M_{n\tau}\left(f\psi\right) / \psi$,
\begin{align*}
\frac{M_{(n+1)\tau}\left(f\psi\right)(x)}{M_{(n+1)\tau}\psi(x)} & =\frac{M_{\tau}(M_{n\tau}\left(f\psi\right))(x)}{M_{(n+1)\tau}\psi(x)} \geq c \nu\left(\frac{M_{n\tau}\left(f\psi\right)}{\psi}\right)\frac{M_{\tau}\psi(x)}{M_{(n+1)\tau}\psi(x)}.
\end{align*}
Besides, since $\nu$ is supported by $K$, \ref{A3} and \ref{A2} yield
\[ \nu\!\left(\!\frac{M_{n\tau}(f\psi)}{\psi}\!\right)= \nu\!\left(\!\frac{M_\tau ( M_{(n-1)\tau}(f\psi))}{\psi}\!\right)
 \geq c\nu\!\left(\!\nu\!\left(\!\frac{M_{(n-1)\tau}(f\psi)}{\psi}\!\right)\!\frac{M_{\tau}\psi}{\psi}\!\right) \geq c\beta\nu\!\left(\!\frac{M_{(n-1)\tau}(f\psi)}{\psi}\!\right).\]
Iterating the last inequality and plugging it in the previous one, we obtain
\[\frac{M_{(n+1)\tau}\left(f\psi\right)(x)}{M_{(n+1)\tau}\psi(x)}\geq c^{n+1}\beta^n\frac{M_{\tau}\psi(x)}{M_{(n+1)\tau}\psi(x)}\nu(f)
\geq c^{n+1}  \frac{\beta^n}{\boulet^n} \nu(f),\]
where the last inequality comes from $iii)$.
The proof is complete.
\end{proof}

\subsection{Contraction property: proofs of Section \ref{sect:auxiliaire}}
\label{Ergsgaux}

First, we prove   that $(V_k)_{k\geq 0}$ is a family of Lyapunov functions for the sequence of operators $(P^{(n\tau)}_{k\tau, (k+1)\tau})_{0 \leq k \leq n-1}$.
\begin{lem}
\label{lem:lyap}
\pg{Under Assumptions~\ref{A1}-\ref{A2}-\ref{A3} we have, for all $k\geq 0$ and $n\geq m\geq k$,}
\begin{align*}
%\label{eq:majV}
P^{(n\tau)}_{k\tau,m\tau} &V_{n-m}  
\leq \mathfrak{a}^{m-k} V_{n-k} + \frac{\C}{c \beta} \sum_{j=k}^{m-1} {\mathfrak a}^{m-j} P^{(n\tau)}_{k\tau,j\tau} \left( \mathbf{1}_K\right).
\end{align*}
\end{lem}
%Note that inequality \eqref{eq:Lyapunov} holds in particular for $k_0=1$ but it will be not efficient to apply Theorem~\ref{th:ergo}. Indeed, Assumption~\eqref{assu:condition_generales2} do not apply on the sublevel set of $V$.  Our strategy is then as in \cite[Theorem 3.4]{HM} CITER Hairer Mat YET another la deuxieme version a la fin!! to use $k_0$ larger than $1$. \cmb{!}\cma{?}
\begin{proof}
By definition of $V_k$ in \eqref{eq:def-lyap}, we have, for $0\leq k\leq n$, 
\[P^{(n\tau)}_{(k-1)\tau,k\tau} V_{n-k} = \frac{M_{\tau}\left( V_{n-k} M_{(n-k) \tau}\psi \right)}{M_{(n-k+1) \tau} \psi}=  \nu\left( \frac{M_{(n-k) \tau} \psi}{\psi}\right) \frac{M_{\tau}  V }{M_{(n-k+1) \tau} \psi}\]
Using \ref{A1} and \ref{A2}, we have  $M_{\tau}  V \leq \alpha  V + \C\psi\mathbf{1}_K$ and 
$M_{(n-k) \tau} \psi\leq  M_{(n-k+1) \tau} \psi/\beta$. We   obtain from the definitions of $\mathfrak a$ and $V_{n-k+1}$ that
\begin{align*}
P^{(n\tau)}_{(k-1)\tau,k\tau} V_{n-k}  &\leq %\nu\left( \frac{M_{(n-k) \tau} \psi}{\psi}\right)  \frac{\alpha  V + \C\psi\mathbf{1}_K }{M_{(n-k+1) \tau} \psi}
\mathfrak a V_{n-k+1} + \nu\left( \frac{M_{(n-k) \tau} \psi}{\psi}\right) \frac{ \C\psi \mathbf{1}_K }{M_{(n-k+1) \tau} \psi} . \label{bornecommebertrand}
\end{align*}
Besides, combining \ref{A2} and \ref{A3} with $f=M_{(n-k)\tau}\psi/\psi$, we get %for every $j\leq k_0$,
\begin{align*}
\nu\left( \frac{M_{(n-k) \tau} \psi}{\psi}\right) \frac{  \psi \, \mathbf{1}_K }{M_{(n-k+1) \tau} \psi}\leq \frac{\mathbf{1}_K }{ c \beta}.
\end{align*}
The last two inequalities yield
\begin{align*}
P^{(n\tau)}_{(k-1)\tau,k\tau} V_{n-k} 
\leq \mathfrak a V_{n-k+1} +\frac{\C \, \mathbf{1}_K}{ c\beta}.
\end{align*}
The conclusion follows from  $P^{(n\tau)}_{k\tau,m\tau}  V_{n-m}=P^{(n\tau)}_{k\tau,(k+1)\tau}\cdots P^{(n\tau)}_{(m-1)\tau,m\tau} V_{n-m}.$
\end{proof}

\begin{proof}[Proof of Lemma~\ref{coro:Lyapunov}]
Using that
 $P^{(n\tau)}_{k\tau,(j-1)\tau} \left( \mathbf{1}_K\right)\leq 1$ and $\mathfrak a<1$, it is % Lemma  \ref{coro:Lyapunov}
 a direct consequence of Lemma~\ref{lem:lyap}.
 \end{proof}

Using \ref{A3} and \ref{A4} and following \cite{BCG17,CV14}, we prove a small set  condition~\eqref{assu:auxi-smallset} on the set $K$ for the embedded propagator $P^{(t)}$. However, Theorem~\ref{th:ergo} requires that \eqref{assu:auxi-smallset} is satisfied on a sublevel set of $V_k$. Although there exists $R>0$ such that $K\subset\lbrace V_k\leq R\rbrace$, nothing guarantees the other inclusion.
% However, $K$ is  not in general a sublevel set of $V_k$.
This situation is reminiscent of \cite[Assumption 3]{HM11} and we adapt here their arguments. For that purpose, we need a lower bound for the Lyapunov functions $(V_k)_{k\geq 0}$, which is stated in the next lemma.

\begin{lem}%[Upper bound on  eigenvector the approximation ou Preliminary estimate]
%[Lower bound on the eigenvalue approximation]
\label{lem:approxvp}
\pg{Under Assumptions~\ref{A1}-\ref{A2}-\ref{A4} we have, for every $n\geq 0$,}
\begin{align*}
d_1M_{(n+1)\tau} \psi 
&\leq \nu \left( \frac{M_{n \tau}\psi}{\psi} \right)  M_{\tau}  V\ \qquad\text{and}\ \qquad V_n\geq d_2,
\end{align*}
with
$d_1=\big(1-\mathfrak a\big)d,  \,  d_2=(\beta-\alpha) d/(\alpha+\theta).$%\frac{\beta}{\alpha+ \C}d_1.\end{equation*}
\end{lem}

\begin{proof} First, using \ref{A4},
\begin{align}
d\,M_{(n+1)\tau} \psi
&= d\,M_{\tau} \left(  \mathbf{1}_K M_{n\tau}\psi +  \mathbf{1}_{K^c} M_{n\tau}\psi  \right)%\nonumber \\
%&\leq M_{\tau} \left[  \nu \left( \frac{M_{n\tau}\psi}{\psi} \right) \mathbf{1}_K \, \psi +d\, \mathbf{1}_{K^c}  M_{n\tau}\psi  \right]
%&\leq d_2 \nu \left( \frac{M_{n\tau}\psi}{\psi} \right) M_{\tau} \psi + M_{\tau} \left[ \mathbf{1}_{K^c}  M_{n\tau}\psi  \right]\nonumber\\
\leq \nu \left( \frac{M_{n\tau}\psi}{\psi} \right) M_{\tau} \psi + d\,\widehat{M}_{\tau}  M_{n\tau}\psi.\nonumber%\label{eq:maj-init}
\end{align}
Then, by iteration, using \ref{A2} and $\psi\leq  V,$ 
\begin{align*}
d\,M_{(n+1)\tau} \psi 
&\leq   \nu \left( \frac{M_{n\tau}\psi}{\psi} \right) \sum_{j=0}^{n} \beta^{-j} \widehat{M}_{\tau}^j M_{\tau} \psi \leq   \nu \left( \frac{M_{n\tau}\psi}{\psi} \right) \sum_{j=0}^{n} \beta^{-j} \widehat{M}_{\tau}^j M_{\tau}  V.
\end{align*}
Hence by Lemma~\ref{lem:first-est} $i)$,
\begin{align*}
d\,M_{(n+1)\tau} \psi 
&\leq \frac{1 }{1-\mathfrak{a}} \nu \left( \frac{M_{n\tau}\psi}{\psi} \right)  M_{\tau}  V
\end{align*}
and the first identity is proved.
From the definition of $V_n$ we deduce
$$V_n\geq d_1\frac{M_{(n+1)\tau}\psi}{M_{n\tau} \psi}\frac{ V}{M_{\tau} V}\geq d_1\frac{\beta}{\alpha+\C}=d_2$$
by using successively \ref{A2},  \ref{A1}, and $\psi \leq  V$.
\end{proof}

We now prove the small set condition \eqref{assu:auxi-smallset} for the embedded  propagator.

\begin{proof}[Proof of Lemma~\ref{lem:minorization}] %[Proof of Proposition~\ref{prop:ergo-auxi}]
First, we introduce the measure $\nu_i$ defined  by
$$\nu_i(f)=\nu\left(f \frac{  M_{i\tau} \psi}{\psi} \right)$$
for all integer $i\geq 0$. For any $x\in K$, $j\leq k\leq n$, we have using Lemma~\ref{lem:first-est} $iv)$ with the function $fM_{(n-k)\tau} \psi/\psi$, 
\begin{align*}
P^{(n\tau )}_{(j-1)\tau,k\tau} f (x)
&=\frac{ M_{(k-j+1) \tau} \left( f M_{(n-k)\tau} \psi \right)(x)}{ M_{(n-j+1)\tau} \psi (x)}  \geq c_{k-j+1} \ \nu_{n-k}
\left( f \right)  \frac{M_{(k-j+1) \tau }\psi(x)}{ M_{(n-j+1)\tau}\psi(x)}.
\end{align*}
for any nonnegative measurable function $f$. Then, Lemma~\ref{lem:approxvp} and \ref{A2}  yield
\begin{align*}
P^{(n\tau )}_{(j-1)\tau,k\tau} f (x)
&\geq d_1c_{k-j+1} \frac{\nu_{n-k}
\left( f \right)}{\nu_{n-j}(\bf{1}) }  \frac{M_{(k-j+1) \tau}\psi(x)}{ M_{\tau} V(x)}\geq  d_1 c_{k-j+1}\beta^{k-j} \frac{\nu_{n-k}
\left( f \right)}{\nu_{n-j}(\bf{1}) }  \frac{M_{\tau}\psi(x)}{M_{\tau} V(x)}.
\end{align*}
Recalling from \ref{A1} and \ref{A2} that for $x\in K$,
$M_{\tau}\psi(x)/M_{\tau} V(x)\geq \beta/(\alpha\ro+\C), $
and from Lemma~\ref{lem:first-est} $iii)$  and  $\nu(K)=1$ that
$\nu_{n-k}({\bf 1})/\nu_{n-j}({\bf 1})\geq  \boulet^{-(k-j)},$
we get 
\begin{equation} \label{eq:firstmin}
 P^{(n\tau )}_{(j-1)\tau,k\tau} f (x)
\geq \alpha_{k-j}  \frac{\nu_{n-k}
\left( f \right)}{\nu_{n-k}(\bf{1})},
\end{equation}
for $x\in K$, where  $\alpha_i=  d_1 c^{i+1} \beta r^{i}/ (\alpha\ro+\C)$ and  $r=\left(\beta/\boulet\right)^{2}.$
The  previous bound holds only on $K$. We  prove now that the propagator charges $K$ at
an intermediate time and derive the expected lower bound. More precisely, setting
 %\cmp{ici la formule de $S_\ell$ n'est valable que si $\mathfrak a\neq cr,$ sinon le deuxième quotient doit être remplacé par $\ell.$ Mais bon c'est un détail qui allourdirait l'écriture donc je pense qu'on peut s'asseoir dessus}
$$
\omega_{i} = \frac{\mathfrak a^{i}}{\alpha_{i}}\quad  \text{ and }\quad  S_{\ell} =  \sum_{j=1}^{\ell}\omega_{\ell-j}=\frac{\alpha R+\theta}{dc(\beta-\alpha)}\sum_{j=1}^\ell \left(\frac{\mathfrak a}{cr}\right)^j,%= \frac{(\alpha R+\theta)}{d_1c\beta}\frac{1-(\mathfrak{a}/cr)^\ell}{1-\mathfrak{a}/cr},
$$
we obtain for $k\leq n-1$ and  $1\leq \ell \leq n-k$ ,
\begin{align*}
P^{(n\tau)}_{k\tau,(k+\ell)\tau}f 
%&= \frac{1}{S_{\ell} }  \sum_{j=k+1}^{k+\ell}  \omega_{k+\ell-j}  P^{(n\tau)}_{k\tau,(j-1)\tau} P^{(n\tau)}_{(j-1)\tau,(k+\ell)\tau}(f)\\
&\geq  \frac{1}{S_{\ell}}  \sum_{j=k+1}^{k+\ell}  \omega_{k+\ell-j}  P^{(n\tau)}_{k\tau,(j-1)\tau} (\1_KP^{(n\tau)}_{(j-1)\tau,(k+\ell)\tau}f)\geq   B^{(\ell)}_{k,n} \frac{\nu_{n-k-\ell}
\left( f \right)}{\nu_{n-k-\ell}(\bf{1}) },
\end{align*}
where the last inequality comes from \eqref{eq:firstmin} and
$B^{(\ell)}_{k,n}=\frac{1}{S_{\ell}}  \ \sum_{j=k+1}^{k+\ell}  {\mathfrak a}^{k+\ell-j} P^{(n\tau)}_{k\tau,(j-1)\tau} \1_K. $

To conclude, we need to find a positive lower bound for $B^{(\ell)}_{k,n}$ which does not depend on $k$ or $n$. For that purpose,
we first observe that the second bound of Lemma~\ref{lem:approxvp}  ensures that $P^{(n\tau)}_{k\tau,(k+\ell)\tau} V_{n-k-\ell}  \geq d_2$. Using now   Lemma~\ref{lem:lyap} yields
$$\sum_{j=k+1}^{k+\ell} \mathfrak{a}^{k+\ell-j} P^{(n\tau)}_{k\tau,(j-1)\tau}  \mathbf{1}_K 
\geq c \beta \frac{d_2-\mathfrak{a}^{\ell} V_{n-k}}{\C},
$$
for $n\geq k+\ell$.
For  $x\in \{V_{n-k} \leq \mathfrak R \}$ and  $\ell=\RR$ defined in \eqref{Rrrr}, we get 
$$
B_{k,n}^{(\RR)}(x)
 \geq c \beta \frac{d_2}{2 \C S_{\RR}}=\frac{c^2\beta^3d_1^2}{2 \C (\alpha+\theta)(\alpha R+\theta)} \frac{1}{\sum_{j=1}^\RR (\mathfrak a /cr)^j},
  %\frac{1-\mathfrak{a}/cr}{1-(\mathfrak{a}/cr)^\RR},
$$
which ends the proof.
\end{proof}

\begin{proof}[Proof of Proposition~\ref{prop:ergo-auxi}]
Let $n$ and $k$ be two integers such that $0\leq k\leq n-\mathfrak p$ and consider $\mathfrak R> 2\mathfrak c/(1-\mathfrak a).$
According to Lemmas~\ref{coro:Lyapunov} and \ref{lem:minorization}, the conservative operator $P^{(n\tau)}_{k\tau,(k+\RR)\tau}$ satisfies condition \eqref{assu:auxi-lyap}
with the functions $V_{n-k-\RR}$ and $V_{n-k}$ and condition \eqref{assu:auxi-smallset} with the probability measure $\nu_{k,n}$.
Applying Theorem~\ref{th:ergo} then yields the contraction result.
\end{proof}

\subsection{Eigenelements: proofs  of Section~\ref{sec:eigenelements}}

%The proof of Theorem~\ref{th:main} is split into several lemmas. 
%We introduce the following constant
%\begin{equation}
%\label{defC1}
%C_0 = \sup_{s\leq \RR\tau}\, \max \left\{ \left\|\frac{M_s V}{V} \right\|_\infty, \left\|\frac{\psi}{M_s \psi} \right\|_\infty \right\},
%\end{equation}
%which is finite under Assumption~\eqref{as:psi}.
 Let us consider, for every $\mu \in \mathcal{M}_+(V)$, the family of operators $(Q^{\mu}_t)_{t\geq 0}$ defined for $f \in \mathcal{B}(V_0)$ by
$$
 Q^\mu_t f =\frac{\mu M_t (\psi f)}{\mu M_t (\psi)}.
$$
Fixing the measure $\mu$, the operator $f\mapsto Q^{\mu}_tf$ is linear.
Observe that $Q^{\delta_x}_t=\delta_x P^{(t)}_{0,t}$ so that Proposition~\ref{prop:ergo-auxi} implies contraction  inequalities for  $\delta_x \mapsto Q^{\delta_x}_{n \RR \tau}$.
Notice that $\mu\mapsto  Q^{\mu}_{n \RR \tau}$ is non-linear and forthcoming  Lemma~\ref{lem:contraction_M} extends the contraction to a more general space of measures. 
Besides, for any positive measure $\mu$, we set
$$\constante{\mu}=\frac{\mu( V)}{\mu(\psi)}.$$
We prove now the existence of the eigenvector and eigenmeasure, respectively stated in Lemma~\ref{lem:h} and Lemma~\ref{lem:asymptotic-measure}. 
%The section ends with the proof of Theorem~\ref{th:main}.\\
%For  convenience, we introduce  the following constants
%\begin{align}\label{eq:thetaR}
 %\CC=\frac{\C}{\beta-\alpha}, \quad R= \sup_K \frac{V}{\psi}, 
%\end{align}
%which are well-defined and finite under Assumption {\bf A}.
Let us first provide a useful upper bound for $V_k$. For that purpose, we also set
\begin{equation}
\label{defp}
p = \left\lfloor \frac{\log\big(2 (1+\C/\alpha)(\CC+ \ro)\big)}{\log\left(1/\mathfrak a\right)}  \right\rfloor +1, \qquad C_1=  \frac{ 2\,\boulet^{p+1}}{c  c_{p-1} \beta^{p+1} }.
\end{equation}
where $(c_k)_{k\geq 0}$ is defined in \eqref{lestic}.
\begin{lem}
\label{lem:upperbound}
\pg{Assume that Assumptions~\ref{A1}-\ref{A2}-\ref{A3} are met. Then,} for all positive measure~$\mu$ such that 
\begin{equation}
\label{eq:mubound}
\constante{\mu}\leq \CC+ \ro,
\end{equation}
we have for all $k\geq p$, 
\begin{align}
\label{laborne}
\nu\left(\frac{M_{k \tau}\psi}{\psi}\right)\leq C_1 \frac{\mu M_{k\tau}\psi}{\mu(\psi)}.
\end{align}
\end{lem}
The idea is the following: condition  \eqref{eq:mubound} ensures the existence of a time $p$ at which the propagator charges $K$. Then, \ref{A3} yields \eqref{laborne}.
It will be needed in this form in the sequel,
but could be extended to more general right-hand side in  \eqref{eq:mubound}.
\begin{proof}
Recalling that $\widehat{M}_{\tau} = M_{\tau}\left(\mathbf{1}_{K^c} \ \cdot \ \right)$ and using that for all $g\in \mathcal{B}(V)$,
$M_{(k+1)\tau}g=M_{\tau}(\mathbf{1}_KM_{k\tau}g)+\widehat{M}_{\tau}(M_{k\tau}g)$, we obtain by induction
%. For $f = \mathbf{1}_K\psi$ and $k\in\mathbb{N}$,  we have
\[M_{k \tau}g = \widehat{M}^{k}g+\sum_{j=1}^{k} \widehat{M}_{\tau}^{k-j}M_{\tau}\left(\mathbf{1}_K M_{(j-1) \tau}g\right).\]
Choose $g= \psi \mathbf{1}_K$. Using Lemma~\ref{lem:first-est} $iv)$ with $f=\mathbf{1}_K$ and that $\nu(K) = 1$, we get
\begin{align*}
M_{k\tau}\left(\mathbf{1}_K\psi\right)
%&\geq  \sum_{j=1}^k \widehat{M}_{\tau}^{k-j}M_{\tau}\left(\mathbf{1}_K M_{(j-1) \tau}(\1_K\psi)\right)\\
&\geq  \sum_{j=1}^k  c_{j-1}\widehat{M}_{\tau}^{k-j}M_{\tau}\left(\mathbf{1}_K M_{(j-1) \tau}\psi\right)\\
%&\geq  c_{k-1} \sum_{j=1}^k  \widehat{M}_{\tau}^{k-j}M_{\tau}\left(\mathbf{1}_K M_{(j-1) \tau}\psi\right)\\
&\geq c_{k-1} \sum_{j=1}^k \left(  \widehat{M}_{\tau}^{k-j}M_{j\tau}\psi -  \widehat{M}_{\tau}^{k-j}M_{\tau}\left(\mathbf{1}_{K^c} M_{(j-1) \tau}\psi\right) \right)
%&= c_{k-1} \sum_{j=1}^k \left(  \widehat{M}^{k-j}_{\tau}M_{j\tau}\psi - \widehat{M}_{\tau}^{k-j+1} M_{(j-1) \tau}\psi \right)\\
= c_{k-1}  \left(  M_{k\tau} \psi - \widehat{M}_{\tau}^{k} \psi \right),
\end{align*}
with the convention that $c_{0}=1$. Then, using \ref{A2} and the fact that $\widehat{M}_{\tau}^{k} \psi\leq \widehat{M}_{\tau}^{k-1}M_{\tau}  V$ together with Lemma~\ref{lem:first-est} $i)$, we arrive at
$M_{k\tau}\left(\mathbf{1}_K\psi\right) \geq c_{k-1} \left( \beta^{k} \psi- \alpha^{k-1} M_{\tau} V\right).$
Next, \ref{A1} and the fact that $ V \geq \psi$ yield
\begin{align}
\label{lamin}
M_{k\tau}\left(\mathbf{1}_K\psi\right) 
& \geq c_{k-1} \beta^{k} \left(  \psi- \mathfrak a^{k}\left(1+\C/\alpha \right)  V \right).
\end{align}
Using that $ \mathfrak a\leq  1$, the definition \eqref{defp} of $p$ ensures that
$\mathfrak a^{p} 2\left(1+\C/\alpha \right)  (\CC+ \ro) \leq 1,$
and \eqref{eq:mubound} yields
$\mathfrak a^{p}\left(1+\C/\alpha \right)\mu(V)\leq \mu(\psi)/2.$
Then, we deduce from \eqref{lamin} that
\begin{equation}
\label{eq:vaten}
\mu M_{p \tau} (\mathbf{1}_K \psi) \geq c_{p-1} \beta^{p}  \mu (\psi)/2.
\end{equation}
Using $\mu M_{k\tau} \psi 
\geq \mu M_{p \tau}(\mathbf{1}_K M_{(k-p)\tau} \psi)= 
\mu M_{p \tau}(\mathbf{1}_K M_{\tau} ((M_{(k-p-1)\tau}\psi/\psi) \psi))$ for $k\geq p$ and successively \ref{A3} with $f = M_{(k-p-1)\tau}\psi/\psi$, \ref{A2} and \eqref{eq:vaten}, we get
\begin{align*}
\mu M_{k\tau} \psi 
&\geq c  c_{p-1} \beta^{p+1} \frac{\mu(\psi)}{2} \nu \left(\frac{M_{(k-p-1)\tau} \psi}{\psi}\right).
\end{align*}
Finally, combining this estimate with Lemma~\ref{lem:first-est} $iii)$ ensures that
\[  \nu \left(\frac{M_{k \tau} \psi}{\psi}\right) \leq \boulet^{p+1}  \nu \left(\frac{M_{(k-p-1) \tau} \psi}{\psi}\right) \leq C_1  \frac{\mu M_{k\tau} \psi }{\mu(\psi)},\]
which ends the proof.
\end{proof}

%\begin{lem}
%\label{lem:upperboundcontinuous}
%For all positive measure $\mu$ such that 
%\begin{equation*}
%\frac{\mu( V)}{\mu(\psi)}\leq \CC+ \ro,
%\end{equation*}
%we have for all $s\geq p\tau$, 
%\begin{align*}
%\nu\left(\frac{M_{s}\psi}{\psi}\right)\leq C_1' \frac{\mu M_{s}\psi}{\mu(\psi)},
%\end{align*}
%where $C_1' = C_0^2C_1R$.
%\end{lem}
%\begin{proof}
%Let $u = s-\lfloor s/\tau\rfloor\tau.$ First, by definition of $C_0$ in~\eqref{defC1}, we have
%\begin{align}\label{eq:1}
%M_u \psi \geq C_0^{-1} \psi.
%\end{align}
%Moreover, for $x\in K$, using that $\psi<V$ and the definition of $R$, we get
%\begin{align}\label{eq:2}
%M_u \psi(x) \leq M_u V(x)\leq C_0V(x)\leq C_0R\psi(x).
%\end{align}
%Then, using successively \eqref{eq:2} combined with the fact that $\nu(K)=1$, Lemma \ref{lem:upperbound}, and \eqref{eq:1}, we get
%\begin{align*}
%\nu\left(\frac{M_{s}\psi}{\psi}\right)\leq C_0R\nu\left(\frac{M_{\lfloor s/\tau\rfloor\tau}\psi}{\psi}\right)\leq C_0RC_1 \frac{\mu M_{\lfloor s/\tau\rfloor\tau}\psi}{\mu(\psi)}\leq C_1' \frac{\mu M_{s}\psi}{\mu(\psi)},
%\end{align*}
%which ends the proof.
%\end{proof}

We extend now Proposition~\ref{prop:ergo-auxi} to  $(Q^{\mu}_{n \RR \tau})_{n\geq 0}$. 
Recall that $\RR$ is defined in Lemma~\ref{lem:minorization}, $\kappa$  and $\mathfrak y$ are defined in Proposition~\ref{prop:ergo-auxi} and $\autro$ is defined in \eqref{eq:autro}.
\begin{lem}\label{lem:contraction_M} %[Contraction type inequality]\
\pg{Under Assumption~{\bf A} we have,} for all measures $\mu_1,\mu_2\in\M_+(V/\psi)$ and all $n\geq 0$,
\begin{equation}
\label{eq:contractionM}
\left\|Q^{\mu_1}_{n\RR\tau} - Q^{\mu_2}_{n\RR\tau}\right\|_{\mathcal{M}(1+\kappa V_0)}\leq C_2 \autro^n \left( \constante{\mu_1}+ \constante{\mu_2} \right),
\end{equation}
where %$t_2 = k_1\tau$ is defined in Proposition \ref{prop:ergo-auxi}, 
\begin{equation*}
C_2= \max\left\{2  \mathfrak a^{-\RR} + \kappa C_1 \left(1 + 2 \Theta \mathfrak a^{-\RR}\right), 2\left(1+\kappa \Theta\right)\mathfrak a^{-(p+ \RR)}  + \kappa \right\}
\end{equation*}
with $C_1$ and $p$ defined in \eqref{defp}.
%$$
%C_2 = \left(2+\kappa C_1\left( 1+  2\CC\right)\right)\vee \left(\kappa + 2\left(1+ \kappa \CC\right)\Big(\frac{\alpha}{\beta} \Big)^{-\left\lfloor p/\ell\right\rfloor}\right).
%$$
\end{lem}
\begin{proof}
Fix $\mu_1, \mu_2 \in \mathcal{M}_+(V/\psi)$, $f\in \mathcal{B}(V/\psi)$ with $\Vert f \Vert_{\mathcal{B}(1+\kappa V_0)} \leq 1$ and an integer $n\geq 0$. Set for convenience
\begin{equation}
\label{eq:kmu}
m= \left\lfloor \frac{\log\left(\constante{\mu_1}+ \constante{\mu_2}  \right)}{ \RR \log\left(1/\mathfrak a \right)} \right\rfloor +1,\qquad
\nn=\RR \tau n, \quad \kk=\RR \tau m.
\end{equation}
%\cmb{ici les via la def de la norme, via la def de P... me paraissent bof}
By definition of the embedded propagator in \eqref{eq:def-aux}, we have
\begin{align*}
&\mu_1 M_{\nn}(f\psi)\, \mu_2M_{\nn} \psi-\mu_2M_{\nn} (f\psi) \, \mu_1 M_{\nn} \psi \\
&\qquad = \int_{\X^2}  M_{\nn} \psi (x)\,  \, M_{\nn} \psi(y)\,  \,  \left( \frac{ M_{\nn} (f\psi)(x)}{ M_{\nn} \psi(x)} - \frac{\ M_{\nn} (f\psi)(y)}{ M_{\nn} \psi(y)}  \right)\mu_1(dx)  \mu_2(dy)\nonumber\\
&\qquad \leq  \int_{\X^2}   M_{\nn} \psi (x)   \,  M_{\nn} \psi (y)    \, \left\|\delta_x P_{0, \nn}^{(\nn)} - \delta_y P_{0, \nn}^{(\nn)} \right\|_{\mathcal{M}(1+\kappa V_{0})} \mu_1(dx) \mu_2(dy).
\end{align*}
Using  Proposition~\ref{prop:ergo-auxi}, we get for $n\geq m$,
\begin{align}
&\mu_1 M_{\nn}(f\psi)\, \mu_2M_{\nn} \psi-\mu_2M_{\nn} (f\psi) \, \mu_1 M_{\nn} \psi \label{eq:PtoM} \\
%&\leq \frac{ }{\mu_1 M_{\nn} \psi . \mu_2 M_{\nn} \psi} \nonumber \\
&\qquad \leq \mathfrak y^{n-m} \int_{\X^2}  M_{\nn} \psi(x) \,  M_{\nn} \psi (y)   \, \left\|\delta_x P_{0, \kk}^{(\nn)} - \delta_y P_{0, \kk}^{(\nn)} \right\|_{\mathcal{M}(1+\kappa V_{(n-m)\RR})} \mu_1(dx)  \mu_2(dy).\nonumber
%\leq &\frac{ \mathfrak y^n}{\mu M_{n\tau} \psi \mu_2 M_{n\tau} \psi} \int\int  \delta_x M_{n\tau} \psi \delta_y M_{n\tau} \psi  \left( 1+ \beta V_n(x) +\beta V_n(y) \right)\mu(dx)  \mu_2(dy).
\end{align}
Using the definitions of the norm on $\mathcal{M}(1+\kappa V_{(n-m)\RR})$ and  of $V_{(n-m)\RR}$ in \eqref{eq:def-lyap}, we obtain 
\begin{align*}
\left\|\delta_x P_{0, \kk}^{(\nn)} - \delta_y P_{0, \kk}^{(\nn)} \right\|_{\mathcal{M}(1+\kappa V_{(n-m)\RR})} 
&\leq \int_{\mathcal{X}}(1+\kappa V_{(n-m)\RR}(z))\left|\delta_x P_{0, \kk}^{(\nn)}-\delta_y P_{0, \kk}^{(\nn)}\right|(dz)\\
&\leq 2+ \kappa \nu\left( \frac{M_{\nn-\kk}\psi}{\psi} \right) \left(\frac{M_{\kk} V(x)}{M_{\nn}\psi(x)} + \frac{M_{\kk} V(y)}{M_{\nn}\psi(y)} \right).
\end{align*}
Combining this inequality with \eqref{eq:PtoM} and $\autro \geq \mathfrak y$, we get
\begin{align}
 Q^{\mu_1}_{\nn}f-Q^{\mu_2}_{\nn}f 
& =\frac{ \mu_1 M_{\nn}(f\psi)\, \mu_2M_{\nn} \psi-\mu_2M_{\nn} (f\psi) \, \mu_1 M_{\nn}\psi}{\mu_1 M_{\nn} \psi . \mu_2 M_{\nn} \psi}\nonumber\\
& \leq \autro^{n}\bigg(2\autro^{-m} %\label{eq:PtoMrho} \\
 + \kappa  \autro^{-m} \nu\left( \frac{M_{\nn-\kk}\psi}{\psi} \right) \left(\frac{\mu_1 M_{\kk} V}{\mu_1 M_{\nn} \psi}   + \frac{\mu_2 M_{\kk} V}{\mu_2 M_{\nn} \psi}\right)\bigg).  \label{eq:PtoMmu}
\end{align}
We now bound each term of the  right-hand side.
First, using that $ \mathfrak a^{\RR}\leq \autro$ and  \eqref{eq:kmu}, 
\begin{equation}\label{eq:rhok-mu}
 \mathfrak a^{\RR} \leq \autro^{m}  \left(\constante{\mu_1}+ \constante{\mu_2} \right).
\end{equation}
Second, Lemma~\ref{lem:first-est} $ii)$ ensures that for $\mu\in\{\mu_1,\mu_2\}$,
\begin{equation}
\label{interr}
\constante{\mu M_{\kk}}   \leq \mathfrak a^{m\RR}\constante{\mu}  + \CC.
\end{equation}
Besides
 \eqref{eq:kmu}  also guarantees that for $\mu\in\{\mu_1,\mu_2\}$,
$\mathfrak a^{m\RR}\constante{\mu} \leq 1$.
It means that the positive measure $\mu M_{\kk}$ satisfies inequality \eqref{eq:mubound}, since $R\geq 1$. Then, Lemma~\ref{lem:upperbound} applied to $\mu M_{\kk}$ with $k=(n-m)\RR$ yields for all $n\geq m + p/\RR$,
$$
\nu\left( \frac{M_{\nn-\kk}\psi}{\psi} \right)\frac{\mu M_{\kk} V}{\mu M_{\nn} \psi} 
\leq C_1\frac{\mu M_{\kk} M_{\nn-\kk}\psi}{\mu M_{\kk}\psi} \frac{\mu M_{\kk} V}{\mu M_{\nn} \psi} 
\leq C_1 \constante{\mu M_{\kk} }.
$$
Finally, using again $\eqref{interr}$ and   \eqref{eq:rhok-mu}, we get
$$\nu\left( \frac{M_{\nn-\kk}\psi}{\psi} \right) \left(\frac{\mu_1 M_{\kk} V}{\mu_1 M_{\nn} \psi}   + \frac{\mu_2 M_{\kk} V}{\mu_2 M_{\nn} \psi}\right)
% \leq C_1\left(\constante{\mu_1 M_{\kk}} + \constante{\mu_2 M_{\kk}} \right)  
\leq   C_1(1+2\CC  \mathfrak a^{-\RR})  \autro^{m} \left(\constante{\mu_1} +\constante{\mu_2}
 \right).$$
Plugging 
the last inequality in \eqref{eq:PtoMmu}
% and using again \eqref{eq:rhok-mu}
ensures that for  all $n\geq m + p/\RR$,
$$Q^{\mu_1}_{\nn}f-Q^{\mu_2}_{\nn}f \leq(2\mathfrak a^{-\RR}+ \kappa C_1(1+2\CC  \mathfrak a^{-\RR}) )\left(\constante{\mu_1} +\constante{\mu_2} \right) \autro^{n}.$$
To conclude, it remains to show that \eqref{eq:contractionM} also holds for $n\leq m + p/\RR$. We have
\begin{align*}
\left\|Q^{\mu_1}_{\nn} - Q^{\mu_2}_{\nn}\right\|_{\mathcal{M}(1+\kappa V_0)}
&\leq \left\|Q^{\mu_1}_{\nn} \right\|_{\mathcal{M}(1+\kappa V_0)}+\left\| Q^{\mu_2}_{\nn}\right\|_{\mathcal{M}(1+\kappa V_0)}\leq  2+\kappa \constante{\mu_1 M_{\nn}} +\kappa \constante{\mu_2 M_{\nn}}.
\end{align*} 
Using again \eqref{interr},
$\constante{\mu M_{\nn} }  \leq \mathfrak a^{n\RR} \constante{\mu} + \Theta
\leq \autro^{n}  \constante{\mu}  + \Theta$   for $\mu\in\{\mu_1,\mu_2\},$
so that 
\begin{align*}
\left\|Q^{\mu_1}_{\nn} - Q^{\mu_2}_{\nn}\right\|_{\mathcal{M}(1+\kappa V_0)}
&\leq 2\left(1+\kappa \Theta\right)+\kappa \autro^{n}\left(\constante{\mu_1} +\constante{\mu_2} \right).
\end{align*}
Finally, $\autro \geq \mathfrak a^{\RR}$ and  $n\leq  m + p/\RR$ and  \eqref{eq:kmu} yield
$1 \leq \autro^{n}  \mathfrak a^{-(p +m \RR)} =  \autro^n \mathfrak a^{-(p+ \RR)}  \mathfrak a^{-(m-1) \RR} \leq \mathfrak  \autro^n \mathfrak a^{-(p+ \RR)} \left(\constante{\mu_1} +\constante{\mu_2} \right)$,
and we get
\begin{align*}
\left\|Q^{\mu_1}_{\nn} - Q^{\mu_2}_{\nn}\right\|_{\mathcal{M}(1+\kappa V_0)}
\leq  \autro^{n} \left( 2(1+\kappa \Theta)\mathfrak a^{-(p+ \RR)}  +\kappa\right)\left(\constante{\mu_1} +\constante{\mu_2} \right),
\end{align*}
for all $n\leq m + p/\RR$, which ends the proof.
\end{proof}

%We now extend the previous lemma to continuous time.

%\begin{lem}\label{lem:contractionM_continuous_time}
%For all $\mu_1,\mu_2 \in \M_+(V/\psi)$ and all $t\geq 0$,
%\begin{equation*}
%%\label{eq:contractionM_continuous_time}
%\left\|Q^{\mu_1}_t- Q^{\mu_2}_t \right\|_{\mathcal{M}(1+\kappa V_0)}\leq C'_2\autro^{\left\lfloor t/\RR\tau\right\rfloor} \left( \constante{\mu_1} +\constante{\mu_2} \right),
%\end{equation*}
%where $C'_2=C_0^2C_2$ and $C_0,C_2$ have been defined in \eqref{defC1} and \eqref{defC4} respectively.
%\end{lem}
%\begin{proof}
%Let $n =\lfloor t/\RR \tau \rfloor$ and $\delta= t - n\RR\tau \in [0,\RR \tau)$. According to Lemma~\ref{lem:contraction_M}, we have
%\begin{align*}
%\left\| Q^{\mu_1}_t- Q^{\mu_2}_t \right\|_{\mathcal{M}(1+\kappa V_0)}
%& = \left\|\frac{\mu_1 M_{\delta}M_{n\RR\tau}(f\psi)}{\mu_1 M_{\delta}M_{n\RR\tau} \psi}- \frac{\mu_2 M_{\delta}M_{n\RR\tau}(f\psi)}{\mu_2 M_{\delta}M_{n\RR\tau} \psi}\right\|_{\mathcal{M}(1+\kappa V_0)}\\
%&\leq C_2\autro^n \left( \frac{\mu_1 M_{\delta}( V)}{\mu_1 M_{\delta}(\psi)} + \frac{\mu_2M_{\delta}( V)}{\mu_2M_{\delta}(\psi)} \right).
%%& \leq C_3\autro^n \left( \frac{\mu_1 ( V)}{\mu_1(\psi)} + \frac{\mu_2( V)}{\mu_2(\psi)} \right),
%\end{align*}
%Using the definition \eqref{defC1} of $C_0$ for the last term ends the proof.
%%for some $C_3>0$ where the last inequality comes from Assumption \eqref{assu:condition_generales1}.
%\end{proof}

We have now all the ingredients to prove the existence of the eigenelements and the associated estimates. We start with the right eigenfunction.

\begin{proof}[Proof of Lemma~\ref{lem:h}]
We define
$$\eta(\cdot)=\nu( \cdot/\psi)\qquad\text{and,  for $k\geq 0$,}\quad h_k= \frac{ M_{k\tau} \psi}{ \nu\left(M_{k\tau}\psi/\psi \right)}= \frac{ M_{k\tau} \psi}{ \eta  M_{k\tau}\psi}.$$ 

\smallskip

\pg{The proof consists in proving that $(h_k)_{k\geq0}$ is a Cauchy sequence in $\mathcal{B}(V^2/\psi)$, and then check that the limit $h$ is an eigenfunction of $M_\tau$.
We also provide estimates on the profile of $h$.
Let $\mu \in \mathcal M_+(V)$. Using that
$$
\mu (h_{k+n})=\frac{\mu M_{(k+n)\tau} \psi/ \mu M_{k\tau} \psi}{ \eta M_{(k+n)\tau}\psi / \eta M_{k\tau}\psi } \mu (h_k)
$$ 
we have
\begin{align*}
| \mu (h_{k+n}) - \mu (h_k) |
&\leq  \left| \frac{\mu M_{(k+n)\tau} \psi}{ \mu M_{k\tau} \psi} -  \frac{\eta M_{(k+n)\tau}\psi}{ \eta M_{k\tau}\psi } \right| \frac{1}{\eta M_{(k+n)\tau}\psi /  \eta M_{k\tau}\psi} \mu (h_{k}) \\
&=  \left| Q_{k\tau}^{\mu} \left( \frac{M_{n\tau} \psi}{ \psi} \right) -  Q_{k\tau}^{\eta} \left( \frac{M_{n\tau} \psi}{ \psi} \right) \right| \frac{\mu M_{k\tau} \psi}{\eta M_{(k+n)\tau}\psi}.
\end{align*}
Then, Lemma~\ref{lem:contraction_M} yields 
\begin{align}
\label{ause}
| \mu (h_{k+n}) - \mu (h_{k})|
&\leq C_2 \autro^{\lfloor k/\RR \rfloor} \left( \constante{\mu M_{m\tau}} + \constante{\eta M_{m\tau}}   \right) \left\| \frac{M_{n\tau} \psi}{\psi} \right\|_{\mathcal B(1+\kappa V_0)} \frac{\mu M_{k\tau} \psi}{\eta M_{(k+n)\tau}\psi},
\end{align}
where $m = k-\lfloor k/\RR \rfloor\RR$.
To get that $(h_k)_{k\geq0}$ is a Cauchy sequence, it remains to bound the three last terms in the right hand side.
First, since $m\leq \RR$, we have, using \ref{A1} and \ref{A2} and recalling that $\psi\leq V$, $V\leq R\psi$ on $K$ and $\nu(K^c)=0$,
\begin{align}\label{eq:muM}
\constante{\mu M_{m\tau}}+\constante{\eta M_{m\tau}} \leq \frac{(\alpha + \theta)^\RR}{\beta^\RR}\left(\constante{\mu}+\nu\left(\frac V\psi\right)\right)\leq \frac{(\alpha + \theta)^\RR}{\beta^\RR}(1+R)\constante{\mu}.
\end{align}
Next, using the first part of Lemma~\ref{lem:approxvp} and \ref{A1}, we have that for any $k\geq 0$,
\begin{equation}
\label{lautreborne}
M_{(k+1)\tau} \psi(x)\leq d_1^{-1} (\alpha + \C)  ( \eta M_{k\tau} \psi) V(x) .
\end{equation}
Since $\left\| M_{k\tau} \psi/\psi \right\|_{\mathcal B (1+\kappa V_0)} 
=\sup_{\X} M_{k\tau} \psi /(\psi + \kappa V)$, we deduce that for any $k\geq 0$,
\begin{align}
\left\| \frac{M_{(k+1)\tau} \psi}{\psi} \right\|_{\mathcal B (1+\kappa V_0)} 
& \leq d_1^{-1} (\alpha + \C) (\eta M_{k\tau} \psi ) \sup_{ \X} \frac{ V}{\psi+\kappa V}   
\leq \frac{ (\alpha + \C)}{d_1\kappa} \eta  M_{k\tau} \psi. \label{lagrandeborne}
\end{align}
Besides, from Lemma~\ref{lem:first-est} {\it ii)} and the fact that $V\leq R\psi$ on $K$, $
\constante{\eta M_{k \tau}}  
 \leq \mathfrak a^{k}  \constante{\eta}  +\Theta\leq R+\Theta,$
so
 $\mu=\eta M_{k \tau}$ verifies \eqref{eq:mubound}. 
Lemma~\ref{lem:upperbound} applied to $\mu$ then gives for all $n\geq p$,
\begin{equation}\label{lanouvelle}
 (\eta M_{k  \tau}) (M_{n\tau} \psi) 
\geq {C_1}^{-1} (\eta M_{n \tau} \psi) ( \eta M_{k\tau} \psi).
\end{equation}
Finally, combining~\eqref{ause} with~\eqref{eq:muM}, \eqref{lautreborne}, \eqref{lagrandeborne}, \eqref{lanouvelle}, and using \ref{A2}, we get for all $n\geq p$
\begin{equation}\label{eq:convh}
|\mu (h_{k+n}) - \mu (h_{k})|\leq C_3\autro^{\lfloor k/\RR \rfloor}  \mu(V) \constante{\mu}
\end{equation}
where $C_3=C_1C_2\frac{(\alpha+\theta)^{\RR + 2}}{d_1^2\beta^{\RR + 2}\kappa}(1+R)$.
Taking $\mu=\delta_x$, we deduce that for all $n\geq p$
\[\| h_{k+n} - h_k \|_{\mathcal{B}(V^2/\psi)}\leq C_3 \autro^{\lfloor k/\RR \rfloor},\]
thus guaranteeing that $(h_k)_{k\geq0}$ is a Cauchy sequence in $\B(V^2/\psi)$.
We denote by $h$ the limit and, passing to the limit $n\to\infty$ in
\[\eta M_{\tau}\left( \frac{M_{n\tau}\psi}{\eta M_{n\tau}\psi} \right) \, \frac{ M_{(n+1)\tau}\psi(x)}{\eta M_{(n+1)\tau}\psi}= \delta_x M_{\tau}\left(\frac{M_{n\tau}\psi}{\eta M_{n\tau}\psi}\right),\]
which is possible due to the bound~\eqref{lautreborne}, we get that $(\eta M_{\tau} h) \cdot  h(x)= M_{\tau} h(x)$.
Setting $\lambda=\tau^{-1}\log(\eta M_\tau h)$, that is to say that $h$ is an eigenfunction of $M_\tau$ associated to the eigenvalue $\e^{\lambda\tau}$.
Letting $n\rightarrow \infty$ in \eqref{eq:convh} gives \eqref{eq:h_speed_conv}.

\medskip

We now give upper and lower bounds on $h$.
First, from Lemma~\ref{lem:approxvp},  $d_2\leq V_m=V/h_{m\tau}$, which yields by letting $m\rightarrow \infty$
\begin{equation}\label{eq:upperboundh}
h \leq d_2^{-1} V.
\end{equation}
For the lower bound, we start by using~\ref{A2} and~\ref{A3} to get that for any $k\leq n$
$$
 M_{n\tau} \psi 
\geq M_{k \tau}(\mathbf{1}_K M_{(n-k)\tau} \psi)= 
 M_{k \tau}(\mathbf{1}_K M_{\tau} ((M_{(n-k-1)\tau}\psi/\psi) \psi))
\geq 
c \beta (\eta M_{(n-k-1)\tau} \psi)M_{k \tau}(\mathbf{1}_K\psi).$$
Using first \eqref{lamin} and then Lemma~\ref{lem:first-est} {\it iii)}, we obtain
\[ h_n=\frac{M_{n\tau} \psi}{\eta M_{n\tau} \psi} 
\geq c \beta c_{k-1} \left( \beta^{k} \psi- \alpha^{k-1} \left(\alpha+\C \right)  V \right)  \frac{\eta M_{(n-k-1)\tau} \psi }{\eta M_{n\tau} \psi}
\geq \frac{c \beta c_{k-1}}{\boulet^{k+1}}\left( \beta^{k} \psi- \alpha^{k-1} \left(\alpha+\C \right)  V \right).
\]
Recalling that $c_k=c^k(\beta/\boulet)^{k-1}$ and $r=(\beta/\boulet)^2$, passing to the limit $n\to\infty$ yields that
\[h \geq c\sqrt{r}\left( cr\right)^{k-1} \left( \psi- \mathfrak a^{k-1}  (\alpha+\C)V/\beta  \right)\]
for any $k\geq0$.
Considering $ k=k(x)=\left\lfloor \log\left(\frac{\psi(x)}{V(x)}\frac{\beta}{2(\alpha+\C)}\right)/\log(\mathfrak a) \right\rfloor+2$, so that $\psi- \mathfrak a^{k-1} \frac{\alpha+\C}{\beta} V\geq\psi/2$, we get 
$$ h \geq c_1\left(\frac{\psi}{V} \right)^{q} \psi, \quad \text{with }
c_1=c\sqrt{r}\left( cr\right)^{1+\log\left(\frac{\beta}{2(\alpha+\C)}\right)/\log(\mathfrak a)}/2>0$$
 and
$q=\log(cr)/\log(\mathfrak{a})>0$, which ends the proof.
}
\end{proof}

\begin{Rq}
Notice that the eigenfunction $h$ built in this proof satisfies $\nu(h/\psi)=1$ and the constants in $(V/\psi)^q\psi\lesssim h\lesssim V$ depend on this normalization.
If we normalize $h$ such that $\|h\|_{\B(V)}=1$ as in Theorem~\ref{V-contraction} we get
$c_1d_2\left(\psi/V \right)^{q} \psi\leq h\leq V.$
\end{Rq}

We consider now  the left eigenelement. 
\begin{proof}[Proof of Lemma~\ref{lem:asymptotic-measure}]
Let us use again $\eta=\nu(\cdot / \psi )$. Applying  Lemma~\ref{lem:contraction_M} to $\mu_1=\eta$ and $\mu_2=\eta M_{n\tau}$ and  using again \eqref{eq:muM}, we get for $k,n\geq 0$, 
\begin{align*}
\big\|Q^{\eta}_{(k+n)\tau}  - Q^{\eta}_{k\tau}  \big\|_{\mathcal{M}(1+ \kappa V_0)}
\leq C'_2\autro^{\left\lfloor k/\RR\right\rfloor} \left( \nu \left(\frac{V}{\psi} \right)+\constante{\eta M_{n\tau}}   \right),
\end{align*}
where $C'_2=C_2\frac{(\alpha+\theta)^\RR}{\beta^\RR}$.
Then, using Lemma~\ref{lem:first-est}~{\it ii)}, $V\leq R\psi$ on $K$ and that $\nu(K)=1$, we have
\begin{align*}
\big\|Q^{\eta}_{(k+n)\tau}- Q^{\eta}_{k\tau} \big\|_{\mathcal{M}(1+ \kappa V_0)}
\leq C'_2\autro^{\left\lfloor k/\RR\right\rfloor} \left( R + \mathfrak a^n R +\CC \right).
\end{align*}
Therefore, the sequence of probabilities $(Q^\eta_{k\tau})_{k\geq0}$ satisfies the Cauchy criterion in $\mathcal{M}(1+\kappa V_0)$ and it then converges to a probability measure $\pi\in \mathcal{M}(1+\kappa V_0)$. Similarly, applying Lemma~\ref{lem:contraction_M} to $\mu_1=\mu$ and  $\mu_2=\eta M_{n\tau}$, we also have 
$$\big\|Q^{\eta}_{(k+n)\tau}  - Q^{\mu}_{k\tau}  \big\|_{\mathcal{M}(1+ \kappa V_0)}
\leq C'_2 \autro^{\left\lfloor k/\RR\right\rfloor} \left( \constante{\mu}  + \mathfrak a^{n}R +\CC \right)$$
for any $\mu \in \mathcal{M}(1+\kappa V_0)$.
Letting $n$ tend to infinity yields
\begin{align}
\label{eq:cvgammab}
\big\| \pi- Q_{k\tau}^{\mu} \big\|_{\mathcal{M}(1+\kappa V_0)}\leq C'_2\left( \constante{\mu}+ \Theta \right)\autro^{\left\lfloor k/\RR\right\rfloor} .
\end{align}
 %Trough the semigroup property \cmb{to check}, we have $Q^{\pi}_t=\pi$.
Besides, $\pi(h/\psi)\leq \pi(V/\psi)= \pi(V_0)<+\infty$ and we can  define $\gamma \in \mathcal{M}(\psi + \kappa V)$ by 
$$
\gamma(f)= \frac{\pi(f/\psi)}{\pi(h/\psi)} %= \pi(f/\psi)%}{\pi(\mathbf{1}_{\psi>0}/\psi)} ,
$$
for $f \in \mathcal{B}(\psi + \kappa V) = \mathcal{B}(V)$. Observe that $\gamma(h)=1$. Next,
%$$
%Q^{\mu}_{t+s} (f\mathbf{1}_{\psi>0}/\psi)= \frac{\mu M_{t+s} (f\mathbf{1}_{\psi>0})}{\mu M_{t+s}\psi}= \frac{\mu M_{s} (M_t (f\mathbf{1}_{\psi>0}))}{\mu M_{s}\psi} \frac{\mu M_{s} (\psi)}{\mu M_{s}(M_t \psi)}
%$$
\begin{align}\label{eq:3}
Q^{\eta}_{(k+1)\tau} (f/\psi)= Q^{\eta}_{k\tau} (M_{\tau} f/\psi) \frac{\eta M_{k\tau}\psi }{\eta M_{(k+1)\tau}\psi}. %\frac{\mu M_{n\tau} (M_{k\tau} (f\mathbf{1}_{\psi>0}))}{\mu M_{n\tau}\psi} \frac{\mu M_{n\tau} (\psi)}{\mu M_{n\tau}(M_{k\tau} \psi)}.
\end{align}
Applying \eqref{eq:h_speed_conv} to $\mu=\eta M_{\tau}$,
% and $\mu=\eta$, 
$$
\frac{\eta M_{k\tau}\psi }{\eta M_{(k+1)\tau}\psi}\xrightarrow[k\rightarrow +\infty]{}\e^{-\lambda \tau}.
$$
Then, letting $k\rightarrow \infty$ in \eqref{eq:3}, we obtain
%Letting $s\rightarrow \infty$ and 
%choosing $\mu=\eta M_s$ in  \eqref{eq:h_speed_conv}
%to check that 
%$$\frac{\eta M_{s}\psi}{\eta M_{t+s}\psi} \rightarrow \e^{-\lambda t},$$
%we obtain
$\pi (f/\psi)= \pi (M_{\tau} f/\psi) \e^{-\lambda \tau},$
which ensures that $\gamma$ is an eigenvector for $M_\tau$.
Adding that $\pi(f)=\gamma(f\psi)/\gamma(\psi)$ since $\pi$ is a probability measure, \eqref{eq:cvgamma} follows
from \eqref{eq:cvgammab}.
\end{proof}

\pg{Now, we are in position to prove Proposition~\ref{th:main_discret}.}
 
\begin{proof}[Proof of Proposition~\ref{th:main_discret}]
Using that
$$\left\|\pi- Q^{\mu}_{k\tau} \right\|_{\mathcal{M}(1+\kappa V_0)}=\sup_{f \in \mathcal{B}(1 + \kappa V_0)} \left| \frac{\gamma(f \psi)}{\gamma(\psi)} - \frac{\mu M_{k\tau} (f\psi)}{  \mu M_{k\tau} \psi} \right|=\left\| \frac{\gamma}{\gamma(\psi)} - \frac{\mu M_{k\tau}}{  \mu M_{k\tau} \psi} \right\|_{\mathcal{M}(\psi + \kappa V)}$$
and multiplying  \eqref{eq:cvgammab} by $ \mu M_{k\tau} \psi$, we get
\begin{align}\label{eq:maj1}
\left\|  \mu M_{k\tau} \psi \frac{\gamma}{\gamma(\psi)} - \mu M_{k\tau} \right\|_{\mathcal{M}(\psi + \kappa V)}
&\leq C'_2\autro^{\left\lfloor k/\RR\right\rfloor} \left(\constante{\mu}+\CC \right) \mu M_{k\tau} \psi.
\end{align}
Moreover, $h\in\mathcal{M}(\psi+\kappa V)$ since $h \lesssim V$. As $\gamma(h)=1$, the previous inequality applied to the eigenfunction $h$ yields
% by Lemma~\ref{lem:asymptotic-measure}
\begin{align*}
\left| \frac{\mu M_{k\tau} \psi}{\gamma(\psi)} - \mu(h) \e^{\lambda k\tau} \right|
&
\leq C'_2\autro^{\left\lfloor k/\RR\right\rfloor} \left(\constante{\mu}+\CC \right)\mu M_{k\tau} \psi.
\end{align*}
Then, recalling that $\psi\leq V$, we have
\begin{align}\label{eq:maj2}
\left\|  \mu M_{k\tau} \psi \frac{\gamma}{\gamma(\psi)} - \e^{\lambda k\tau} \mu(h)  \gamma \right\|_{\mathcal{M}(\psi + \kappa V)}
&= \left| \frac{\mu M_{k\tau} \psi}{\gamma(\psi)} -  \e^{\lambda k\tau} \mu(h) \right| \left\| \gamma \right\|_{\mathcal{M}(\psi + \kappa V)}\\\nonumber
%&\leq (1+\kappa)\gamma(V) \left| \frac{\mu M_t \psi}{\gamma(\psi)} -  \e^{\lambda t} \mu(h) \right| \\
&\leq  C'_2\autro^{\left\lfloor k/\RR\right\rfloor} \left(\constante{\mu}+\CC \right) \mu M_{k\tau} \psi \,  \times \, (1+\kappa) \gamma(V). 
\end{align}
Combining \eqref{eq:maj1} and \eqref{eq:maj2}, by triangular inequality, we get
\begin{equation}\label{eq:firstpart}
\left\| \mu M_{k\tau} - \gamma \e^{\lambda k\tau} \mu(h)  \right\|_{\mathcal{M}(\psi + \kappa V)} \leq C'_2\autro^{\left\lfloor k/\RR\right\rfloor} \left(\constante{\mu}+\CC \right) \mu M_{k\tau} \psi \left(1 +  (1+\kappa) \gamma(V)\right).
\end{equation}
This gives the first part of  \eqref{eq:main_result1_discret}. Finally, by integration of \eqref{lautreborne},
\begin{equation}\label{eq:muMk}
 \mu M_{k\tau} \psi \leq d_1^{-1} (\alpha + \C)    \nu(M_{k\tau} \psi/\psi) \, \mu(V).
\end{equation}
Adding that $\constante{\gamma}\leq \Theta$ according to Lemma~\ref{lem:first-est} {\it ii)} and $\gamma M_{k\tau} = \e^{\lambda k\tau}\gamma $ from Lemma~\ref{lem:asymptotic-measure},  Lemma~\ref{lem:upperbound} applied to  $\mu = \gamma$ yields $ \nu(M_{k\tau} \psi/\psi)\leq C_1 \e^{\lambda k\tau}$
for $k\geq p$ and we obtain
$ \mu M_{k\tau} \psi \leq C_1 d_1^{-1} (\alpha + \C)    \mu(V) \e^{\lambda k\tau }.$
It proves  \eqref{eq:main_result1_discret}  for 
 $k\geq p $ 
 with 
 $$C=C'_2 (1+\Theta) \left(1 +  (1+\kappa) \gamma(V)\right) \max( 1, C_1 d_1^{-1} (\alpha + \C)).$$
The fact that  \eqref{eq:main_result1_discret} holds  for some constant $C$ also for $k\leq p$ comes directly from \eqref{eq:firstpart}, \eqref{eq:muMk}, Lemma~\ref{lem:first-est} {\it iii)} and $\nu(K)=1$.
% The fact that  \eqref{eq:main_result1} holds  for some constant $C$ also for $k\leq p$ is a consequence of~\eqref{as:psi} and the local boundedness of $t\mapsto\|M_tV\|_{\B(V)}$.
\end{proof}

Finally, we prove Corollary~\ref{lesnormandsmecassentlescouillesavecleurslabels}.
\begin{proof}[Proof of Corollary~\ref{lesnormandsmecassentlescouillesavecleurslabels}]
For convenience and without loss of generality, we assume that $V\geq \mathbf 1$. Then, $\gamma( \mathbf 1)<\infty$. Next, if
$\gamma( \mathbf 1)=0$, then $\gamma(\X)=0$. In this case, $ \gamma=0$,  which is absurd because $\gamma(\psi)>0$ and $\psi>0$. Therefore, $\gamma( \mathbf 1)>0$. 

We set $\pi(\cdot)= \gamma(\cdot)/\gamma(\mathbf{1})$ and we have by triangular inequality
\begin{align*}
\left\| \frac{\mu M_t}{\mu M_t \mathbf 1}  -  \pi \right\|_{\textrm{TV}} 
&\leq \frac{\e^{\lambda t}}{\mu M_t \mathbf{1}} \left\|  \e^{-\lambda t} \mu M_t - \pi \e^{-\lambda t}  \mu M_t \mathbf{1} \right\|_{\mathcal{M}(V)} \\
  &\leq \frac{\e^{\lambda t}}{\mu M_t \mathbf{1}} \left( \left\|  \e^{-\lambda t} \mu M_t - \gamma  \mu(h) \right\|_{\mathcal{M}(V)} 
 + \left|  \gamma(\mathbf{1})  \mu(h) -  \e^{-\lambda t}  \mu M_t \mathbf{1} \right| \pi(V) \right).
 \end{align*}
 From \eqref{eq:conv_norm}, we have  
 $  \left\|  \e^{-\lambda t} \mu M_t - \gamma  \mu(h) \right\|_{\mathcal{M}(V)} \leq  C \mu(V) \e^{-\omega t}.
 $
Using this estimate with $V\geq \mathbf 1$, we also have
\begin{align}
\label{eq:cvMt1}
\left|  \gamma(\mathbf{1})  \mu(h) -  \e^{-\lambda t}  \mu M_t \mathbf{1} \right| 
&\leq C \mu(V) \e^{-\omega t}.
\end{align}
Combining the three last estimates yields
\begin{align}
\label{compp}
\left\|  \frac{\mu M_t}{\mu M_t \mathbf{1}} - \pi \right\|_{\textrm{TV}} \leq C \frac{\e^{\lambda t}}{\mu M_t \mathbf{1}}  \mu(V) \e^{-\omega t}(1+\pi(V) ).
\end{align}
Now on the first hand, Equation~\eqref{eq:cvMt1} also gives
$
\e^{- \lambda t}\mu M_t \mathbf{1} \geq \gamma(\mathbf 1) \mu(h) - C \mu(V) \e^{-\omega t},
$
and for any  $t\geq t(\mu)=\frac{1}{\omega}\log\big(\frac{2C}{\gamma(\mathbf{1})}\frac{\mu(V)}{\mu(h)}\big)$, we have
\begin{equation}
 \label{eq:lwbdMt1}
\e^{- \lambda t}\mu M_t \mathbf{1} \geq \mu(h)\gamma(\mathbf{1})/2.
\end{equation}
Plugging \eqref{eq:lwbdMt1} in \eqref{compp} yields \eqref{eq:cor_result1} when $t\geq t(\mu)$. Otherwise,
$$
\left\| \frac{\mu M_t}{\mu M_t \mathbf 1}  -  \pi \right\|_{\textrm{TV}}  \leq 2 \leq 2\,\e^{-\omega t} \e^{ \omega t(\mu)} \leq C \frac{\mu(V)}{\mu(h)} \e^{-\omega t},
$$ 
which ends the proof.
\end{proof}

\subsection{Proofs of Section~\ref{mainetpastweedie}}% Theorem~\ref{V-contraction}}
\label{proofthun}

\pg{Before proving Theorem~\ref{V-contraction}, we start by checking that $h$ and $\gamma$ are the unique (up to normalization) positive eigenelements of $M_t$ for any time $t\geq0$.

\begin{lem}\label{lem:eigenelements}
For all $t\in[0,\infty)$ we have $M_th=\e^{\lambda t}h$ and $\gamma M_t=\e^{\lambda t}\gamma$.
Besides, they are the unique such elements in $\B_+(V)$ and $\M_+(V)$ with the normalization $\gamma(h)=\|h\|_{\B(V)}=1$.
\end{lem}

\begin{proof}
From Lemma~\ref{lem:h} and the semigroup property we have that $M_{k\tau} h=\e^{\lambda k\tau}h$ for any integer $k\geq0$.
Then, the convergence result in Proposition~\ref{th:main_discret} applied to $\mu=\delta_x$ implies that any other such function $\tilde h$ must satisfy $\tilde h=\gamma(\tilde h)h$, hence the uniqueness.

Next, by semigroup property, we have $M_\tau M_th=\e^{\lambda \tau}M_th$ for any $t\geq0$.
By uniqueness, there exists $(c_t)_{t\geq 0}$ such that for all $t\geq 0$, $M_th = c_th$. Moreover, still by semigroup property, $c_{t+s} = c_tc_s$ for all $t,s\geq 0$.
Therefore, since $t\mapsto c_t$ is locally bounded due to~\eqref{eq:upperboundh} and the local boundedness of $t\mapsto \left\Vert M_tV\right\Vert_{\mathcal{B}(V)}$, there exists $\tilde{\lambda}\in\mathbb{R}$ such that $c_t = \e^{\tilde{\lambda}t}$. Since $c_\tau = e^{\lambda \tau}$, we finally get $\tilde{\lambda} = \lambda$.

The proof for $\gamma$ follows exactly the same scheme, and we do not repeat it.
\end{proof}

Now we are in position to prove Theorem~\ref{V-contraction}.}

\begin{proof}[Proof of Theorem~\ref{V-contraction} i)]
We assume that Assumption~{\bf A} is satisfied by $(V,\psi)$ for a set $K$, constants $\alpha,\beta,\theta,c,d$ and a probability $\nu$.
Then, from Lemmas~\ref{lem:h}, \ref{lem:asymptotic-measure} and~\ref{lem:eigenelements},  there exist
eigenelements $(\gamma,h, \lambda)$ such that
 $$\beta\leq\e^{\lambda\tau}\leq\alpha+\theta, \qquad c_1d_2(\psi/V)^q \psi \leq h \leq V.$$ 
We check now that $(V,h)$  satisfies also Assumption~{\bf A} with the same set $K$ and constant $\alpha$
 as 
$(V,\psi)$ but other constants $\beta',\theta',c',d'$ and an other probability measure $\nu'$. \\
First the inequality $V/h\leq\frac{1}{c_1d_2}(V/\psi)^{q+1}$ ensures that $\sup_KV/h<\infty$ since $\sup_KV/\psi<\infty$.
Moreover $h$ satisfies \ref{A2} with $\beta'=\exp(\lambda \tau)\geq \beta>\alpha$. 
Recalling that $R=\sup_K V/\psi <\infty$, we also have
\[\sup_K \frac{\psi}{h}\leq\frac{R^q}{c_1d_2}.\]
Adding that   $(V,\psi)$ satisfies \ref{A1}  with constants $\alpha,\theta$ yields
$ M_{\tau}  V \leq \alpha   V + \theta' \mathbf{1}_K h$  with $\theta'=\C \frac{R^q}{c_1d_2} ,$
which gives \ref{A1}  for $(V,h)$. 
We use now  \ref{A3} and \ref{A2} for $(V,\psi)$ and get for $x\in K$
\begin{align*}
\delta_x M_{\tau}(f h) & \geq c \nu\left( \frac{f h}{\psi} \right) \delta_x M_{\tau} \psi \geq  c \beta \nu\left( \frac{f h}{\psi} \right) \psi(x).
\end{align*}
Using again that $M_{\tau}h=\e^{\lambda \tau}h$, we obtain for $x\in K$,
$
\delta_x M_{\tau}(f h) \geq c'\nu'(f)\delta_x M_{\tau}h,
$
with 
$$
\nu'= \frac{ \nu\big( \cdot \, \frac{ h}{\psi} \big)}{  \nu\big( \frac{h}{\psi}\big)},\quad  c'= c \beta \e^{-\lambda \tau}    \nu\left( \frac{h}{\psi}\right)  \inf_K \frac{\psi}{h} \geq\frac{\beta}{\alpha+\theta}\frac{c_1d_2}{R^{q+1}} >0.
$$
Finally, \ref{A4} is satisfied since
\begin{equation*}
\sup_{x\in K}\frac{M_{n\tau}h(x)}{h(x)} =\e^{\lambda n \tau} = \nu'\left(\frac{ M_{n\tau}h}{h}\right).
\end{equation*}

Therefore, every result stated above holds replacing $\psi$ by $h$ and the constants $\beta, \theta, c,d$ of Assumption~{\bf A} by $\beta', \theta', c',d'$ defined above. In particular, for all $\mu\in \mathcal{M}_+(V)$, \eqref{eq:cvgamma} becomes
\begin{align*}
\left\Vert \gamma(\cdot \, h) - \frac{\mu M_{k\tau}(\cdot \, h)}{\e^{\lambda k\tau}\mu(h)}\right\Vert_{\mathcal M (1+\kappa V/h)}\leq C\left(\frac{\mu(V)}{\mu(h)}+ \CC' \right)\autro^{\left\lfloor k/\RR\right\rfloor},
\end{align*}
where $\CC'=\C'/(\beta'-\alpha)$ and $$C = \max\left\{2  \mathfrak a^{-\RR} + \kappa C_1 \left(1 + 2 \Theta \mathfrak a^{-\RR}\right), 2\left(1+\kappa \Theta\right)\mathfrak a^{-(p+ \RR)}  + \kappa \right\}
\max\left(1, \frac{(\alpha + \theta')^\RR}{\beta'^\RR}\right)$$
and $\rho$ is defined in \eqref{eq:autro}, $\mathfrak{p}$ in \eqref{Rrrr}, $\mathfrak{a}$ in \eqref{eq:fraka-frakc},  $\kappa$ in Proposition \ref{prop:ergo-auxi}, $\Theta$ in \eqref{eq:thetaR}, and $p,C_1$ in \eqref{defp} (replacing $\beta, \theta, c, d$ by $\beta',\theta', c',d'$ in each definition).
Using that 
$$\left\Vert \gamma( \cdot \, h) - \frac{\mu M_{k\tau}(\cdot \, h)}{\e^{\lambda k\tau}\mu(h)}\right\Vert_{\mathcal M (1+\kappa V/h)}=\left\Vert \gamma - \frac{\mu M_{k\tau}}{\e^{\lambda k\tau}\mu(h)}\right\Vert_{\mathcal M (h+\kappa V)}\geq \kappa \left\Vert \gamma - \frac{\mu M_{k\tau}}{\e^{\lambda k\tau}\mu(h)}\right\Vert_{\mathcal M (V)},
$$
multiplying by $\e^{\lambda k\tau}\mu(h)$ and using that $h\leq V$ , we get
$$
\left\Vert \e^{\lambda k\tau}\mu(h)\gamma - \mu M_{k\tau}\right\Vert_{\mathcal M ( V)} \leq \frac{C}{\kappa\rho}\mu(V)\left(1+ \CC'\right)\e^{-\omega k\tau}\e^{\lambda k\tau},$$
where $\omega = -\log\rho / \mathfrak{p}\tau$.
Let $t\geq 0$ and $(k,\varepsilon)\in\mathbb{N}\times [0,\tau)$ be such that $t = k\tau + \varepsilon$.
Applying the above inequality to the measure $\mu M_\varepsilon$ in place of $\mu$ yields
\begin{equation*}
\left\Vert e^{\lambda t}\mu(h)\gamma - \mu M_{t}\right\Vert_{\mathcal M ( V)}  \leq\frac{C}{\kappa\rho} \mu (M_\varepsilon V)\left(1+ \CC' \right)\e^{-\omega k\tau} \e^{\lambda k\tau}\leq C' \mu(V) \e^{-\omega t},
\end{equation*}
where $C' = \frac{C}{\kappa\rho} \big(1+ \CC' \big) \e^{\omega\tau}\e^{|\lambda |\tau}\sup_{s\leq \tau}\left\|\frac{M_s V}{V} \right\|_\infty$.
Adding that uniqueness is a direct consequence of $\omega>0$ ends the proof of  Theorem~\ref{V-contraction}~{\it i)}.
\end{proof}

\begin{proof}[Proof of Theorem~\ref{V-contraction} ii)]
The proof follows the usual equivalence in Harris Theorem \cite[Chapter 15]{DMPS}. It  also  used the necessary condition
of small set  obtained in  \cite[Theorem 2.1]{CV14}.
Assume that there exist a positive measurable function $V,$ a triplet $(\gamma,h,\lambda)\in\M_+( V)\times \B_+( V)\times \mathbb{R},$ and constants $C,\omega>0$ such that
\eqref{vecteursales} and \eqref{eq:conv_norm} hold.
Without loss of generality we can suppose that $\|h\|_{\B(V)}=\gamma(h)=1$. It remains to check that $(V,h)$ satisfies Assumption~{\bf A}.

\

Fix $R > \gamma(V)$ and $\tau>0$ such that
\begin{equation}
\label{eq:taufixe}
\e^{-\omega \tau/2} C \left(R +\gamma(V)\right) < 1-\frac{\gamma(V)}{R}.
\end{equation}
It ensures that
$$\alpha:= \e^{\lambda\tau} \left( C \e^{- \omega \tau} + \frac{\gamma(V)}{R} \right) \, < \, \beta:= \e^{\lambda \tau}.$$
Using~\eqref{vecteursales}, we obtain that \ref{A2} and  \ref{A4}  are satisfied by $h$ with $d=1$  and  any probability measure $\nu$, which ends the proof.

\

By \eqref{eq:conv_norm}, we have for all $x\in \X$,
$
\e^{-\lambda t} M_t V(x) - h(x)\gamma(V) \leq C V (x) \e^{-\omega t}.
$
We define $K= \{x\in \X, V(x) \leq R h(x)\},$ which is not empty since $\|h\|_{\B(V)}=1$ and $R>\gamma(V)\geq\gamma(h)=1.$
Writing $ \theta = \gamma(V) \e^{\lambda \tau}$ and using $h(x)=\mathbf{1}_{K^c}   (h(x)/V(x)) V(x)+\mathbf{1}_K   h(x)$, we get 
$$
M_{\tau} V(x) \leq \alpha V(x) + \mathbf{1}_K  \theta h(x)
$$
for all $x\in\X$.
Therefore, \ref{A1} holds for $(V,h)$ and it remains to prove \ref{A3}.

\

We define the probability measure 
$\pi:=\gamma( \cdot\,h)$
and we use the Hahn-Jordan decomposition of the following family of  signed measure indexed by $x\in \mathcal{X}$, 
$$\nu^x= \frac{\delta_{x} M_{\tau/2}(\cdot\,h)}{\e^{\lambda \tau/2} h(x)}-\pi =\nu^x_+-\nu^x_-.$$
As $h \leq  V$, Equation~\eqref{eq:conv_norm} with $t=\tau/2$ and $\mu=\delta_x$ yields
\begin{equation}\label{eq:bornenux}\nu^x_\pm (\mathbf{1})\leq \nu^x_\pm (V/h) \leq\left\|\nu^x\right\|_{\M(V/h)} =\left\|\frac{\delta_{x} M_{\tau/2}}{\e^{\lambda \tau/2} h(x)}-\gamma \right\|_{\M(V)} \leq C\frac{ V(x)}{h(x)}\e^{-\omega\tau/2}.\end{equation}
For every $f\in \B_+(V/h)$ and $x\in\X$ we have
\begin{align}\label{eq:5}
\frac{\delta_{x} M_{\tau} ( hf)}{\e^{\lambda \tau} h(x)}
&= \frac{\delta_{x} M_{\tau/2}}{ \e^{\lambda \tau/2} h(x)}\left( \frac{  M_{\tau/2}(hf)}{ \e^{\lambda \tau/2}h}h\right) 
\geq  (\pi-  \nu^x_-) \left( \frac{  M_{\tau/2}(hf)}{ \e^{\lambda \tau/2}h }\right).
\end{align}
Next,
\begin{align}\label{eq:6}
\pi \left( \frac{  M_{\tau/2}(hf)}{ \e^{\lambda \tau/2}h}\right) =\frac{\gamma M_{\tau/2}(hf)}{ \e^{\lambda \tau/2}}=\frac{\e^{\lambda \tau/2}\gamma(hf)}{\e^{\lambda \tau/2}}=\pi(f)
\end{align}
and writing $(\nu^y (f))_+$ the positive part of the real number $\nu^y (f)$,
\begin{align}\label{eq:7}
\nu^x_- \left( \frac{  M_{\tau/2}(hf)}{ \e^{\lambda \tau/2}h }\right) =  \int_{\mathcal X}  \frac{ \delta_y M_{\tau/2}(hf) }{\e^{\lambda \tau/2} h(y)} \nu^x_-(dy)
\leq  \int_{\mathcal X} \left( \pi (f) +(\nu^y (f))_+  \right) \nu^x_-(dy).
\end{align}
 Combining \eqref{eq:5} with \eqref{eq:6} and \eqref{eq:7}, we get 
\begin{align*}
\frac{\delta_{x} M_{\tau} ( hf)}{\e^{\lambda \tau} h(x)} &\geq \pi(f) (1 - \nu^x_-(\1)) - \int_\X (\nu^y(f))_+   \nu^x_-(dy). 
\end{align*}
The  minimality property of the Hahn-Jordan decomposition entails that $\nu^x_- \leq \pi$ and~\eqref{eq:bornenux} yields t $\nu^x_-(\1)\leq C R\, \e^{-\omega \tau/2}$ for $x\in K.$
We deduce that for all $x\in K$,
\[ \frac{\delta_{x} M_{\tau} (h f)}{\e^{\lambda \tau} h(x)}
\geq \pi(f) (1 -  C \ro\, \e^{-\omega \tau/2}) - \int (\nu^y (f))_+   \pi (dy)=:\eta(f).\]
We consider  now the infimum measure $\nu_0$
of the left hand side. More precisely, using \cite[Lemma 5.2]{CV14} we can define  the  mesure $\nu_0$      by
$$\nu_0(A)=\inf\left\{ \sum_{i=1}^n \frac{\delta_{x_i} M_{\tau} ( h1_{B_i})}{\e^{\lambda \tau} h(x_i)} \,  : \, n\geq 1;  \ (x_1, \ldots,x_n)\in \X^n ; \ (B_1, \ldots, B_n) \in  \mathfrak P(A)\right\},$$
 where $A$ is a measurable subset of $\X$ and  $\mathfrak P(A)$ is the set of finite partitions of $A$ formed by measurable sets of $\X$. Besides for any measurable partition $B_1, \ldots, B_n$ of $\X$,
we have $\sum_{i=1}^n {(\nu^y(\mathbf{1}_{B_i}))}_+\leq \nu^y_+(\mathcal X)\leq C  \e^{-\omega \tau/2} V(y)/h(y)$ from  \eqref{eq:bornenux}. Then
%~\eqref{eq:bornenux},
$$
\sum_{i=1}^n \eta(\mathbf{1}_{B_i})
\geq 1- CR\,\e^{- \omega \tau/2} - \int_\X  \sum_{i=1}^n (\nu^y(\mathbf{1}_{B_i}))_+ \pi(dy)\geq c,$$
where $c %=1- CR\,\e^{- \omega \tau/2} - C \pi(V/\psi) \e^{-\omega \tau/2}
= 1- C\e^{- \omega \tau/2}(R+\gamma(V))>0$ using \eqref{eq:taufixe}.
Adding that
\[\eta(\mathbf{1}_{K^c}) \leq \eta \left(\frac{V}{R \psi}\right)\leq \pi\left(\frac{V}{R \psi}\right)= \frac{\gamma(V)}{R},\]
we obtain that $\nu_0(\X)\geq c$ and $\nu_0(K^c)<c$. As a conclusion, $\nu_0$ is a positive measure such that $\nu_0(K)>0$ and for any $x\in \X$ and  $f\in \B_+(V/h)$,
$$
\frac{\delta_{x} M_{\tau} ( hf)}{\e^{\lambda \tau} h(x)} \geq \nu_0(f)
$$
which yields \ref{A3} and ends the proof.
\end{proof}

We end this section by proving the sufficient conditions of Section \ref{ssec:generator}.
%Propositions~\ref{prop:lyapun} and~\ref{prop:irr}.

%\subsection{Drift condition and irreducibility: proofs of Propositions~\ref{prop:lyapun} and~\ref{prop:irr}}
%\label{ssec:proofdrift}

\begin{proof}[Proof of Proposition~\ref{prop:lyapun}]
Let $C>0$ be such that
$C^{-1}\psi\leq\varphi\leq C\psi.$
By assumptions $\L\psi\geq b\psi$ and $\L\varphi\leq \xi\varphi$ so that we have for all $t\geq0$
\[M_t\psi\geq\psi+b\int_0^tM_s\psi\,ds\qquad\text{and}\qquad M_t\varphi\leq\varphi+\xi\int_0^tM_s\varphi\,ds,\]
which yield by Gr\"onwall's lemma \pg{(see Remark~\ref{rk:Gronwall} below for the cases when $b$ or $\xi$ are not positive)}
$M_t\psi\geq\e^{bt}\psi$ and $M_t\psi\leq CM_t\varphi\leq C\e^{\xi t}\varphi\leq C^2\e^{\xi t}\psi.$
Similarly, setting $\phi=V-\frac{\zeta}{b-a}\psi$, we have
$\L\phi\leq a V+\zeta\psi-\frac {b\,\zeta }{b-a}\psi=a\phi,$
so by Gr\"onwall's lemma $M_t\phi\leq\e^{at}\phi$ and
\[M_t V\leq\e^{at}\Big(V-\frac{\zeta}{b-a}\psi\Big)+\frac{\zeta}{b-a}M_t\psi\leq\e^{at}V+\frac{C^2\zeta}{b-a}\e^{\xi t}\psi.\]
Consider now $\tau,R >0,$ and
\[\alpha=\e^{a\tau}+\frac{C^2\zeta}{(b-a)R}\e^{\xi \tau},\qquad\beta=\e^{b\tau},\qquad\theta=\frac{C^2\zeta}{b-a}\e^{\xi \tau}, \qquad R_0=\frac{\theta}{\e^{b\tau}-\e^{a\tau}}\]
positive constants.
Defining $K=\{x\in\X,\ V(x)\leq R\,\psi(x)\}$,  we get $\psi\leq V/R$ on $K^c$ and
$M_\tau V\leq \alpha V+\theta \1_K\psi$.
Besides for $R>R_0$, $\alpha<\beta$.
So Assumptions~\ref{A1}-\ref{A2} are verified for $K=\{V\leq R\,\psi\},$ $R>R_0,$ and constants $\alpha,\beta, \theta$ defined just above.
\end{proof}

\pg{\begin{Rq}\label{rk:Gronwall}
Integrating $M_{t-s}\psi-\psi\geq b\int_0^{t-s}M_\tau\psi\,d\tau$ against $\delta_xM_s$ we get, due to the semigroup property and Tonelli's theorem, that for all $t\geq s\geq0$, $M_t\psi-M_s\psi\geq b\int_s^t M_\tau\psi\,d\tau$.
Now for any $\alpha\in\R$ we have through integration by parts
\[b\!\int_0^t\! \e^{\alpha s}M_s\psi\,ds=b\!\int_0^t\! M_s\psi\,ds+b\,\alpha\!\int_0^t\!\e^{\alpha s}\!\int_s^t\! M_\tau\psi\,d\tau\,ds\leq b\!\int_0^t\! M_s\psi\,ds+\alpha\!\int_0^t\!\e^{\alpha s}(M_t\psi-M_s\psi)\,ds\]
which gives
\[(\alpha+b)\int_0^t \e^{\alpha s}M_s\psi\,ds\leq b\int_0^t M_s\psi\,ds+(\e^{\alpha t}-1)M_t\psi\leq \e^{\alpha t}M_t\psi-\psi.\]
So when choosing $\alpha$ such that $\alpha+b>0$ we can use Gr\"onwall's lemma to get $\e^{\alpha t}M_t\psi\geq \e^{(\alpha+b)t}\psi$, which is to say $M_t\psi\geq\e^{bt}\psi$.
\end{Rq}}

\begin{proof}[Proof of Proposition~\ref{prop:irr}]
Let $\psi:\X\to(0,\infty)$ and define $\nu= (\# K)^{-1}\sum_{x\in K}\delta_{x}$ the uniform measure on $K,$ where $\# K$ stands for the cardinal of $K.$ We have for all $f\geq0$ and $x,y\in K$,
\[\delta_xM_\tau(f\psi)\geq \delta_x M_\tau(\{y\}) f(y)\psi(y)\geq c  f(y) M_\tau \psi(x).
\]
where 
$$c=\min_{x,y \in K} \frac{\psi(y) \delta_{x} M_\tau(\{y\})}{M_\tau \psi(x)}>0$$
using the irreducibility condition $\delta_{x} M_\tau(\{y\})>0.$
Integrating with respect to $\nu$ shows that~\ref{A3} holds  and 
Assumption~\ref{A4} is trivially satisfied with $d=1/\# K$.
\end{proof}

\smallskip

\section{Applications}\label{sec:applications}
\subsection{Convergence to quasi-stationary distribution}
\label{QSD}

Let $(X_t)_{t\geq 0}$ be a c\`ad-l\`ag Markov process on the state space $\X \cup \{\partial\}$, where 
$\X$ is measurable space and $\partial$ is an absorbing state.
In this section, we apply the results to the (non-conservative) semigroup defined for any measurable bounded function $f$ on $\mathcal X$ and any $x\in\mathcal X$ by
$$M_tf(x)=\E_{x}\left[f(X_t)\1_{X_t \ne  \partial }\right].$$
We assume that there exists a positive function $V$ defined on $\X$ such that for any $t>0$, there exists $C_t>0$ such that, $\E_x[V(X_t)\1_{X_t \ne  \partial }]\leq C_t V(x)$ for any $x\in \X$. 
%This allows to extend the definition above and  
This ensures that the semigroup $M$ acts on $\mathcal B(V)$ and that we can use the framework of Section \ref{mainetpastweedie}.   \\
A quasi-stationary distribution (QSD)  is a  probability 
law  $\pi$ on  $\X$ such that
$$
\forall t\geq 0, \ \mathbb{P}_\pi (X_t \in \cdot \ | X_t\ne \partial) = \pi(\cdot).
$$
Corollary~\ref{lesnormandsmecassentlescouillesavecleurslabels} directly gives existence and uniqueness of a QSD  and quantitative estimates for the convergence.
Recall that $\P(V)$ stands for the set of probability measures which integrate $V.$ 
\begin{theo}
\label{th:qsd}
Assume that $(M_t)_{t\geq 0}$ satisfies Assumption \textbf{A} with $\inf_{\X}V > 0$.
Then, there exist a unique quasi-stationary distribution $\pi \in \mathcal{P}(V)$, and  $\lambda_0 >0, h\in \mathcal B_+(V)$, $C,w>0$  such that for all $\mu \in  \mathcal{P}(V)$ and $t\geq 0$
$$\left\| \e^{\lambda _0t} \mathbb{P}_\mu (X_t \in \cdot) - \mu(h) \pi \right\|_{\mathrm{TV}} \leq C \mu(V)  \e^{- \omega  t},$$
and
%\begin{equation}
%\label{eq:cvqsd}
$$
\left\|  \mathbb{P}_\mu (X_t \in \cdot \ | \ X_t\ne \partial) - \pi \right\|_{\mathrm{TV}} \leq C \frac{\mu(V)}{\mu(h)} \e^{- \omega  t}.
$$
%\end{equation}

\end{theo}

\noindent
It extends and complements recent results, see {\it e.g.} \cite{CV18} for various interesting examples and discussions below for comparisons of statements. In particular, it relaxes the conditions of boundedness for $\psi$ and provides  a quantitative estimate for exponential convergence of the conditional distribution to the QSD using the eigenfunction $h$.\\

As an application,  we consider the simple  but interesting case of a continuous time random walk  on integers,  with jumps $+1$ and $-1$, absorbed at $0$. 
We obtain optimal results for the exponential convergence to the QSD. Let us consider the 
Markov process $X$ %random walk on $\mathbb N=\{0,1,\ldots\}$ absorbed in $0$  
whose transition rates and  generator are given by the linear operator defined for any $n\in \mathbb N$ and $f : \mathbb N\rightarrow \mathbb R$ by
%\begin{equation}
%\label{eq:gen-bd}
$$
\LL f(n) = b_n  (f(n+1)-f(n)) + d_n (f(n-1)-f(n)),
$$
%\end{equation}
where
$b_i=b>0$, $d_i=d>0$ for any $i\geq 2$, $b_1,d_1>0$, $b_0=d_0=0$.
This process is  a Birth and Death (BD) process which follows a simple random walk before reaching $1$.  
If $d\geq b$, this process is almost surely absorbed at $0$. The convergence in law  of such processes conditionally  on non-absorption  has been studied in many works \cite{AT12,GJ18,HMS03,Maillard2016,Ma15,vD91,vD03,V15,ZZ13}.\\
%\cmb{a ranger dans l'ordre que vous voulez}\\
The  necessary and sufficient condition for $\xi-$positive recurrence of BD processes is known from the work of van Doorn \cite{vD03}. More precisely here, the fact that there exists $\lambda>0$ such that
for any $x>0$ and $i>0$,  $\e^{-\lambda t}\mathbb P_x(X_t =i )$ converges to a positive finite limit as $t\rightarrow \infty$  is given by the  following condition 
$${\bf (H)} \qquad  \qquad  \Delta: =(\sqrt{b}-\sqrt{d})^2 +b_1\left(\sqrt{d/b}-1\right)- d_1 >0. $$
We notice that $b=d$ is excluded by condition ${\bf (H)}$ and indeed in this case $t\mapsto \mathbb{P}(X_t\neq 0)$ decreases polynomially. Similarly, the case $b_1=b$ and $d_1=d$ is excluded and  there is an additional linear term in the exponential decrease of $\mathbb P_x(X_t =i)$. \\
Moreover we know  from  \cite{vD91} that  condition ${\bf (H)}$ ensures that
$\mathbb P_x(X_t \in \cdot \vert X_t\ne 0)$ converges to the unique QSD $\pi$ for any $x>0$.
To the best of our knowledge,  under Assumption ${\bf (H)}$, the speed of convergence of $\e^{-\lambda t}\mathbb P_x(X_t =i)$  or $\mathbb P_x(X_t =i \vert X_t \ne 0)$ and  the extension of the convergence to infinite support masses were not proved, see {\it e.g.}   \cite[page 695]{vD91}. For a  subset of parameters satisfying  {\bf (H)}, \cite{V15} obtains
the convergence to QSDs for non-compactly supported initial measures $\mu$ such that $\mu(V)<\infty$.
%Using the same Lyapunov functions as in \cite{V15} or those defined below,  results in \cite{CV18} allow to get exponential convergence for  a subset of parameters satisfying  {\bf (H)}. 
We obtain below quantitative exponential estimates for the full range of parameters given by ${\bf (H)}$ and  
 non-compactly supported initial measures. 

\

More precisely, we set
$\X=\N\setminus\{0\}=\{1, \ldots\}$, $V:n\mapsto \sqrt{d/b}^{\ n}$, $\psi:n\mapsto\eta^n$,
for $n\in \X$, where $\eta=\sqrt{d/b}-\Delta/2b_1\in (0,\sqrt{d/b})$. 

\begin{coro}%[Linear drift case]
\label{coro:linearbd}
Under Assumption ${\bf (H)}$, there exists a unique QSD $\pi \in \mathcal{P}(V)$, and  $\lambda_0 >0, h\in \mathcal B_+(V)$ and $C,w>0$  such that for all $\mu \in  \mathcal{P}(V)$ and $t\geq 0$,
$$\left\| \e^{\lambda _0t} \mathbb{P}_\mu (X_t \in \cdot) - \mu(h) \pi \right\|_{\mathrm{TV}} \leq C \mu(V)  \e^{- \omega  t}$$
and
%\begin{equation}
%\label{eq:cvqsdbd}
$$
\left\|  \mathbb{P}_\mu (X_t \in \cdot \ | \ X_t\ne 0) - \pi \right\|_{\mathrm{TV}} \leq C \frac{\mu(V)}{\mu(h)} \e^{- \omega  t}.
$$
%\end{equation}
\end{coro}

Note that the constants above  can be explicitly derived from Lemma~\ref{lem:h}, that these  estimates hold for non-compactly supported initial laws and that $V$ and $\psi$ are not eigenelements. As perspectives, 
we expect this statement to be generalized to BD processes where $b_n,d_n$ are constant outside some  compact set of $\N$.
Finally, we  hope that the proof will help to study the non-exponential decrease of the non-absorption probability, in particular for  random walks, corresponding to $b=b_1, d=d_1$.

\begin{proof}[Proof of Corollary~\ref{coro:linearbd}]
 For $u\geq 1$, let $\varphi_u : n \mapsto u^n$ for $n\geq 1$ and $\varphi_u(0)=0$. We have
 %\footnote{Functions $\varphi_u$ are unbounded and then do not belong to the domain of the generator (in a functionnal point of view). We show in Appendix~\ref{sect:appendixqsd} that they verify such inequality in the sense of Section~\ref{ssec:generator}  though a localisation argument. This proof hold for more general Markov process.}
$
\LL \varphi_u(n) = \lambda_u(n) \varphi_u(n),
$
for any $n\in \N$, where 
$$ 
\lambda_u(n) =\lambda_u= b(u -1) + d \left(1/u-1\right) \quad (n\geq 2), \qquad \lambda_u(1)=   b_1(u -1) -d_1.
$$
We set
$$a=\inf_{u>0} \lambda_u=\lambda_{\sqrt{d/b}}= -(\sqrt{d}-\sqrt{b})^2, \qquad \zeta=\Delta\frac{V(1)}{\psi(1)}.$$
Note that from  ${\bf (H)}$, $\zeta>0$. Then, setting $V(0)=\psi(0)=0$, we have on $\N=\X\cup\{0\}$ that $V=\varphi_{\sqrt{d/b}}$ and %By assumption $a= -(\sqrt{d}-\sqrt{b})^2\leq b_1(u -1) -d_1=\lambda_u(1)$ and we obtain for $V=\varphi_{\sqrt{d/b}}$
\begin{equation}
\label{LLV}
\LL V= aV+ \zeta \1_{\{n=1\}} \psi\leq aV+\zeta\psi.
\end{equation}
Moreover, on $\N$, $\psi=\varphi_\eta$ and  
\begin{equation}
\label{LLp}
b\psi \leq \LL \psi \leq \xi \psi
\end{equation}
with
$b=\min(\lambda_\eta,\lambda_{\eta}(1)) = \min\left(\lambda_\eta, a + \frac{\Delta}{2}\right) >a = \inf_{u>0} \lambda_u$, and $\xi=\max(\lambda_\eta,\lambda_\eta(1)).$

Using now a classical localization argument (see Appendix~\ref{app:localization} for details), the drift conditions \eqref{LLV}-\eqref{LLp} ensure that for any $n\geq 1$ and $t\geq 0$,
\begin{align}
\label{semiLV}
\E_n[V(X_t)] & \leq V(n)+\int_0^t \E_n[(aV+\zeta\psi)(X_s)]ds, \\
\psi(n)+ \int_0^t \E_n[b \psi(X_s)]ds & \leq \E_n[\psi(X_t)]\leq \psi(x)+ \int_0^t \E_n[\xi\psi(X_s)]ds. \label{semiLpsi}
\end{align}
%\begin{eqnarray}
%\label{semiLV}
%&& \E_n(V(X_t))\leq V(x)+\int_0^t \E_n((aV+\zeta\psi)(X_s))ds, \\
%&&  \qquad  \qquad \qquad  \int_0^t \E_n(b \psi(X_s))ds \leq \E_x(\psi(X_t))-\psi(x)\leq \int_0^t \E_n(\xi\psi(X_s))ds \label{semiLpsi}
%\end{eqnarray}
Considering  the generator $\L$ of the semigroup
$M_tf(x)=\E[f(X_t)1_{X_t\ne 0}]$ for $x\in \X$ and $f\in \B(V)$ and recalling the definition of  Section~\ref{ssec:generator}, these inequalities ensure 
\[\L V\leq aV+\zeta\psi,\qquad \L\psi\geq b\psi,\qquad\L\psi\leq \xi\psi.\]
Finally, the fact that $b_i,d_i>0$ for $i\geq 1$ ensures $\delta_iM_t(\{j\}) >0$ for any $i,j\in \X$ and $t>0$.
Then combining 
Propositions~\ref{prop:lyapun} and~\ref{prop:irr} ensures that   Assumption~{\bf A} holds for $M$ with the functions
$(V,\psi)$.  Applying Theorem~\ref{th:qsd} ends the proof.
\end{proof}

\subsection{The growth-fragmentation equation}
\label{EDP}

In this section we apply our general result to the growth-fragmentation partial differential equation
\begin{equation}
\label{eq:GF}
\partial_{t}u_t(x)+\partial_x u_t(x) + B(x)u_t(x) = \int_0^1B\Bigl(\frac xz\Bigr)u_t\Bigl(\frac xz\Bigr)\frac{\wp(d z)}{z}
\end{equation}
for $t,x>0$.
This nonlocal partial differential equation is complemented with the zero flux boundary condition $u_t(0)=0$ for all $t>0$ and an initial data $u_0=\mu.$
The unknown $u_t(x)$ represents the population density at time $t$ of some ``particles'' with ``size'' $x>0,$
which can be for instance the size of a cell~\cite{DHT,HW89}, the length of a polymer fibril~\cite{EPW}, the window size in data transmission over the Internet~\cite{BCGMZ,ChafaiMalrieuParoux}, the carbon content in a forest~\cite{BGP19}, or the time elapsed since the last discharge of a neuron~\cite{CanizoYoldas,PPS}.
Each particle grows with speed~$1,$ and splits with rate $B(x)$ to produce smaller particles of sizes $zx$ with $0<z<1$ distributed with respect to the fragmentation kernel~$\wp.$

\

We assume that $B:[0,\infty)\to[0,\infty)$ is a continuously differentiable increasing function and $\wp$ is a finite positive measure on $[0,1]$ for which there exist $z_0\in(0,1),$ $\epsilon\in[0,z_0]$ and $c_0>0$ such that
\begin{equation}\label{as:wp_low_bound}\wp(dz)\geq \frac{c_0}{\epsilon}\1_{[z_0-\epsilon,z_0]}(z)dz\quad\text{if}\ \epsilon>0\qquad\text{or}\qquad\wp\geq c_0\delta_{z_0}\quad\text{if}\ \epsilon=0.\end{equation}
For any $r\in\R$ we denote by $\wp_r\in(0,+\infty]$ the moment of order $r$ of $\wp$
\[\wp_r=\int_0^1z^r\wp(dz).\]
Notice that Assumption~\eqref{as:wp_low_bound} implies that $r\mapsto\wp_r$ is strictly decreasing.
The conservation of mass during the fragmentation leads to impose%is reflected in the condition
\[\wp_1=1.\]
The zero order moment $\wp_0$ represents the mean number of fragments.
The conditions above ensure that $\wp_0>1$ and as a consequence the growth-fragmentation equation we consider is not conservative.
The conservative form where $\wp_1=1$ is replaced by $\wp_0=1$ also appears in some situations~\cite{BCGMZ,BGP19,CanizoYoldas,ChafaiMalrieuParoux,PPS}.
In this case, the eigenelements are given by $h(x)=1,$ $\lambda=0,$ and the classical theory of the conservative Harris theorem applies~\cite{Bouguet,BGP19}.
Here we are interested in the more challenging non-conservative case.

\

An important feature in the long time behavior of the (non-conservative) growth-fragmentation equation is the property of asynchronous exponential growth~\cite{Webb87},
which refers to a separation of the variables $t$ and $x$ when time $t$ becomes large: the size repartition of the population stabilizes and the total mass grows exponentially in time.
This question attracted a lot of attention in the last decades, \cite{BW18,CCM11,DHT,MMP2,MS16,PR05} to mention only a few. 
As far as we know, the existing literature assume either that the state space is a bounded interval instead of the whole $(0,\infty)$, as in~\cite{BPR,DHT}, or that the fragmentation rate has at most a polynomial growth, as in \cite{BCG13,BG2,CCM11,MS16}.
We can consider here unbounded state space and we relax the latter condition by not assuming any upper bound on the division rate with a similar approach through $h$-transform).
We thus obtain the existence of the Perron eigentriplet for super-polynomial fragmentation rates.
Second,  an explicit spectral gap was known only in the case of a constant division rate~\cite{BCGMZ,ChafaiMalrieuParoux,LP09,MS16,PR05,ZvBW2}. Our method allows us to get it for  more general fragmentation rates.
Finally, it guarantees exponential convergence for measure solutions, thus drastically improving~\cite{DDGW}.

\

Let us now state the main result of this section.

\begin{theo}\label{th:GF}
Let $k>1$ and $V(x)=1+x^k$.
Under the above assumptions \pg{that $B$ is a $C^1$ increasing function, $\wp$ satisfies~\eqref{as:wp_low_bound} and $\wp_1=1$}, there exists a unique triplet $(\gamma,h,\lambda)\in\M_+(V)\times \B_+(V)\times \mathbb{R}$ of Perron-Frobenius eigenelements with $\gamma(h)=\left\|h\right\|_{\B(V)}=1,$
{\it i.e.} satisfying
\[\LL h=\lambda h\qquad\text{and}\qquad \gamma(\LL f)=\lambda\gamma(f)\]
for all $f\in C^1_c([0,\infty)),$ where $\LL:C^1([0,\infty))\to C^0([0,\infty))$ is defined by
\[\LL f(x)=f'(x)+B(x)\bigg(\int_0^1f(zx)\wp(dz)-f(x)\bigg).\]
Moreover there exist constants $C,\omega>0$ such that for all $\mu \in \M(V)$ the solution to Equation~\eqref{eq:GF} with $u_0=\mu$ satisfies for all $t\geq0$, 
\begin{equation}\label{eq:conv_norm_GF}
\left\|\e^{-\lambda t}u_t-u_0(h)\gamma\right\|_{\M(V)}\leq C \e^{-\omega t}\left\|u_0\right\|_{\M(V)}.
\end{equation}
\end{theo}

\

The constants $\lambda,\omega$, and $C$ can be estimated quantitatively, and the eigenfunction $h$ satisfies
$(1+x)^{1-q(k-1)} \lesssim h\lesssim (1+x)^k$,
for some explicit $q>0$.
Note also that we cannot expect the convergence~\eqref{eq:conv_norm_GF} to hold true in $\M(h)$ in general.
In particular it is known to be wrong when $B$ is bounded~\cite{BG1}.

\

The end of the section is devoted to the proof of Theorem~\ref{th:GF}.
We can associate to Equation~\eqref{eq:GF} a semigroup $(M_t)_{t\geq0}$, to which we will apply our result in Theorem~\ref{V-contraction} after having checked that it verifies Assumption~{\bf A} with the functions
\[V(x)=1+x^k \qquad\text{and}\qquad \psi(x)=\frac12(1+x)\]
with $k>1$.
The factor $\frac12$ in the definition of $\psi$ is to satisfy the inequality $\psi\leq V$.
We only give here the definition of this semigroup as well as its main properties which are useful to verify Assumption~{\bf A}, and we refer to Appendix~\ref{app:GFsemigroup} for the proofs.
For any $f:[0,\infty)\to\R$ measurable and locally bounded on $[0,\infty)$, we define the family $(M_tf)_{t\geq0}$ as the unique solution to the equation
\begin{equation}\label{eq:Duhamel}
M_tf(x)=f(x+t)\e^{-\int_0^tB(x+s)ds}+\int_0^t\e^{-\int_0^s B(x+s')ds'}B(x+s)\int_0^1M_{t-s}f(z(x+s))\wp(dz)\,ds.
\end{equation}
This semigroup is positive and preserves $C^1([0,\infty))$.
More precisely if $f\in C^1([0,\infty)),$ then the function $(t,x)\mapsto M_tf(x)$ is continuously differentiable on $[0,\infty)^2$ and 
%satisfies
\[\partial_tM_tf(x)=\LL M_tf(x)=M_t\LL f(x),\]
where $\LL$ is defined in Theorem~\ref{th:GF}.
The space $\B(V)$ is invariant under $(M_t)_{t\geq0}$ and for any $\mu\in\M(V)$, we can define by duality $\mu M_t\in\M(V).$
The family $(\mu M_t)_{t\geq0}$ is  solution to  \eqref{eq:GF} with initial data $\mu,$ in a weak sense made precise in Appendix~\ref{app:GFsemigroup}. 

\

Let $x_0\geq0$ and $\underline B>0$ such that for all $x\geq x_0$, $B(x)\geq \underline B.$ Now define
\begin{equation}\label{eq:tau}
t_0=\frac{1+z_0+(1+\epsilon)x_0}{1-z_0}+\frac12,\qquad t_1=\frac{1-z_0}{2z_0},\qquad \tau=t_0+t_1,
\end{equation}
and for all integer $n\geq0$,
\[y_n=\Big(\frac{1+z_0}{2z_0}\Big)^n+x_0.\]

\begin{lem} \label{lemmeaveccond}
i) Setting $\varphi(x)=1-\sqrt x+x$, we have
$\psi\leq\varphi\leq2\psi$ and there exist $\zeta>0$ and $a<b<\xi$ such that
\[\L V\leq aV+\zeta\psi,\qquad \L\psi\geq b\psi,\qquad\L\varphi\leq \xi\varphi,\]
where $\L$ is the generator of $(M_t)_{t\geq0}$ in the sense defined in Section~\ref{ssec:generator}. \\
ii) For all $n\geq0,$ all $x\in[0,y_n],$ and all $f:[0,\infty)\to[0,\infty)$ locally bounded we have
\[M_\tau f(x)\geq\e^{-\tau B(y_n+\tau)}\frac{(c_0\underline B)^{n+1}}{1-z_0} \frac{t_1^n}{n!}\, \nu(f),\]
where $\nu$ is the probability measure defined by
\[\nu(f)=\int_{z_0(y_0+\tau)}^{z_0(y_0+\tau)+1}f(y)\,dy.\]
iii) %}\label{cor:harnack}
For all $\eta>0$ there exists $c_\eta>0$ such that for all $t,x\geq 0$ and $y\in[\eta x,x]$
\[c_\eta\leq\frac{M_t\psi(y)}{M_t\psi(x)}\leq1.\]
iv) %\begin{coro}\label{cor:A4GF}
 For all $n\geq0,$ there exists $d>0$ such that 
\[d\frac{M_t\psi(x)}{\psi(x)}\leq\frac{M_t\psi(y)}{\psi(y)}\]
for all $t\geq 0,$ $x\in[0,y_n]$ and $y\in[z_0(y_0+\tau),z_0(y_0+\tau)+1]=\supp\nu.$
\end{lem}
\begin{proof}[Proof of  Lemma \ref{lemmeaveccond} i)]
Since by Proposition \ref{prop:Mtdroit} the identity $\partial_tM_t=M_t\LL$ is valid for all $C^1$ functions and the semigroup $M$ is positive,
we only need to prove that
$\LL V\leq aV+\zeta\psi,\  \LL\psi\geq b\psi,\ \LL\varphi\leq \xi\varphi.$
First, 
\[\LL x^r=rx^{r-1}+(\wp_r-1)B(x)x^r\]
for any $r\geq0$.
We deduce that
\[2\,\LL\psi(x)=1+(\wp_0-1)B(x)\geq0,\]
so that $b=0$ suits.
For $\varphi$ we have
\[\LL\varphi(x)=1-\frac{1}{2\sqrt x}+(\wp_0-1)B(x)-(\wp_{\frac12}-1)B(x)\sqrt x.\]
Since $x\mapsto(\wp_0-1)-(\wp_{\frac12}-1)\sqrt x$ is negative for $x>\big(\frac{\wp_0-1}{\wp_{1/2}-1}\big)^2$ and $B$ is increasing we deduce
\[\LL\varphi(x)\leq1+(\wp_0-1)B\bigg(\Big(\frac{\wp_0-1}{\wp_{\frac12}-1}\Big)^2\,\bigg)=:\frac{\xi}{2}\leq \xi\varphi(x).\]
For $V$ we have
\begin{align*}
\LL V(x)&=kx^{k-1}+(\wp_0-1)B(x)+(\wp_k-1)B(x)x^k=\underbrace{\Big[\big((\wp_k-1)+(\wp_0-1)x^{-k}\big)B(x)+\frac kx\Big]}_{\to\,l:= (\wp_k-1)\lim_{x\to+\infty}B(x)\ \text{when}\ x\to+\infty}x^k.
\end{align*}
Since $\wp_k<1$ and $B$ is increasing, the limit $l$ belongs to $[-\infty,0)$ and we can find $x_1>0$ such that for all $x\geq x_1$,
\[\LL V(x)\leq a x^k=aV(x)-a,\]
where $a=\max\{l/2,-1\}<0.$
For all $x\in[0,x_1]$ we have
\[\LL V(x)\leq kx_1^{k-1}+(\wp_0-1)B(x_1)\]
and finally setting $\zeta=2(kx_1^{k-1}+(\wp_0-1)B(x_1)-a)$, we get that for all $x\geq0$
\[\LL V(x)\leq a V(x)+\frac\zeta2\leq a V(x)+\zeta\psi(x).\]
It ends the proof of {\it i)}.
\end{proof}
Before proving {\it ii)}, let us briefly comment on the definition of $t_0,t_1$ and $y_n.$
The time~$t_1$ and the sequence $y_n$ are chosen in such a way that
\[y_0>0,\quad y_0\geq x_0,\qquad\lim_{n\to\infty}y_n=+\infty,\qquad\text{and}\qquad z_0(y_{n+1}+t_1)\leq y_n.\]
The choice of $t_0$ appears in the proof of the case $n=0$ and the definition of $\nu.$

\

Since $\tau$ is independent of $n$ and $y_n\to+\infty$ when $n\to\infty$ we can find $R$ and $n$ large enough so that $\supp\nu\subset K\subset[0,y_n],$ where $K=\{x,\ V(x)\leq R\psi(x)\},$ and thus {\it ii)} guarantees that Assumption~\ref{A3} is satisfied with time $\tau$ on $K.$
More precisely it suffices to take $R$ and $n$ large enough so that
\begin{equation}\label{eq:Rn}
\frac{1+(z_0(y_0+\tau)+1)^k}{1+z_0(y_0+\tau)+1}\leq \frac R2\leq\frac{1+y_n^k}{1+y_n}.
\end{equation}

\begin{proof}[Proof of  Lemma \ref{lemmeaveccond} ii)]
Let $f\geq0.$
We prove by induction on $n$ that for all $x\in[0,y_n]$ and all $t\in[0,t_1]$ we have
\begin{equation}\label{HR}
M_{t_0+t} f(x)\geq\e^{-(t_0+t) B(y_n+\tau)}\frac{(c_0\underline B)^{n+1}}{1-z_0} \frac{t^n}{n!}\, \nu(f),
\end{equation}
which yields the desired result by taking $t=t_1.$

\

We start with the case $n=0.$
The Duhamel formula~\eqref{eq:Duhamel}
%\[M_tf(x)=f(x+t)\e^{-\int_0^tB(x+s)ds}+\int_0^t\e^{-\int_0^s B(x+s')ds'}B(x+s)\int_0^1M_{t-s}f(z(x+s))\wp(dz)\,ds\]
ensures, using the positivity of $M_t$ and the growth of $B,$ that for all $t,x\geq0$
\[M_tf(x)\geq\e^{-tB(x+t)}\int_0^tB(x+s)\int_0^1f(z(x+s)+t-s)\wp(dz)\,ds.\]
Thus for $t\geq x_0$ we have for all $x\geq0,$ using Assumption~\eqref{as:wp_low_bound} for the second inequality,
\[
M_tf(x)\geq \e^{-tB(x+t)}\underline B\int_{x_0}^{t}\!\int_0^1\!f(z(x+s)+t-s)\wp(dz)ds
%&\geq \e^{-tB(x+t)}\underline B\int_0^1\int_{z(x+t)}^{z(x+x_0)+t-x_0}f(y)\,dy\,\frac{\wp(dz)}{1-z}\\
\geq \e^{-tB(x+t)}\frac{\underline B\,c_0}{1-z_0}\int_{z_0(x+t)}^{(z_0-\epsilon)(x+x_0)+t-x_0}\!f(y)dy.
\]
We deduce that for $t\in[t_0,t_0+t_1]$ and $x\in[0,y_0]$
\[M_tf(x)\geq \e^{-tB(x_0+\tau)}\frac{\underline B\,c_0}{1-z_0}\int_{z_0(x_0+t_0+t_1)}^{(z_0-\epsilon)x_0+t_0-x_0}f(y)\,dy.\]
The time $t_0$ has been defined in such a way that $(z_0-\epsilon)x_0+t_0-x_0=z_0(x_0+t_0+t_1)+1$ so
$\int_{z_0(x_0+t_0+t_1)}^{(z_0-\epsilon)x_0+t_0-x_0}f(y)\,dy=\nu(f)$
and this finishes the proof of the case $n=0.$

\

Assume  that~\eqref{HR} holds for $n$ and let's check it for $n+1.$
By  Duhamel formula, using that $y_n\geq x_0$ and $z_0(y_{n+1}+t_1)\leq y_n,$ we get for $x\in[y_n,y_{n+1}]$ and $t\in[0,t_1]$,
\begin{align*}
M_{t_0+t}f(x)&\geq\int_0^t\e^{-\int_0^s B(x+s')ds'}B(x+s)\int_0^1M_{t_0+t-s}f(z(x+s))\wp(dz)\,ds\\
&\geq\underline B\int_0^t\e^{-sB(x_{n+1}+t_1)}\int_0^{z_0}M_{t_0+t-s}f(z(x+s))\wp(dz)\,ds\\
&\geq\underline B^{n+2}\,\frac{c_0^{n+1}}{1-z_0}\nu(f)\int_0^t\e^{-sB(x_{n+1}+t_1)}\e^{-(t_0+t-s) B(x_n+\tau)}\frac{(t-s)^n}{n!}\int_0^{z_0}\wp(dz)\,ds\\
&\geq\e^{-(t_0+t) B(y_{n+1}+\tau)}\underline B^{n+2}\,\frac{c_0^{n+2}}{1-z_0}\frac{t^{n+1}}{(n+1)!}\nu(f)
\end{align*}
and the proof is complete.
\end{proof}

We now turn to the proof of {\it iii)} and {\it iv)}, which uses the monotonicity results proved in Lemma~\ref{onsennuiecest}, see Appendix~\ref{app:GFsemigroup}.

\begin{proof}[Proof of Lemma \ref{lemmeaveccond}  iii)]
The second inequality readily follows from Lemma~\ref{onsennuiecest}~$ ii)$.
For the first one, we start with a technical result on $\wp.$
Due to the assumption we made on $\wp,$ if we set $z_1>\max(z_0,1-c_0\big(z_0-\epsilon/2))$, we have
%Due to the assumption we made on $\wp,$ we can find $z_1\in(0,1)$ such that $\int_{z_1}^1\wp(dz)<1.$
%To do so we pick up $z_1\in(z_0,1)$ such that $z_1>\max(z_0,1-c_0\big(z_0-\epsilon/2))$ and we have
\[\varrho:=\int_{z_1}^1\wp(dz)\leq
%\frac{1}{z_1}\int_{z_1}^1z\,\wp(dz)=
\frac{1}{z_1}\Big(1-\int_0^{z_1}z\,\wp(dz)\Big)\leq
%\frac{1}{z_1}\Big(1-c_0\int_0^{z_1}z\,\delta_{z_0}^\epsilon(dz)\Big)=
\frac{1-c_0\big(z_0-\epsilon/2)}{z_1}<1.\]
Using Lemma  \ref{onsennuiecest} {\it ii)} and {\it iii)}, we deduce that for all $t\geq s\geq0$ and all $x\geq 0$
\[\int_0^1M_{t-s}\psi(z(x+s))\,\wp(dz)\leq\int_0^1M_t\psi(zx)\,\wp(dz)\leq\wp_0M_t\psi(z_1x)+\varrho M_t\psi(x).\] 
Now from the Duhamel formula we get, using that $t\mapsto t\,\e^{-\int_0^tB(s)\,ds}$ is bounded on $[0,\infty),$
\begin{align*}
M_t\psi(x)&=\psi(x+t)\e^{-\int_0^tB(x+s)\,ds}+\int_0^t\e^{-\int_0^s B(x+s')\,ds'}B(x+s)\int_0^1M_{t-s}\psi(z(x+s))\,\wp(dz)ds\\
&\leq(1+x+t)\e^{-\int_0^tB(s)\,ds}+\Big(1-\e^{-\int_0^t B(x+s)\,ds}\Big)\Big(\wp_0M_t\psi(z_1x)+\varrho M_t\psi(x)\Big)\\
&\leq C_0\psi(x)+\wp_0M_t\psi(z_1 x)+\varrho M_t\psi(x).
\end{align*}
Choosing an integer $n$ such that $z_1^n\leq\eta$ we obtain
\[M_t\psi(x)\leq \frac{C_0}{1-\varrho}\sum_{k=0}^{n-1}\Big(\frac{\wp_0}{1-\varrho}\Big)^k\,\psi(x)+\Big(\frac{\wp_0}{1-\varrho}\Big)^nM_t\psi(\eta x)=C_1\psi(x)+C_2\,M_t\psi(\eta x)\]
and since $M_t\psi(\eta x)\geq\psi(\eta x)\geq\eta\psi(x)$ we obtain for any $y\in[\eta x,x]$,
\[\frac{M_t\psi(y)}{M_t\psi(x)}\geq\frac{M_t\psi(\eta x)}{C_1\psi(x)+C_2M_t\psi(\eta x)}\geq\frac{\eta}{C_1+C_2\eta},\]
which ends the proof.
\end{proof}

\begin{proof}[Proof of Lemma \ref{lemmeaveccond} iv)]
Point {\it iii)} applied with $\eta=z_0(y_0+\tau)/y_n$
gives, together with Lemma~\ref{onsennuiecest}~{\it ii)}, that for all $x\in[0,y_n]$ and $y\in[z_0(y_0+\tau),z_0(y_0+\tau)+1]$,
\[M_t\psi(y)\geq M_t\psi(z_0(y_0+\tau))\geq c_\eta M_t\psi(y_n)\geq c_\eta M_t\psi(x).\]
Besides, 
$\psi(x)/\psi(y)=(1+x)/(1+y)\geq 1/(z_0(y_0+\tau)+2)$
and we get
\[\frac{M_t\psi(y)}{\psi(y)}\geq \frac{c_\eta}{z_0(y_0+\tau)+2} \frac{M_t\psi(x)}{\psi(x)},\]
which ends the proof.
\end{proof}

We are now in position to prove Theorem~\ref{th:GF}.

\begin{proof}[Proof of Theorem~\ref{th:GF}]
Fix $\tau$ defined in~\eqref{eq:tau}.
In Lemma~\ref{lemmeaveccond}~{\it i)} we have verified the assumptions of Proposition~\ref{prop:lyapun}, so we can find a real $R>0$ and an integer $n\geq0$ large enough so that~\eqref{eq:Rn} and Assumptions~\ref{A1}-\ref{A2} are satisfied with $K=\{V\leq R\psi\}.$
Then, points {\it ii)} and {\it iv)} in Lemma~\ref{lemmeaveccond} ensure that Assumptions~\ref{A3} and~\ref{A4} are also satisfied.
So Assumption~{\bf A} is verified for $(V,\psi)$ and by virtue of Theorem~\ref{V-contraction} inequality~\eqref{eq:conv_norm_GF} is proved, as well as the bounds on $h$.
It remains to check that $h$ is continuously differentiable and that $h$ and $\gamma$ satisfy the eigenvalue equations $\LL h=\lambda h$ and $\gamma\LL=\lambda\gamma.$
By definition of $h$, the Duhamel formula gives
\[h(x)\e^{\lambda t}=h(x+t)\e^{-\int_0^tB(x+s)ds}+\int_0^t\e^{-\int_0^s B(x+s')ds'}B(x+s)\int_0^1\e^{\lambda(t-s)}h(z(x+s))\wp(dz)\,ds\]
and we deduce that for any $x>0$ the function $t\mapsto h(t+x)$ is continuous and then continuously differentiable.
Moreover, we have the identity $\partial_tM_th=M_t\LL h$ and since $M_th=\e^{\lambda t}h$ we deduce $\LL h=\lambda h.$\\
For the equation on $\gamma$, we start from Proposition~\ref{prop:Mtgauche} which ensures that 
\[\e^{\lambda t}\gamma(f)=(\gamma M_t)(f)=\gamma (f)+\int_0^t(\gamma M_s)(\LL f)\,ds=\gamma (f)+\frac{\e^{\lambda t}-1}{\lambda}\gamma(\LL f).\]
for any $f\in C^1_c([0,\infty))$. Differentiating with respect to $t$ yields the result.
\end{proof}

\subsection{Comments and a few perspectives}\label{ssec:discrete}

The proof of Theorem~\ref{V-contraction} consists in first proving the $V$-uniform ergodicity of the discrete time semigroup $(M_{n\tau})_{n\geq 0}= (M_{\tau}^n)_{n\geq 0}$,
and then extend it to the continuous setting.
We can thus state analogous results for a discrete time semigroup $(M^n)_{n\in \N}$
by  making the following  assumption for a couple of positive functions $(V,\psi)$ with $\psi\leq V$.
\begin{hypoB} \label{assu:condition_generales_discret}
There exist some integer $\tau>0$, real numbers $\beta >\alpha>0$, $\C> 0$, $(c,d)\in(0,1]^2$, some set $K \subset \X$ such that $\sup_KV/\psi<\infty$, and some probability measure $\nu$ on $\X$ supported by $K$ such that 

\medskip

\begin{enumerate}[label=(B\arabic*), parsep=3mm]
\item \label{B1} $ M^{\tau}  V \leq \alpha   V + \C  \mathbf{1}_K \psi$,
\item \label{B2}$M^{\tau} \psi \geq \beta \psi$,
\item \label{B3}
$\dis \inf_{x\in K}\frac{M^{\tau }(f\psi)(x)}{M^{\tau}\psi(x)}\geq c\, \nu(f)\quad$ for all $f\in\B_+(V/\psi)$,
\item \label{B4}$
\dis \nu\left(\frac{ M^{n\tau}\psi}{\psi}\right)\geq d\sup_{x\in K}\frac{M^{n\tau}\psi(x)}{\psi(x)}\quad$ for all positive integers $n$.
\end{enumerate}
\end{hypoB}

\medskip

The discrete time counterpart of Theorem~\ref{V-contraction} is stated below. 
\begin{theo}\label{th:main-discret}
i) Let $(V,\psi)$ be a couple of measurable functions from $\X$ to $(0,\infty)$ such that $\psi\leq V$ and which satisfies  Assumption {\bf B}. Then, there exists a unique triplet $(\gamma,h,\lambda)\in\M_+( V)\times \B_+( V)\times \mathbb{R}$ of eigenelements of $M$ with $\gamma (h) =\|h\|_{\B(V)}= 1$,
{\it i.e.} satisfying 
\begin{equation}
\label{vecteursales_discr}
\gamma M= \lambda \gamma\qquad\text{and}\qquad M h = \lambda h.
\end{equation}
Moreover,  there exists $C>0$ and $\rho \in (0,1)$ such that for all $n\geq0$ and  $\mu \in \M( V)$, 
\begin{equation}\label{eq:conv_norm_discr}
\big\| \lambda^{-n} \mu M^n - \mu(h)\gamma \big\|_{\M( V)}\leq C\left\|\mu\right\|_{\M( V)}\rho^{-n}.
\end{equation}
ii) Assume that there exist a positive measurable function $V,$ a triplet $(\gamma,h,\lambda)\in\M_+( V)\times \B_+( V)\times \mathbb{R},$ and constants $C,\rho>0$ such that
\eqref{vecteursales_discr} and \eqref{eq:conv_norm_discr} hold. Then, the couple $(V,h)$ satisfies Assumption ${\bf B}$.
\end{theo}

When $\inf_\X V>0$ we also have, as in Corollary~\ref{lesnormandsmecassentlescouillesavecleurslabels}, the existence of $C>0$ and $\pi\in \mathcal{P}(V)$ such that for all $\mu \in \mathcal{P}(V)$ and $n\geq 0$,
\begin{equation*}
%\label{eq:cvqsddiscret}
\left\| \frac{\mu M^n}{ \mu M^n \mathbf{1}} - \pi  \right\|_{\textrm{TV}} \leq C \frac{\mu(V)}{\mu(h)} \rho^n.
\end{equation*}
As a consequence, Assumption $\mathbf{B}$ gives sufficient conditions to have the existence, uniqueness and convergence to a QSD for a Markov chain.
The convergence of the $Q$-process,  the description of the domain of attraction and  the bounds on the extinction times  can be then obtained by usual procedure, see {\it e.g.}  \cite{CV18,V19}.

\medskip

For  the sake of simplicity, we have  not allowed $\psi$ to vanish in this paper. This excludes reducible structures. To illustrate this, let us consider the simple case of~\cite[Example 3.5]{BCP} where $\X=\{1,2\}$ and
$$
M = \begin{pmatrix} 
a & b \\
0 & c 
\end{pmatrix}
$$
with $a,b,c>0$. For any $\mu$ such that $\mu( \{1\})=0$ we have $c^{-n}\mu M^n=\mu(\{2\})\delta_2$.
If $\mu( \{1\})>0$ and $c>a$, we can apply Theorem~\ref{th:main-discret} with the right eigenvector
$$
\psi= (b,c-a)= b \mathbf{1}_{\{1\}} + (c-a) \mathbf{1}_{\{2\}}
$$
and show that, up to normalisation, $\mu M^n$ converges to $\delta_2$.
When $c\leq a$, we cannot use Theorem~\ref{th:main-discret} and there is no positive right eigenvector.
However, allowing $\psi$ to vanish enables handling the case $c<a$, as in \cite[Section 6]{CV18}.  Focusing on initial measures $\mu$ such that $\mu(\psi)>0$, a large part of our results actually holds when $\psi\geq 0$. Indeed \ref{B2} gives that if $\mu(\psi)>0$ then $\mu M^n\psi >0$ for every $n\geq 0$.
The critical case $a=c$ is a situation where there is no spectral gap.
We believe that our approach could also be extended to the study of semigroups without spectral gap by allowing $n$-dependent constants $d=d_n$ in~\ref{B4}, similarly as in \cite[Assumption ($\mathbf{H_4}$)]{BCP}.

\medskip

The discrete time result allows us to deduce $V$-uniform ergodicity, not only for continuous time semigroups as in Theorem~\ref{V-contraction}, but also in a time periodic setting.
We say that an propagator $(M_{s,t})_{0\leq s\leq t}$ is $T$-periodic if $M_{s+T,t+T}=M_{s,t}$ for all $t\geq s\geq0$,
and Theorem~\ref{th:main-discret} allows deriving some extension of the Floquet theory of periodic matrices~\cite{Floquet} to such periodic propagators.
We say that $(\gamma_{s,t},h_{s,t},\lambda_F)_{0\leq s\leq t}$ is a Floquet eigenfamily for the $T$-periodic propagator $(M_{s,t})_{0\leq s\leq t}$ if the families $(\gamma_{s,t})_{0\leq s\leq t}$ and $(h_{s,t})_{0\leq s\leq t}$ are $T$-periodic in the sense that
\[\gamma_{s+T,t+T}=\gamma_{s,t}=\gamma_{s,t+T}\qquad\text{and}\qquad h_{s+T,t+T}=h_{s,t}=h_{s,t+T}\]
for all $t\geq s\geq0$, and are associated to the Floquet eigenvalue $\lambda_F$ in the sense that
\[\gamma_{s,s}M_{s,t}=\e^{\lambda_F(t-s)}\gamma_{s,t}\qquad\text{and}\qquad M_{s,t}h_{t,t}=\e^{\lambda_F(t-s)}h_{s,t}\]
for all $t\geq s\geq0$.
Starting from Theorem~\ref{th:main-discret} and following the proof of~\cite[Theorem 3.15]{BCG17}, we obtain the periodic result stated below.

\begin{theo}\label{th:periodic}
Let $(M_{s,t})_{0\leq s\leq t}$ be a $T$-periodic propagator such that $(s,t)\mapsto\| M_{s,t}V\|_{\B(V)}$ is locally bounded, and suppose that $M_{s,s+T}$ satisfies Assumption~{\bf B} for some functions $V\geq\psi>0$ and some $s\in[0,T)$.
Then there exist a unique $T$-periodic Floquet family $(\gamma_{s,t},h_{s,t},\lambda_F)_{0\leq s\leq t}\subset\M_+(V)\times\B_+(V)\times\R$ such that $\gamma_{s,s}(h_{s,s})=\|h_{s,s}\|_{\B(V)}=1$ for all $s\geq0$,
and there exist $C\geq1,$ $\omega>0$ such that for all $t\geq s\geq0$ and all $\mu\in\M(V)$,
\[\big\|\e^{-\lambda_F (t-s)}\mu M_{s,t}-\mu(h_{s,s})\gamma_{s,t}\big\|_{\M( V)}\leq C\e^{-\omega (t-s)}\left\|\mu-\mu(h_{s,s})\gamma_{s,s}\right\|_{\M( V)}.\]
\end{theo}

This theorem may prove useful for investigating models that arise in biology when taking into account the time periodicity of the environment.
An example is given in~\cite[Section~6]{PG12} which leads to a periodic growth-fragmentation equation, and the Floquet eigenvalue is compared to some time averages of Perron's.

\medskip

Beyond the periodic case, several extensions to the fully non-homogeneous setting are expected, in the same vein as \cite{BCG17}.
We can now relax the ``coming down from infinity'' property imposed by the generalized Doeblin condition of~\cite{BCG17} and thus capture a larger class of non-autonomous linear PDEs.
This would allow extending some ergodic results of optimal control problems, as the ones in~\cite{CGG}, to the infinite dimension.
Similarly, let us recall that the expectation of a branching process yields the first moment semigroup, which usually drives the extinction of the process (criticality) and provides its deterministic renormalization (Kesten-Stigum theorem).
The method of this paper should provide a powerful tool to analyse the first moment semigroup of a branching process with infinite number of types, including in varying environments, see \cite{B15,BUS, marguet17} for some motivations in population dynamics and queuing systems. We also mention that time inhomogeneity provides a natural point of view to deal with non-linearity in large population approximations of systems with interaction. These points should be partially addressed in forthcoming works.

\section{Appendix}
\subsection{Conservative operators}
\label{sect:conservatif}

The following useful result is a direct generalization of the contraction result of Hairer and Mattingly~\cite{HM11}, \pg{inspired from the pioneering works of Meyn and Tweedie, see~\cite{MT}.}
Let us consider a positive operator $P$ acting both on bounded measurable functions $f:\X\to\R$ on the right and on bounded measures $\mu$ of finite mass on the left, and such that $(\mu P)f=\mu(Pf).$
Note that the right action of $P$ extends trivially to any measurable function $f:\X\to[0,+\infty].$
We assume that $P$ is  conservative in the sense that $P\1=\1$, or in other words,  if $\mu$ is a probability measure, then so does $\mu P.$

\begin{theo}
\label{th:ergo}
%\begin{hypo}\label{assu:auxi}
Assume that there exist two measurable functions $\mathscr V,\mathscr W : \mathcal{X} \to [0,\infty)$, $(\mathfrak a,\mathfrak b) \in(0,1)^2$, $\mathfrak c>0$, $\mathfrak R>2\mathfrak c/(1-\mathfrak a)$ and a probability measure $\nu$ on $\X$
such that:
\begin{itemize}
\item for all $x\in \mathcal{X}$,
\begin{equation}\label{assu:auxi-lyap}%\tag{A1}
P \mathscr V(x)\leq \mathfrak a \mathscr W  +\mathfrak c,
\end{equation}
\item for all $x\in \{\mathscr W \leq \mathfrak R\}$, 
\begin{equation}\label{assu:auxi-smallset}%\tag{A2}
 \delta_x P \geq \mathfrak b \nu.
\end{equation} 
\end{itemize}	 
Then, there exist $\mathfrak y\in (0,1)$ and $\kappa>0$ such that for all probability measures $\mu_1,\mu_2$,
\begin{align*}%\label{eq:ergo_traj}%\tag{A3}
\left\|\mu_1 P-\mu_2 P\right\|_{\M(1+\kappa\mathscr V)}\leq \mathfrak y  \left\|\mu_1-\mu_2\right\|_{\M(1+\kappa \mathscr W)}.
\end{align*}
In particular, for any $\mathfrak b'\in(0,\mathfrak b)$ and $\mathfrak a'\in(\mathfrak a+2\mathfrak c/\mathfrak R,1)$, one can choose 
$$
\kappa=\frac{\mathfrak b'}{\mathfrak c}, \qquad \mathfrak y= \max\left\{1-(\mathfrak b-\mathfrak b'),\frac{2+\kappa \mathfrak R \mathfrak a'}{2+\kappa \mathfrak R}\right\}.
$$ 
\end{theo}

Usually, for conservative semigroups and Markov chains \cite{HM11,MT}, Theorem~\ref{th:ergo} is stated and used  with one single  function $\mathscr V=\mathscr W=V$.
\pg{In this case, it yields the existence of a unique invariant probability measure and geometric convergence of the iterates $P^n$ to the projection on this invariant measure when $n$ goes to infinity, for the operator norm associated to $\M(1+\kappa V)$.

It is worth pointing out that, although they are similar, the assumptions~\eqref{assu:auxi-lyap} and~\eqref{assu:auxi-smallset} differ from the conditions~\eqref{eq:Lyap_cons} and~\eqref{eq:Doeb_cons} given in the introduction.
Actually, under the hypothesis that $\nu$ is supported by the set $K$, the proof of Lemma~\ref{lem:minorization} applied to $P$ ensures that if $P$ satisfies~\eqref{assu:auxi-lyap} and~\eqref{assu:auxi-smallset} then there is an integer $n_0$ such that $P^{n_0}$ verifies~\eqref{eq:Lyap_cons} and~\eqref{eq:Doeb_cons}, thus guaranteeing the existence of a unique invariant probability measure for $P$ and its geometric stability.
}

\subsection{Localization argument}\label{app:localization}
We detail here how the drift conditions \eqref{LLV}-\eqref{LLp} ensure~\eqref{semiLV}-\eqref{semiLpsi}, namely
\begin{align*}
\E_n[V(X_t)] & \leq V(x)+\int_0^t \E_n[(aV+\zeta\psi)(X_s)]ds, \\
\psi(x)+ \int_0^t \E_n[b \psi(X_s)]ds & \leq \E_x[\psi(X_t)]\leq \psi(x)+ \int_0^t \E_n[\xi\psi(X_s)]ds,
\end{align*}
for all $n\geq 1$ and $t\geq 0$.
Following \cite{MTIII}, for $m\geq 1$, we let $T_m=\inf\{ t>0  : X_t \geq m\}$ and $(X^m_t)_{t\geq 0}$ be the Markov process defined by
$
X^m_t = X_{t} \mathbf{1}_{t < T_m}.
$
We extend functions $V,\psi$ on $\mathbb{N}$ by setting $V(0)=\psi(0)=0$. Using~\eqref{LLV}, \eqref{LLp},  $V(0)\leq V(m)$ and
writing  $O_m=\{0,1, \ldots m-1\}$, 
the strong generator $\LL^m$ of $X^m$ satisfies 
\begin{eqnarray*}
&&\LL^m V \leq aV+\zeta\psi \ \text{ and  } \ \LL^m\psi\leq \xi\psi \quad \text{ on } O_m \\
%=\{0,1, \ldots m-1\}, \\
&&\LL^m \psi(m-1) = \LL \psi(m-1)-b \psi(m) \geq %b\psi(m-1) - b\psi(m) =
b(1-\eta)\psi(m-1)  \\ %- b\eta\psi(m-1)$$ and
&& \LL^m\psi\geq b\psi \quad \text{ on } O_{m-1}.
\end{eqnarray*}
First, using $\LL^m V \leq (a+\zeta) V$ on $O_m$  and $V(n)\rightarrow \infty$ as $n\rightarrow \infty$, \cite[Theorem 2.1]{MTIII} ensures that $\lim_{m\to \infty} T_m=\infty$ and
$$
\mathbb{E}_n[V(X_t)] \leq e^{(a+\zeta)t} V(n)
$$
for every $n\in \mathbb N$.
Second $\LL ^{m}\psi \leq \xi \psi$ on $O_m$  and $\psi$ is bounded on $O_m$. Using that
$X^m$ coincides with $X$ on $[0,T_m)$,  Fatou's lemma and Kolmogorov equation give
\begin{align*}
\mathbb{E}_n[\psi(X_t)] 
&\leq \liminf_{m \to \infty} \mathbb{E}_x[\psi(X^m_{t})]\\
&=\psi(n) + \liminf_{m \to\infty}  \mathbb{E}_n\left[\int_0^{t\wedge T_m} \LL^{m}\psi(X_s) ds\right]\leq \psi(x) + \xi\int_0^{t} \mathbb{E}_n[ \psi(X_s) ] ds.
\end{align*} 
Moreover $\psi(X^m_t) = \mathbf{1}_{t\leq T_m} \psi(X_t) \leq \psi(X_t)$ and $X^m_t\rightarrow X_t$ as $m\rightarrow \infty$.
Using   $\LL^{m}\psi\geq b\psi - b\eta \psi \mathbf{1}_{m-1}$ on $O_{m}$ and bounded convergence twice yields $\E(\int_0^t\psi(X_s)1_{X_s=m-1}ds)\rightarrow 0$ as $m\rightarrow\infty$ and
\begin{align*}
\mathbb{E}_n[\psi(X_t)] 
&= \lim_{m \to \infty} \mathbb{E}_n[\psi(X^m_{t})]\\
&=\psi(n) + \lim_{m \to\infty}  \mathbb{E}_n\left[\int_0^{t\wedge T_m} \LL^{m}\psi(X_s) ds\right]\geq \psi(x) + b\int_0^{t} \mathbb{E}_n[ \psi(X_s) ] ds.
\end{align*} 
Using Fatou's lemma as above for $V$ ends the proof of~\eqref{semiLV}-\eqref{semiLpsi}.

\subsection{The growth-fragmentation semigroup}\label{app:GFsemigroup}

We give here the details of the construction of the growth-fragmentation semigroup and prove its basic properties, along the lines of~\cite{Dumont2017,Gabriel2018,GabrielMartin19}.
%For a subset $\Omega$ of $\R^n$, we denote by $\B_{loc}(\Omega)$ the space of measurable functions  $\Omega\to\R$ which are locally bounded on $\overline\Omega$.
For a function $f\in\B_{loc}([0,\infty))$,
we define the family $(M_tf)_{t\geq0}\subset\B_{loc}([0,\infty))$ through the Duhamel formula
\[M_tf(x)=f(x+t)\e^{-\int_0^tB(x+s)ds}+\int_0^t\e^{-\int_0^s B(x+s')ds'}B(x+s)\int_0^1M_{t-s}f(z(x+s))\wp(dz)\,ds.\]
We first prove that this indeed defines uniquely the family $(M_tf)_{t\geq0}.$
Then, we verify that the associated family $(M_t)_{t\geq0}$ is a semigroup of linear operators, which provides the unique solution to the growth-fragmentation~\eqref{eq:GF} on the space $\M(V)$ with $V(x)=1+x^k,$ $k>1.$
Finally we provide some useful monotonicity properties for this semigroup, which are consequences of the monotonicity assumption on $B.$ 

\begin{lem}\label{lm:fixedpoint}
For any $f\in\B_{loc}([0,\infty))$ there exists a unique $\bar f\in\B_{loc}([0,\infty)^2)$ such that for all $t\geq0$ and $x\geq 0$
\[\bar f(t,x)=f(x+t)\e^{-\int_0^tB(x+s)ds}+\int_0^t\e^{-\int_0^s B(x+s')ds'}B(x+s)\int_0^1\bar f(t-s,z(x+s))\wp(dz)\,ds.\]
Moreover if $f$ is nonnegative/continuous/continuously differentiable, then so does $\bar f.$
In the latter case $\bar f$ satisfies the partial differential equation
\[\partial_t\bar f(t,x)=\LL\bar f(t,x)=\partial_x\bar f(t,x)+B(x)\bigg[\int_0^1\bar f(t,zx)\wp(dz)-\bar f(t,x)\bigg].\]
\end{lem}

\begin{proof}
Let $f\in\B_{loc}([0,\infty)$ and define on $\B_{loc}([0,\infty)^2)$ the mapping $\Gamma$ by
\[\Gamma g(t,x)=f(x+t)\e^{-\int_0^tB(x+s)ds}+\int_0^t\e^{-\int_0^s B(x+s')ds'}B(x+s)\int_0^1g(t-s,z(x+s))\wp(dz)\,ds.\]
Now for $T,A>0$ define the set $\Omega_{T,A}=\{(t,x)\in[0,T]\times[0,\infty),\ x+t<A\}$
and denote by $\B_b(\Omega_{T,A})$ the Banach space of bounded measurable functions on $\Omega_{T,A},$ endowed with the supremum norm $\|\cdot\|_\infty.$
Clearly $\Gamma$ induces a mapping $\B_b(\Omega_{T,A})\to\B_b(\Omega_{T,A}),$ still denoted by $\Gamma.$
To build a fixed point of $\Gamma$ in $\B_{loc}([0,\infty)^2)$ we prove that it admits a unique fixed point in any $\B_b(\Omega_{A,A}).$

\smallskip

Let $A>0$ and $T<1/(\wp_0B(A)).$
For any $g_1,g_2\in\B_b(\Omega_{T,A})$ we have
\[\|\Gamma g_1-\Gamma g_2\|_\infty\leq\wp_0T B(A)\|g_1-g_2\|_\infty\]
and $\Gamma$ is a contraction.
The Banach fixed point theorem then guarantees the existence of a unique fixed point $g_{T,A}$ of $\Gamma$ in $\B_b(\Omega_{T,A}).$
The same argument on $\Omega_{T,A-T}$
% with $f$ being replaced by $g_{T,A}(T,\cdot)$
 ensures that $g_{T,A}$
can be extended into a unique fixed point $g_{2T,A}$ of $\Gamma$ on $\Omega_{2T,A}.$
Iterating the procedure we finally get a unique fixed point $g_A$ of $\Gamma$ in $\B_b(\Omega_{A,A}).$

\smallskip

For $A'>A>0$ we have ${g_{A'}}_{|\Omega_A}=g_A$ by uniqueness of the fixed point in $\B_b(\Omega_A),$ and we can define $\bar f$ by setting $\bar f_{|\Omega_A}=g_A$ for any $A>0.$
Clearly the function $\bar f$ thus defined is the unique fixed point of $\Gamma$ in $\B_{loc}([0,\infty)^2).$
Since $\Gamma$ preserves the closed cone of nonnegative functions if $f$ is nonnegative, the fixed point $\bar f$ necessarily belongs to this cone when $f$ is nonnegative.
Similarly, the space $C([0,\infty)^2)$ of continuous functions being a closed subspace of $\B_{loc}([0,\infty)^2),$ the fixed point $\bar f$ is continuous when $f$ is so.

\smallskip

Consider now that $f$ is continuously differentiable on $[0,\infty)$.
The space $C^1([0,\infty)^2)$ is not closed in $\B_{loc}([0,\infty)^2)$ for the norm $\|\cdot\|_\infty.$
For proving the continuous differentiability of $\bar f$ we repeat the fixed point argument in
$\{g\in C^1(\Omega_{T,A}),\ g(0,\cdot)=f\},$ endowed with the norm $\|g\|_{C^1}=\|g\|_\infty+\|\partial_tg\|_\infty+\|\partial_xg\|_\infty$.
Differentiating $\Gamma g$ with respect to $t$ we get
\[\partial_t\Gamma g(t,x)=\LL f(x+t)\e^{-\int_0^tB(x+s)ds}+\int_0^t\e^{-\int_0^s B(x+s')ds'}B(x+s)\int_0^1\partial_tg(t-s,z(x+s))\wp(dz)\,ds\]
and differentiating the alternative formulation
\[\Gamma g(t,x)=f(x+t)\e^{-\int_x^{x+t}B(y)dy}+\int_x^{x+t}\e^{-\int_x^y B(y')dy'}B(y)\int_0^1g(t+x-y,zy)\wp(dz)\,dy\]
with respect to $x$ we obtain
\begin{align}
\partial_x\Gamma g(t,x)&
=\LL f(x+t)\e^{-\int_x^{x+t}B(y)dy}+B(x)\Big(f(x+t)\e^{-\int_x^{x+t}B(y)dy}-\int_0^1g(t,zx)\wp(dz)\Big)\nonumber\\
&\qquad+B(x)\int_x^{x+t}\e^{-\int_x^y B(y')dy'}B(y)\int_0^1g(t+x-y,zy)\wp(dz)\,dy\nonumber\\
&\qquad\qquad+\int_x^{x+t}\e^{-\int_x^y B(y')dy'}B(y)\int_0^1\partial_tg(t+x-y,zy)\wp(dz)\,dy\nonumber\\
&=\Big[\LL f(x+t)+B(x)f(x+t)-B(x)\int_0^1f(zx)\wp(dz)\Big]\e^{-\int_x^{x+t}B(y)dy}\nonumber\\
&\qquad+\int_x^{x+t}\e^{-\int_x^y B(y')dy'}\big(B(y)-B(x)\big)\int_0^1\partial_tg(t+x-y,zy)\wp(dz)\,dy\nonumber\\
&\qquad\qquad+B(x)\int_x^{x+t}\e^{-\int_x^y B(y')dy'}\int_0^1\partial_xg(t+x-y,zy)z\wp(dz)\,dy.\nonumber
\end{align}
On the one hand using the second expression of $\partial_x\Gamma g(t,x)$ above we deduce that for $g_1,g_2\in C^1(\Omega_{T,A})$ such that $g_1(0,\cdot)=g_2(0,\cdot)=f$ we have
\begin{align*}
\|\Gamma g_1-\Gamma g_2\|_{C^1}
%&\leq\wp_0T B(A)\|g_1-g_2\|_\infty+2\wp_0T B(A)\|\partial_tg_1-\partial_tg_2\|_\infty+TB(A)\|\partial_xg_1-\partial_xg_2\|_\infty\\
&\leq2\wp_0TB(A)\|g_1-g_2\|_{C^1}.
\end{align*}
Thus $\Gamma$ is a contraction for $T<1/(2\wp_0TB(A))$ and this guarantees that the fixed point $\bar f$ necessarily belongs to $C^1([0,\infty)^2).$
On the other hand using the first expression of $\partial_x\Gamma g(t,x)$ we have
\[\partial_t\Gamma g(t,x)-\partial_x\Gamma g(t,x)=B(x)\bigg[\int_0^1g(t,zx)\wp(dz)-\Gamma g(t,x)\bigg]\]
and accordingly the fixed point satisfies $\partial_t\bar f=\LL\bar f.$
\end{proof}

With Lemma~\ref{lm:fixedpoint} at hand we can define for any $t\geq0$ the mapping $M_t$ on $\B_{loc}(0,\infty)$ by setting
\[M_tf(x)=\bar f(t,x).\]

\begin{prop}\label{prop:Mtdroit}
The family $(M_t)_{t\geq0}$ defined above is a positive semigroup of linear operators on $\B_{loc}([0,\infty))$.
If $f\in C^1([0,\infty))$ then the function $(t,x)\mapsto M_tf(x)$ is continuously differentiable and satisfies
\[\partial_tM_tf(x)=\LL M_tf(x)=M_t\LL f(x).\]
Additionally for any $k>1$ the space $\B(V)$ with $V(x)=1+x^k$ is invariant under $(M_t)_{t\geq0},$
and for all $t\geq0$ the restriction of $M_t$ to $\B(V)$ is a bounded operator.
\end{prop}

\begin{proof}
The linearity and the semigroup property readily follow from the uniqueness of the fixed point in Lemma~\ref{lm:fixedpoint}.
The positivity and the stability of $C^1([0,\infty))$ are direct consequences of Lemma~\ref{lm:fixedpoint}, as well as the relation $\partial_tM_tf=\LL M_tf.$
For getting the second one $\partial_tM_tf=M_t\LL f,$ it suffices to remark from the computation of $\partial_t\Gamma g$ in the proof of Lemma~\ref{lm:fixedpoint} that $\partial_tM_tf$ is the unique fixed point of $\Gamma$ with initial data $\LL f.$
For the invariance of $\B(V)$ we compute
\[\LL V(x)=1+kx^{k-1}+(\wp_0-1)B(x)+(\wp_k-1)B(x)x^k\]
which is bounded on $[0,\infty)$ since $(\wp_0-1)B(x)+(\wp_k-1)B(x)x^k\leq0$ when $x\geq\big(\frac{\wp_0-1}{1-\wp_k}\big)^{\frac1k}.$
We deduce that there exists $C>0$ such that $\LL V\leq C V$ and since $V\in C^1([0,\infty))$ we get
\[M_tV(x)=V(x)+\int_0^tM_s(\LL V)(x)\,ds\leq \e^{Ct}V(x).\]
%Finally by positivity of $M_t$ we have for any $f\in\B(V)$
%\[|M_tf|\leq M_t|f|\leq \|f\|_{\B(V)}M_tV\leq \e^{Ct}\|f\|_{\B(V)}V\]
%which yields $\|M_tf\|_{\B(V)}\leq \e^{Ct}\|f\|_{\B(V)}.$
Positivity of $M_t$ then yields $\|M_tf\|_{\B(V)}\leq \e^{Ct}\|f\|_{\B(V)}.$
\end{proof}

Now we define, for $t\geq0$ and $\mu\in\M_+(V)$, the positive measure $\mu M_t$ by setting for any measurable set $A\subset[0,\infty)$
\[(\mu M_t)(A)=\mu(M_t\1_A).\]
%The axioms of a positive measure are satisfied.
%Clearly $(\mu M_t)(\emptyset)=0$ and $(\mu M_t)(A\cup B)=(\mu M_t)(A)+(\mu M_t)(B)$ when $A$ and $B$ are two disjoint measurable sets.
%If $(A_n)_{n\geq0}$ is an increasing sequence of measurable sets then by positivity of $M_t$ the sequence ${(M_t\1_{A_n})}_{n\geq0}$ is an increasing sequence of measurable functions bounded by $M_tV.$
%Passing to the limit in the Duhamel formula we deduce from the uniqueness of the solution that the pointwise limit of ${(M_t\1_{A_n})}_{n\geq0}$ is $M_t\1_A,$ where $A=\bigcup_{n\geq0}A_n.$
%We conclude by dominated or monotone convergence theorem that $(\mu M_t)(A)=\lim_{n\to\infty}(\mu M_t)(A_n).$
%By construction we have $(\mu M_t)(f)=\mu(M_tf)$ for any positive measurable function $f,$
%and since $\B(V)$ is invariant under $M_t$ the measure $\mu M_t$ belongs to $\M_+(V).$
Then, for $\mu\in\M(V)$ we define $\mu M_t\in\M(V)$ as the equivalence class of $(\mu_+M_t,\mu_-M_t).$

\begin{prop}\label{prop:Mtgauche}
The family $(M_t)_{t\geq0}$ defined above is a positive semigroup of bounded linear operators on $\M(V)$.
Moreover for any $\mu\in\M(V)$ the family $(\mu M_t)_{t\geq0}$ is solution to Equation~\eqref{eq:GF} in the sense that for all $f\in C^1_c([0,\infty))$ and all $t\geq0$
\[(\mu M_t)(f)=\mu (f)+\int_0^t(\mu M_s)(\LL f)\,ds.\]
\end{prop}

\begin{proof}
Let $\mu\in\M(V)$ and $f\in C^1_c([0,\infty))$.
From Proposition~\ref{prop:Mtdroit} we know that $\partial_tM_tf=M_t\LL f$
which gives by integration in time
\[M_tf(x)=f(x)+\int_0^tM_s\LL f(x)\,ds=f(x)+\int_0^tM_s(f'-Bf)(x)\,ds+\int_0^tM_s\mathcal F f(x)\,ds\]
for all $x\geq0$, where we have set
\[\mathcal F f(x)=B(x)\int_0^1f(zx)\wp(dz).\]
Since $f'-Bf\in\B(V)$ we have $|M_s(f'-Bf)|\leq\left\|f'-Bf\right\|_{\B(V)}\e^{Cs}V$ and Fubini's theorem ensures that
\[\mu\Big(\int_0^tM_s(f'-Bf)\,ds\Big)=\int_0^t(\mu M_s)(f'-Bf)\,ds.\]
The last term deserves a bit more attention since $\mathcal Ff$ can be not bounded by $V.$
Consider $g\in C^1_c([0,\infty))$ such that $g\geq|f|.$
By positivity of $M_s$ and $\mathcal F$ we have $|M_s\mathcal F f|\leq M_s\mathcal F|f|\leq M_s\mathcal Fg$ and since $g\in C^1_c(0,\infty)$
\[\mu_\pm\Big(\int_0^tM_s\mathcal Fg\,ds\Big)=\mu_{\pm}\Big(M_tg-g-\int_0^tM_s(g'-Bg)\,ds\Big)<+\infty.\]
Then $(s,x)\mapsto M_s\mathcal F f(x)$ is $(ds\times\mu)$-integrable and Fubini's theorem yields
\[\mu\Big(\int_0^tM_s\mathcal Ff\,ds\Big)=\int_0^t(\mu M_s)(\mathcal Ff)\,ds,\]
which ends the proof.
\end{proof}

We end this appendix by giving some monotonicity results on $(M_t)_{t\geq0},$ which are useful for verifying~\ref{A4} in Section~\ref{EDP}.
They are valid under the monotonicity assumption we made on the fragmentation rate $B.$

\begin{lem} \label{onsennuiecest}
i) For any $x\geq 0$, $t\mapsto M_t\psi(x)$ is increasing. \\
%\label{lm:croissance_en_x}
ii) For any $t\geq0,$  $x\mapsto M_t\psi(x)$ is increasing. \\
iii) %\label{lm:croissance_croisee}
For any $T>0,$ $z\in[0,1]$ and $x\geq0,$  $t\mapsto M_t\psi\big(z(x+T-t)\big)$  increases on $[0,T].$
\end{lem}

\begin{proof} The point {\it i)} readily follows from $\partial_tM_t\psi=M_t(\LL\psi),$ since $M_t$ is positive and $2\,\LL\psi(x)=1+(\wp_0-1)B(x)\geq0.$\\
Let us prove {\it ii)}.
Define $f(t,x)=\partial_xM_t\psi(x)$ which satisfies
\begin{align*}
\partial_tf(t,x)&=\partial_xf(t,x)-B(x)f(t,x)+B(x)\int_0^1f(t,zx)z\wp(dz)+C(x),
\end{align*}
with $C(x)=-B'(x)M_t\psi(x)+B'(x)\int_0^1M_t\psi(zx)\,\wp(dz)$.
Since $\partial_tM_t\psi(x)=\LL M_t\psi(x),$ $\partial_t M_t\psi(x)\geq0,$ and $B'\geq0,$ we have
\[C(x)=\frac{B'(x)}{B(x)}\big(\partial_tM_t\psi(x)-\partial_xM_t\psi(x)\big)
\geq-\frac{B'(x)}{B(x)}f(t,x)\]
and as a consequence
\[\partial_tf(t,x)\geq\A f(t,x):=\partial_xf(t,x)-\bigg(B(x)+\frac{B'(x)}{B(x)}\bigg)f(t,x)+B(x)\int_0^1f(t,zx)z\wp(dz).\]
Similarly to $\LL$ the operator $\A$ generates a positive semigroup $(U_t)_{t\geq0}.$
It is a standard result that it enjoys the following maximum principle
$\partial_t f(t,x)\geq\A f(t,x) \implies f(t,x)\geq U_tf_0(x)$,
where $f_0=f(0,\cdot).$
Since $f(0,x)=\psi'(x)=\frac12\geq0$ we deduce from the positivity of $U_t$ that $f(t,x)\geq0$ for all $t,x>0,$ which ends 
 the proof of {\it ii)}.

We turn to the proof of {\it iii)}.
The case $z=0$ corresponds to {\it ii)}
and we consider now $z\in(0,1].$
Setting $f(t,x)=M_t\psi\big(z(x+T-t)\big)$ we have using {\it ii)},
$\partial_xf(t,x)=z\,\partial_xM_t\psi\big(z(x+T-t)\big)\geq0$
and
\begin{align*}\label{eq:dtf}\nonumber
\partial_tf(t,x)&=(\partial_tM_t\psi)\big(z(x+T-t)\big)-z\,(\partial_xM_t\psi)\big(z(x+T-t)\big)\\
%&=(1-z)\,\partial_xM_t\psi\big(z(x+T-t)\big)-B\big(z(x+T-t)\big)M_t\psi\big(z(x+T-t)\big)\\\nonumber
%&\hspace{20mm}+B\big(z(x+T-t)\big)\int_0^1M_t\psi\big(z'z(x+T-t)\big)\,\wp(dz')\\\nonumber
=&\frac{1-z}{z}\,\partial_xf(t,x)-B\big(z(x+T-t)\big)\bigg[f(t,x)-\int_0^1f\big(t,z'x-(1-z')(T-t)\big)\,\wp(dz')\bigg].
\end{align*}
Now we define $g(t,x)=\partial_tf(t,x)$ and get, by differentiating the above equation with respect to~$t$ and %use again \eqref{eq:dtf} to get
%\begin{align*}
%\partial_tg(t,x)&=\frac{1-z}{z}\,\partial_xg(t,x)-B\big(z(x+T-t)\big)g(t,x)\\
%&\hspace{36mm}+B\big(z(x+T-t)\big)\int_0^1g\big(t,z'x-(1-z')(T-t)\big)\,\wp(dz')\\
%&\quad+z B'\big(z(x+T-t)\big)f(t,x)-z B'\big(z(x+T-t)\big)\int_0^1f\big(t,z'x-(1-z')(T-t)\big)\,\wp(dz')\\
%&\quad+B\big(z(x+T-t)\big)\int_0^1(1-z')\partial_xf\big(t,z'x-(1-z')(T-t)\big)\,\wp(dz').
%\end{align*}
%\begin{align*}
%\partial_tg(t,x)&=\frac{1-z}{z}\,\partial_xg(t,x)-B\big(z(x+T-t)\big)\bigg[g(t,x)-\int_0^1g\big(t,z'x-(1-z')(T-t)\big)\,\wp(dz')\bigg]\\
%&\qquad+z \frac{B'}{B}\big(z(x+T-t)\big)\Big(\frac{1-z}{z}\partial_xf(t,x)-g(t,x)\Big)\\
%&\qquad\qquad+B\big(z(x+T-t)\big)\int_0^1(1-z')\partial_xf\big(t,z'x-(1-z')(T-t)\big)\,\wp(dz').
%\end{align*}
using the positivity of $\partial_xf,$ $B$ and $B'$, %we finally obtain
\begin{align*}
\partial_tg(t,x)&\geq\frac{1-z}{z}\,\partial_xg(t,x)-\Big(B+z\frac{B'}{B}\Big)\big(z(x+T-t)\big)g(t,x)\\
&\hspace{36mm}+B\big(z(x+T-t)\big)\int_0^1g\big(t,z'x-(1-z')(T-t)\big)\,\wp(dz').
\end{align*}
Since $g(0,x)=\frac{1-z}{2}+\frac{\wp_0-1}{2}B\big(z(x+T)\big)\geq0$ we deduce from the maximum principle that \mbox{$g(t,x)\geq0$.}
\end{proof}

\section*{Acknowledgments}
The authors are grateful to the anonymous referees for their suggestions  and remarks which have improved the previous version of the manuscript.
V.B., B.C., and A.M. have received the support of the Chair ``Mod\'elisation Math\'ematique et Biodiversit\'e'' of VEOLIA-Ecole Polytechni\-que-MnHn-FX.
The  authors have been also supported by ANR projects, funded by the French Ministry of Research:
B.C. by ANR MESA (ANR-18-CE40-006),
V.B. by ANR ABIM (ANR-16-CE40-0001) and ANR CADENCE (ANR-16-CE32-0007), and A.M. by ANR MEMIP (ANR-16-CE33-0018),
and the four authors by ANR NOLO (ANR-20-CE40-0015).

%\bibliographystyle{abbrv}
%\bibliography{ref}

\begin{thebibliography}{10}

\bibitem{Nagel86}
W.~Arendt, A.~Grabosch, G.~Greiner, U.~Groh, H.~P. Lotz, U.~Moustakas,
  R.~Nagel, F.~Neubrander, and U.~Schlotterbeck.
\newblock {\em One-parameter semigroups of positive operators}, volume 1184 of
  {\em Lecture Notes in Mathematics}.
\newblock Springer-Verlag, Berlin, 1986.

\bibitem{AT12}
A.~{Asselah} and M.-N. {Thai}.
\newblock {A note on the rightmost particle in a Fleming-Viot process}.
\newblock Preprint, arXiv:1212.4168, 2012.

\bibitem{BCG13}
D.~Balagu\'e, J.~A. Ca\~nizo, and P.~Gabriel.
\newblock Fine asymptotics of profiles and relaxation to equilibrium for
  growth-fragmentation equations with variable drift rates.
\newblock {\em Kinet. Relat. Models}, 6(2):219--243, 2013.

\bibitem{Banasiak}
J.~Banasiak and L.~Arlotti.
\newblock {\em Perturbations of Positive Semigroups with Applications}.
\newblock Springer Monographs in Mathematics. Springer-Verlag, London, 2006.

\bibitem{BLL19}
J.~Banasiak, W.~Lamb, and P.~Lauren\c{c}ot.
\newblock {\em Analytic Methods for Coagulation-Fragmentation Models, Volume
  I}.
\newblock Chapman \& Hall/CRC Monographs and Research Notes in Mathematics.
  2019.

\bibitem{BPR}
J.~Banasiak, K.~Pich{\'o}r, and R.~Rudnicki.
\newblock Asynchronous exponential growth of a general structured population
  model.
\newblock {\em Acta Appl. Math.}, 119:149--166, 2012.

\bibitem{B15}
V.~{Bansaye}.
\newblock {Ancestral lineage and limit theorems for branching Markov chains}.
\newblock {\em J. Theor. Probab.}, 2018.

\bibitem{BUS}
V.~{Bansaye} and A.~{Camanes}.
\newblock {Queueing for an infinite bus line and aging branching process}.
\newblock {\em Queueing Syst.}, 2018.

\bibitem{BCG17}
V.~Bansaye, B.~Cloez, and P.~Gabriel.
\newblock Ergodic {B}ehavior of {N}on-conservative {S}emigroups via
  {G}eneralized {D}oeblin's {C}onditions.
\newblock {\em Acta Appl. Math.}, 166:29--72, 2020.

\bibitem{BDMT}
V.~Bansaye, J.-F. Delmas, L.~Marsalle, and V.~C. Tran.
\newblock Limit theorems for {M}arkov processes indexed by continuous time
  {G}alton-{W}atson trees.
\newblock {\em Ann. Appl. Probab.}, 21(6):2263--2314, 2011.

\bibitem{BCGMZ}
J.-B. Bardet, A.~Christen, A.~Guillin, F.~Malrieu, and P.-A. Zitt.
\newblock Total variation estimates for the {TCP} process.
\newblock {\em Electron. J. Probab.}, 18(10):1--21, 2013.

\bibitem{BCP}
M.~Benaim, B.~Cloez, and F.~Panloup.
\newblock Stochastic approximation of quasi-stationary distributions on compact
  spaces and applications.
\newblock {\em Ann. Appl. Probab.}, 28(4):2370--2416, 2018.

\bibitem{BG1}
{\'E}.~Bernard and P.~Gabriel.
\newblock Asymptotic behavior of the growth-fragmentation equation with bounded
  fragmentation rate.
\newblock {\em J. Funct. Anal.}, 272(8):3455--3485, 2017.

\bibitem{BG2}
{\'E}.~Bernard and P.~Gabriel.
\newblock Asynchronous exponential growth of the growth-fragmentation equation
  with unbounded fragmentation rate.
\newblock {\em J. Evol. Equ.}, 20(2):375--401, 2020.

\bibitem{Bertoin18}
J.~Bertoin.
\newblock On a {F}eynman-{K}ac approach to growth-fragmentation semigroups and
  their asymptotic behaviors.
\newblock {\em J. Funct. Anal.}, 277(11):108270, 29, 2019.

\bibitem{BW18}
J.~Bertoin and A.~R. Watson.
\newblock A probabilistic approach to spectral analysis of growth-fragmentation
  equations.
\newblock {\em J. Funct. Anal.}, 274(8):2163--2204, 2018.

\bibitem{B57}
G.~Birkhoff.
\newblock Extensions of {J}entzsch's theorem.
\newblock {\em Trans. Amer. Math. Soc.}, 85:219--227, 1957.

\bibitem{Bouguet}
F.~Bouguet.
\newblock A probabilistic look at conservative growth-fragmentation equations.
\newblock In {\em S\'eminaire de {P}robabilit\'es {XLIX}}, volume 2215, pages
  57--74. Springer, Cham, 2018.

\bibitem{BGP19}
J.~Broda, A.~Grigo, and N.~P. Petrov.
\newblock Convergence rates for semistochastic processes.
\newblock {\em Discrete Contin. Dyn. Syst. Ser. B}, 24(1), 2019.

\bibitem{CanizoYoldas}
J.~A. Ca\~{n}izo and H.~Yolda\c{s}.
\newblock Asymptotic behaviour of neuron population models structured by
  elapsed-time.
\newblock {\em Nonlinearity}, 32(2):464--495, 2019.

\bibitem{CCM10}
M.~J. C\'{a}ceres, J.~A. Ca\~{n}izo, and S.~Mischler.
\newblock Rate of convergence to self-similarity for the fragmentation equation
  in ${L}^1$ spaces.
\newblock {\em Comm. Appl. Ind. Math.}, 1(2):299--308, 2010.

\bibitem{CCM11}
M.~J. C\'{a}ceres, J.~A. Ca\~{n}izo, and S.~Mischler.
\newblock Rate of convergence to an asymptotic profile for the self-similar
  fragmentation and growth-fragmentation equations.
\newblock {\em J. Math. Pures Appl.}, 96(4):334 -- 362, 2011.

\bibitem{CGG}
V.~Calvez, P.~Gabriel, and S.~Gaubert.
\newblock Non-linear eigenvalue problems arising from growth maximization of
  positive linear dynamical systems.
\newblock In {\em 53rd IEEE Conference on Decision and Control}, pages
  1600--1607, 2014.

\bibitem{ChafaiMalrieuParoux}
D.~Chafa\"{i}, F.~Malrieu, and K.~Paroux.
\newblock On the long time behavior of the {TCP} window size process.
\newblock {\em Stochastic Process. Appl.}, 120(8):1518--1534, 2010.

\bibitem{CV14}
N.~Champagnat and D.~Villemonais.
\newblock Exponential convergence to quasi-stationary distribution and
  {$Q$}-process.
\newblock {\em Probab. Theory Related Fields}, 164(1-2):243--283, 2016.

\bibitem{CV18}
N.~{Champagnat} and D.~{Villemonais}.
\newblock {General criteria for the study of quasi-stationarity}.
\newblock Preprint, arXiv:1712.08092, 2018.

\bibitem{C17}
B.~Cloez.
\newblock Limit theorems for some branching measure-valued processes.
\newblock {\em Adv. in Appl. Probab.}, 49(2):549--580, 2017.

\bibitem{CloezGabriel}
B.~Cloez and P.~Gabriel.
\newblock On an irreducibility type condition for the ergodicity of
  nonconservative semigroups.
\newblock {\em C. R. Math. Acad. Sci. Paris}, 358(6):733--742, 2020.

\bibitem{CMMS13}
P.~Collet, S.~Martinez, S.~M{\'e}l{\'e}ard, and J.~San~Mart{\'\i}n.
\newblock Stochastic models for a chemostat and long-time behavior.
\newblock {\em Adv. Appl. Probab.}, 45(3):822--836, 2013.

\bibitem{CMM13}
P.~Collet, S.~Mart{\'{\i}}nez, and J.~San~Mart{\'{\i}}n.
\newblock {\em Quasi-stationary distributions}.
\newblock Probability and its Applications (New York). Springer, Heidelberg,
  2013.

\bibitem{DautrayLions3}
R.~Dautray and J.-L. Lions.
\newblock {\em Mathematical analysis and numerical methods for science and
  technology. {V}ol.~3}.
\newblock Springer-Verlag, Berlin, 1990.

\bibitem{DaviesSimon}
E.~B. Davies and B.~Simon.
\newblock Ultracontractivity and the heat kernel for {S}chr\"{o}dinger
  operators and {D}irichlet {L}aplacians.
\newblock {\em J. Funct. Anal.}, 59(2):335--395, 1984.

\bibitem{DDGW}
T.~Debiec, M.~Doumic, P.~Gwiazda, and E.~Wiedemann.
\newblock Relative entropy method for measure solutions of the
  growth-fragmentation equation.
\newblock {\em SIAM J. Math. Anal.}, 50(6):5811--5824, 2018.

\bibitem{DelM04}
P.~Del~Moral.
\newblock {\em Feynman-{K}ac formulae}.
\newblock Probability and its Applications (New York). Springer-Verlag, New
  York, 2004.

\bibitem{DM13}
P.~Del~Moral.
\newblock {\em Mean field simulation for {M}onte {C}arlo integration}, volume
  126 of {\em Monographs on Statistics and Applied Probability}.
\newblock CRC Press, Boca Raton, FL, 2013.

\bibitem{MM02}
P.~Del~Moral and L.~Miclo.
\newblock On the stability of nonlinear {F}eynman-{K}ac semigroups.
\newblock {\em Ann. Fac. Sci. Toulouse Math. (6)}, 11(2):135--175, 2002.

\bibitem{DM04}
P.~Del~Moral and L.~Miclo.
\newblock On convergence of chains with occupational self-interactions.
\newblock {\em Proc. R. Soc. Lond. Ser. A Math. Phys. Eng. Sci.},
  460(2041):325--346, 2004.

\bibitem{DHT}
O.~Diekmann, H.~J. A.~M. Heijmans, and H.~R. Thieme.
\newblock On the stability of the cell size distribution.
\newblock {\em J. Math. Biol.}, 19(2):227--248, 1984.

\bibitem{D40}
W.~Doblin.
\newblock {\'El\'ements d'une th\'eorie g\'en\'erale des cha\^\i nes simples
  constantes de {M}arkoff}.
\newblock {\em Ann. \'Ecole Norm. (3)}, 57:61--111, 1940.

\bibitem{D57}
J.~L. Doob.
\newblock Conditional brownian motion and the boundary limits of harmonic
  functions.
\newblock {\em Bull. Soc. Math. France}, 85(4):431--458, 1957.

\bibitem{doob2012classical}
J.~L. Doob.
\newblock {\em Classical potential theory and its probabilistic counterpart:
  Advanced problems}, volume 262.
\newblock Springer Science \& Business Media, 2012.

\bibitem{DMPS}
R.~Douc, E.~Moulines, P.~Priouret, and P.~Soulier.
\newblock {\em Markov chains}.
\newblock Springer Series in Operations Research and Financial Engineering.
  Springer, Cham, 2018.

\bibitem{DGR}
V.~Dumas, F.~Guillemin, and P.~Robert.
\newblock {Limit results for Markovian models of TCP}.
\newblock In {\em GLOBECOM'01. IEEE Global Telecommunications Conference (Cat.
  No. 01CH37270)}, volume~3, pages 1806--1810. IEEE, 2001.

\bibitem{Dumont2017}
G.~Dumont and P.~Gabriel.
\newblock The mean-field equation of a leaky integrate-and-fire neural network:
  measure solutions and steady states.
\newblock {\em Nonlinearity}, 33:6381--6420, 2020.

\bibitem{EHK10}
J.~Engl{\"a}nder, S.~C. Harris, and A.~E. Kyprianou.
\newblock Strong law of large numbers for branching diffusions.
\newblock {\em Ann. Inst. Henri Poincar\'e Probab. Stat.}, 46(1):279--298,
  2010.

\bibitem{EPW}
H.~Engler, J.~Pr{\"u}ss, and G.~F. Webb.
\newblock Analysis of a model for the dynamics of prions. {II}.
\newblock {\em J. Math. Anal. Appl.}, 324(1):98--117, 2006.

\bibitem{ferre-these}
G.~Ferr\'e.
\newblock {\em Large Deviations Theory in Statistical Physics: Some Theoretical
  and Numerical Aspects}.
\newblock PhD thesis, Universit\'{e} Paris-Est, 2019.

\bibitem{FRS}
G.~Ferr\'e, M.~Rousset, and G.~Stoltz.
\newblock {More on the long time stability of Feynman-Kac semigroups}.
\newblock {\em Stoch. PDE: Anal. Comp.}, 2020.

\bibitem{Floquet}
G.~Floquet.
\newblock Sur les \'equations diff\'erentielles lin\'eaires \`a coefficients
  p\'eriodiques.
\newblock {\em Ann. Sci. \'Ecole Norm. Sup. (2)}, 12:47--88, 1883.

\bibitem{Frobenius}
G.~{Frobenius}.
\newblock {\"Uber Matrizen aus nicht negativen Elementen.}
\newblock {\em {Berl. Ber.}}, 1912:456--477, 1912.

\bibitem{PG12}
P.~Gabriel.
\newblock Long-time asymptotics for nonlinear growth-fragmentation equations.
\newblock {\em Commun. Math. Sci.}, 10(3):787--820, 2012.

\bibitem{Gabriel2018}
P.~Gabriel.
\newblock Measure solutions to the conservative renewal equation.
\newblock {\em ESAIM Proc. Surveys}, 62:68--78, 2018.

\bibitem{GabrielMartin19}
P.~Gabriel and H.~Martin.
\newblock Periodic asymptotic dynamics of the measure solutions to an equal
  mitosis equation.
\newblock {\em Ann. H. Lebesgue}, to appear, 2022.

\bibitem{GJ18}
P.~Groisman and M.~Jonckheere.
\newblock {Front propagation and quasi-stationary distributions for
  one-dimensional L{\'e}vy processes}.
\newblock {\em Electron. Commun. Probab.}, 23:11 pp., 2018.

\bibitem{H16}
M.~Hairer.
\newblock {\em Convergence of Markov Processes}.
\newblock 2016.
\newblock http://www.hairer.org/notes/Convergence.pdf.

\bibitem{HM11}
M.~Hairer and J.~C. Mattingly.
\newblock Yet another look at {H}arris' ergodic theorem for {M}arkov chains.
\newblock In {\em Seminar on {S}tochastic {A}nalysis, {R}andom {F}ields and
  {A}pplications {VI}}, volume~63 of {\em Progr. Probab.}, pages 109--117.
  Birkh\"auser/Springer Basel AG, Basel, 2011.

\bibitem{HW89}
A.~J. Hall and G.~C. Wake.
\newblock A functional-differential equation arising in modelling of cell
  growth.
\newblock {\em J. Austral. Math. Soc. Ser. B}, 30(4):424--435, 1989.

\bibitem{harris1956}
T.~E. Harris.
\newblock The existence of stationary measures for certain {M}arkov processes.
\newblock In {\em Proceedings of the Third Berkeley Symposium on Mathematical
  Statistics and Probability, Volume 2: Contributions to Probability Theory},
  pages 113--124, Berkeley, Calif., 1956. University of California Press.

\bibitem{Harris_book}
T.~E. Harris.
\newblock {\em The theory of branching processes}.
\newblock Die Grundlehren der Mathematischen Wissenschaften, Bd. 119.
  Springer-Verlag, Berlin; Prentice-Hall, Inc., Englewood Cliffs, N.J., 1963.

\bibitem{HMS03}
A.~G. Hart, S.~Mart{\'\i}nez, and J.~San~Mart{\'\i}n.
\newblock The $\lambda$-classification of continuous-time birth-and-death
  processes.
\newblock {\em Adv. Appl. Probab.}, 35(4):1111--1130, 2003.

\bibitem{HuisingaMeynSchutte}
W.~Huisinga, S.~Meyn, and C.~Sch\"{u}tte.
\newblock Phase transitions and metastability in {M}arkovian and molecular
  systems.
\newblock {\em Ann. Appl. Probab.}, 14(1):419--458, 2004.

\bibitem{KimSong}
P.~Kim and R.~Song.
\newblock Intrinsic ultracontractivity of non-symmetric diffusion semigroups in
  bounded domains.
\newblock {\em Tohoku Math. J. (2)}, 60(4):527--547, 2008.

\bibitem{KM03}
I.~Kontoyiannis and S.~P. Meyn.
\newblock Spectral theory and limit theorems for geometrically ergodic {M}arkov
  processes.
\newblock {\em Ann. Appl. Probab.}, 13(1):304--362, 2003.

\bibitem{KM05}
I.~Kontoyiannis and S.~P. Meyn.
\newblock Large deviations asymptotics and the spectral theory of
  multiplicatively regular {M}arkov processes.
\newblock {\em Electron. J. Probab.}, 10:no. 3, 61--123, 2005.

\bibitem{KM12}
I.~Kontoyiannis and S.~P. Meyn.
\newblock Geometric ergodicity and the spectral gap of non-reversible {M}arkov
  chains.
\newblock {\em Probab. Theory Related Fields}, 154(1-2):327--339, 2012.

\bibitem{Krein-Rutman}
M.~G. Kre\u{\i}n and M.~A. Rutman.
\newblock Linear operators leaving invariant a cone in a {B}anach space.
\newblock {\em Amer. Math. Soc. Translation}, 1950(26):128, 1950.

\bibitem{KLPP}
T.~Kurtz, R.~Lyons, R.~Pemantle, and Y.~Peres.
\newblock {\em A conceptual proof of the Kesten-Stigum theorem for multi-type
  branching processes}.
\newblock Classical and Modern Branching Processes. Springer, New York, 1997.

\bibitem{LP09}
P.~Lauren\c{c}ot and B.~Perthame.
\newblock Exponential decay for the growth-fragmentation/cell-division
  equation.
\newblock {\em Comm. Math. Sci.}, 7(2):503--510, 2009.

\bibitem{Maillard2016}
P.~Maillard.
\newblock Speed and fluctuations of n-particle branching brownian motion with
  spatial selection.
\newblock {\em Probab. Theory Related Fields}, 166(3):1061--1173, 2016.

\bibitem{marguet17}
A.~Marguet.
\newblock A law of large numbers for branching {M}arkov processes by the
  ergodicity of ancestral lineages.
\newblock {\em ESAIM Probab. Stat.}, 23:638--661, 2019.

\bibitem{M16}
A.~Marguet.
\newblock Uniform sampling in a structured branching population.
\newblock {\em Bernoulli}, 25(4A):2649--2695, 2019.

\bibitem{Ma15}
N.~Mari{\'c}.
\newblock Fleming--viot particle system driven by a random walk on
  $\mathbb{N}$.
\newblock {\em J. Stat. Phys.}, 160(3):548--560, 2015.

\bibitem{MD86}
J.~A.~J. Metz and O.~Diekmann.
\newblock {\em The Dynamics of Physiologically Structured Populations},
  volume~68 of {\em Lecture notes in Biomathematics}.
\newblock Springer, 1st edition, August 1986.

\bibitem{MT}
S.~Meyn and R.~L. Tweedie.
\newblock {\em Markov chains and stochastic stability}.
\newblock Cambridge University Press, Cambridge, second edition, 2009.

\bibitem{MTIII}
S.~P. Meyn and R.~L. Tweedie.
\newblock Stability of {M}arkovian processes. {III}. {F}oster-{L}yapunov
  criteria for continuous-time processes.
\newblock {\em Adv. in Appl. Probab.}, 25(3):518--548, 1993.

\bibitem{MMP2}
P.~Michel, S.~Mischler, and B.~Perthame.
\newblock General relative entropy inequality: an illustration on growth
  models.
\newblock {\em J. Math. Pures Appl.}, 84(9):1235--1260, 2005.

\bibitem{MS16}
S.~Mischler and J.~Scher.
\newblock Spectral analysis of semigroups and growth-fragmentation equations.
\newblock {\em Ann. Inst. H. Poincar\'e Anal. Non Lin\'eaire}, 33(3):849--898,
  2016.

\bibitem{Nummelin}
E.~Nummelin.
\newblock {\em General irreducible Markov chains and nonnegative operators}.
\newblock Cambridge University Press, Cambridge, 1984.

\bibitem{Nussbaum}
R.~D. Nussbaum.
\newblock Hilbert's projective metric and iterated nonlinear maps.
\newblock {\em Mem. Amer. Math. Soc.}, 75(391):iv+137, 1988.

\bibitem{PPS}
K.~Pakdaman, B.~Perthame, and D.~Salort.
\newblock Adaptation and fatigue model for neuron networks and large time
  asymptotics in a nonlinear fragmentation equation.
\newblock {\em J. Math. Neurosci.}, 4(1):14, 2014.

\bibitem{Perron}
O.~Perron.
\newblock Zur {T}heorie der {M}atrices.
\newblock {\em Math. Ann.}, 64(2):248--263, 1907.

\bibitem{BP}
B.~Perthame.
\newblock {\em Transport equations in biology}.
\newblock Frontiers in Mathematics. Birkh\"auser Verlag, Basel, 2007.

\bibitem{PR05}
B.~Perthame and L.~Ryzhik.
\newblock Exponential decay for the fragmentation or cell-division equation.
\newblock {\em J. Diff. Equations}, 210(1):155--177, 2005.

\bibitem{RW00}
L.~C.~G. Rogers and D.~Williams.
\newblock {\em Diffusions, Markov processes and martingales: Volume 2, It{\^o}
  calculus}, volume~2.
\newblock Cambridge university press, 2000.

\bibitem{RTK}
R.~Rudnicki and M.~Tyran-Kami\'nska.
\newblock {\em Piecewise deterministic processes in biological models}.
\newblock SpringerBriefs in Applied Sciences and Technology. Springer, Cham,
  2017.

\bibitem{Schaefer}
H.~H. Schaefer.
\newblock {\em Banach lattices and positive operators}.
\newblock Springer-Verlag, New York-Heidelberg, 1974.
\newblock Die Grundlehren der mathematischen Wissenschaften, Band 215.

\bibitem{Seneta}
E.~Seneta.
\newblock {\em Non-negative matrices and {M}arkov chains}.
\newblock Springer Series in Statistics. Springer, New York, 2006.
\newblock Revised reprint of the second (1981) edition.

\bibitem{vD91}
E.~A. van Doorn.
\newblock Quasi-stationary distributions and convergence to quasi-stationarity
  of birth-death processes.
\newblock {\em Adv. Appl. Probab.}, 23(4):683--700, 1991.

\bibitem{vD03}
E.~A. van Doorn.
\newblock Birth--death processes and associated polynomials.
\newblock {\em Journal of Computational and Applied Mathematics}, 153(1):497 --
  506, 2003.

\bibitem{V18}
A.~{Velleret}.
\newblock {Unique Quasi-Stationary Distribution, with a possibly stabilizing
  extinction}.
\newblock Preprint arXiv:1802.02409, 2018.

\bibitem{V19}
A.~{Velleret}.
\newblock {Exponential quasi-ergodicity for processes with discontinuous
  trajectories}.
\newblock Preprint arXiv:1902.01441, 2019.

\bibitem{V15}
D.~Villemonais.
\newblock Minimal quasi-stationary distribution approximation for a birth and
  death process.
\newblock {\em Electron. J. Probab.}, 20:18 pp., 2015.

\bibitem{Webb87}
G.~F. Webb.
\newblock An operator-theoretic formulation of asynchronous exponential growth.
\newblock {\em Trans. Amer. Math. Soc.}, 303(2):751--763, 1987.

\bibitem{Wu}
L.~Wu.
\newblock Essential spectral radius for {M}arkov semigroups. {I}. {D}iscrete
  time case.
\newblock {\em Probab. Theory Related Fields}, 128(2):255--321, 2004.

\bibitem{ZvBW2}
A.~A. Zaidi, B.~Van~Brunt, and G.~C. Wake.
\newblock Solutions to an advanced functional partial differential equation of
  the pantograph type.
\newblock {\em Proc. A.}, 471(2179):20140947, 15, 2015.

\bibitem{ZZ13}
H.~Zhang and Y.~Zhu.
\newblock Domain of attraction of the quasistationary distribution for
  birth-and-death processes.
\newblock {\em J. Appl. Probab.}, 50(1):114--126, 03 2013.

\end{thebibliography}

\end{document}